\documentclass[a4paper,11pt]{article}

\usepackage{amsfonts}
\usepackage{epsf,latexsym,amsfonts,amsbsy,mathrsfs}
\usepackage{cite}
\usepackage{amsmath, amssymb, amsthm,amscd,amsxtra}
\usepackage[]{graphicx,subfigure}
\usepackage{float}
\usepackage{multirow}
\usepackage{bbm}
\usepackage{stmaryrd}
\usepackage{epsfig}
\usepackage{color}
\usepackage{ifpdf}
\ifpdf 
  \usepackage[pdftex]{graphicx}
  \DeclareGraphicsExtensions{.pdf,.png,.jpg,.jpeg,.mps}
  \usepackage{pgf}
  \usepackage{tikz}
\else 
  \usepackage{graphicx}
  \DeclareGraphicsExtensions{.eps,.bmp}
  \usepackage{pgf}
  \usepackage{tikz}
  \usepackage{pstricks}
\fi
\usepackage{epic,bez123}
\usepackage{floatflt}
\usepackage{wrapfig}

\newtheorem{theorem}{Theorem}[section]

\theoremstyle{definition}

\theoremstyle{remark}
\newtheorem{remark}[theorem]{Remark}

\numberwithin{equation}{section}

\newcommand{\nn}{\nonumber}

\newtheorem{thm}{Theorem}[section]

\newtheorem{lem}{Lemma}[section]

\newcommand{\Bn}{{\boldsymbol{n}}}

\newcommand{\Bq}{{\boldsymbol{q}}}
\newcommand{\Br}{{\boldsymbol{r}}}

\newcommand{\Bv}{{\boldsymbol{v}}}
\newcommand{\Bw}{{\boldsymbol{w}}}

\newcommand{\BB}{{\boldsymbol{B}}}

\newcommand{\BL}{{\boldsymbol{L}}}

\newcommand{\BQ}{{\boldsymbol{Q}}}

\newcommand{\BV}{{\boldsymbol{V}}}

\newcommand{\Ct}{{\mathcal T}}

\date{}
\begin{document}

\title{A Hybridizable Discontinuous Galerkin Method for the Helmholtz Equation with High Wave Number}

\author{Huangxin Chen\thanks{School of Mathematical Sciences, Xiamen University, Xiamen, 361005, People's Republic of China. {\tt chx@xmu.edu.cn}.}~~, Peipei Lu$^\dag$, and Xuejun Xu\thanks{LSEC, Institute of
Computational Mathematics and Scientific/Engineering Computing,
Academy of Mathematics and System Sciences, Chinese Academy of
Sciences, P.O.Box 2719, Beijing, 100190, People's Republic of China. {\tt lupeipei@lsec.cc.ac.cn, xxj@lsec.cc.ac.cn}.}}

\maketitle
\begin{abstract}
This paper analyzes the error estimates of the hybridizable
discontinuous Galerkin (HDG) method for the Helmholtz equation with
high wave number in two and three dimensions. The approximation
piecewise polynomial spaces we deal with are of order $p\geq 1$.
Through choosing a specific parameter and using the duality
argument, it is proved that the HDG method is stable without any
mesh constraint for any wave number $\kappa$. By exploiting the
stability estimates, the dependence of convergence of the HDG method
on $\kappa,h$ and $p$ is obtained. Numerical experiments are given
to verify the theoretical results.

\end{abstract}

{\bf Key words}. Hybridizable discontinuous Galerkin method,
Helmholtz equation, high wave number, error estimates

\section{Introduction}
The numerical solutions of Helmholtz problems have been an area of
active research for almost half of a century. Because of the well
known pollution effect, the standard Galerkin finite element methods
can maintain a desired level of accuracy only if the mesh resolution
is also appropriately increased. In order to remedy this problem and
to obtain more stable and accurate approximation, numerous
nonstandard methods have been proposed recently (cf.
\cite{Ihlenburg}). One type of methods applies the stabilized
discrete variational form to approximate the Helmholtz equation,
which includes Galerkin-least-squares finite element methods
\cite{Chang,Harari-Hughes,MW}, quasi-stabilized finite element
methods \cite{Babu}, absolutely stable discontinuous Galerkin (DG)
methods \cite{Wu,Wuhp,feng} and continuous interior penalty finite
element methods (CIP-FEM) \cite{Wu2011}. Other approaches include
the partition of unity finite element methods
\cite{Melenk1,Melenk2,LB}, the ultra weak variational formulation
\cite{CD}, plane wave DG methods \cite{Amara,Hiptmair}, spectral
methods \cite{Shen}, generalized Galerkin/finite element methods
\cite{Babu-0,Melenk}, meshless methods \cite{Babu-1}, and the
geometrical optics approach \cite{Engquist}.

Discontinuous Galerkin methods have several attractive features
compared with conforming finite element methods. For example, the
polynomial degrees can be different from element to element, and
they work well on arbitrary meshes. For the Helmholtz equation, the
interior penalty discontinuous Galerkin methods (cf. \cite{Wu,Wuhp})
and the local discontinuous Galerkin methods \cite{feng} perform
much better than the standard finite element methods, and they are
well posed  without any mesh constraint. Despite all these
advantages, the dimension of the approximation DG space is much
larger than the dimension of the corresponding classical conforming
space.

The hybridizable discontinuous Galerkin methods were recently
introduced to try to address this issue. The HDG methods retain the
advantages of the standard DG methods and result in a significant
reduced degrees of freedom. New variables on the boundary of
elements are introduced such that the solution inside each element
can be computed in terms of them. In particular, element by element,
volume degrees of freedom can be parameterized by the surface
degrees, and the resulting algebraic system is only due to the
unknowns on the skeleton of the mesh. For a comprehensive
understanding of the HDG methods, we can refer to \cite{unified} for
a unified framework for second order elliptic problems and to
\cite{KSC2012} for the implementations.

In \cite{GM2011}, the authors give error estimates of the HDG method
for the interior Dirichlet problem for the Helmholtz equation, but
it is under the condition that $C_{\kappa}\kappa hM_{\tau}^{{\rm
min}}$ and $C_{\kappa}\kappa hM_{\tau}^{{\rm max}}$ are sufficiently
small, where $C_{\kappa}$ is a constant which is dependent on
$\kappa$ but is not characterized explicitly, and $M_{\tau}^{{\rm
min}}, M_{\tau}^{{\rm max}}$ depend on the parameters defined in the
numerical fluxes. Motivated by this work, the primary objective of this
paper is to analyze the explicit dependence of convergence of HDG method
for the Helmholtz equation on $\kappa, h$ and $p$. In this paper, we consider the Helmholtz equation
with Robin boundary condition which is the first order approximation
of the radiation condition:
\begin{align}
\label{Hem}
-\triangle u-\kappa^2u&=\tilde f  \qquad {\rm in }\ \Omega,\\
\label{Robin} \frac{\partial u}{\partial {\Bn}}+{\bf i}\kappa
u&=\tilde g \qquad {\rm on }\ \partial \Omega,
\end{align}
where $\Omega \subset \mathbb R^d$, $d=2,3$ is a
polygonal/polyhedral domain, $\kappa>0$ is known as the wave number, ${\bf i}=\sqrt{-1} $ denotes the
imaginary unit, and ${\Bn}$ denotes the unit outward normal to $\partial \Omega$.

The main difficulty of analyzing the Helmholtz equation lies in the
strong indefiniteness of the problem which makes it hard to
establish the stability for the numerical approximation. For the HDG
method, we use a duality argument to obtain the stability estimates
of the numerical solution. In the analysis, a crucial step lies in
the derivation of the dependence of convergence on $p$. We utilize
the explicit error estimates of $L^2$ projection operator (see Lemma
\ref{lemma34}) to overcome this problem. Then we obtain that the HDG
method for the Helmholtz problem (\ref{Hem})-(\ref{Robin}) attains a
unique solution for any $\kappa >0$, $h>0$. Furthermore, the
stability results not only guarantee the well-posedness of the HDG
method but also play a key role in the derivation of the error
estimates.

The duality argument can not be directly applied to establish the
error estimates. Thus, we first construct an auxiliary problem and
show its HDG error estimates by the duality technique. Then,
combining the stability estimates, the error estimates of HDG scheme
for the original Helmholtz problem (\ref{Hem})-(\ref{Robin}) are
deduced. Let $u_h$ and ${\Bq}_h$ be the HDG approximation to $u$ and
${\Bq}:={\bf i}\nabla u/\kappa$ respectively. We obtain the
following results:

(i) The following stability and error estimates hold without any constraint:
$$\|u_h\|_{0,\Omega}+\|{\Bq}_h\|_{0,\Omega}\lesssim \Big(1+\frac{\kappa^3h^2}{p^2}\Big)\|f\|_{0,\Omega}+\Big(1+\frac{\kappa^{\frac{3}{2}}h}{p}\Big)\|g\|_{0,\partial\Omega},
$$
$$
\|u-u_h\|_{0,\Omega}\lesssim \Big(\frac{\kappa h^2}{p^2}+\frac{\kappa^2h^2}{p^2}+\frac{\kappa^5h^4}{p^4}\Big)M(\tilde f,\tilde g),
$$
$$
\kappa\|{\Bq}-{\Bq}_h\|_{0,\Omega}\lesssim \Big(\frac{\kappa h}{p}+\frac{\kappa^3h^2}{p^2}+\frac{\kappa^6h^4}{p^4}\Big)M(\tilde f,\tilde g),
$$
where $f:=-{\bf i} \tilde f/\kappa$, $g:=-{\bf i} \tilde g/\kappa$
and $M(\tilde f,\tilde g):=\|\tilde f\|_{0,\Omega}+\|\tilde
g\|_{0,\partial\Omega}$. We use notations $A\lesssim B$ and
$A\gtrsim B$ for the inequalities $A\leq CB$ and $A\geq CB$, where
$C$ is a positive number independent of the mesh size, polynomial
degree and wave number $\kappa$, but the value of which can take on
different values in different occurrences.

(ii) Suppose $\frac{\kappa^3h^2}{p^2}\lesssim 1$, there hold the following improved results:
$$\|u_h\|_{0,\Omega}+\|{\Bq}_h\|_{0,\Omega}\lesssim \|f\|_{0,\Omega}+\|g\|_{0,\partial\Omega},
$$
$$
\|u-u_h\|_{0,\Omega}\lesssim \Big(\frac{\kappa h^2}{p^2}+\frac{\kappa^2h^2}{p^2}\Big)M(\tilde f,\tilde g),
$$
$$
\kappa\|{\Bq}-{\Bq}_h\|_{0,\Omega}\lesssim \Big(\frac{\kappa h}{p}+\frac{\kappa^3h^2}{p^2}\Big)M(\tilde f,\tilde g).
$$
Comparing to the estimates in $hp$-IPDG method for Helmholtz
problem, we find that the condition for the above improved results
weakens the mesh condition $\frac{\kappa^3h^2}{p}\lesssim 1$ which
is requested in \cite{Wuhp}. For the estimates under the mesh
condition $\frac{\kappa^3h^2}{p^2}\gtrsim 1$, the results in (i) can
not be directly applied, but we may still get the following improved
estimates.

(iii) Suppose $\frac{\kappa^3h^2}{p^2}\gtrsim 1$, there hold
$$\|u-u_h\|_{0,\Omega}\lesssim \frac{\kappa^2h^2}{p^2}M(\tilde f,\tilde g),$$
$$\kappa \|\Bq-{\Bq}_h\|_{0,\Omega}\lesssim \big(\frac{\kappa h}{p}+\frac{\kappa^3 h^2}{p^2}\big)M(\tilde f,\tilde g).$$

We remark that in this work the {\it local stabilization parameter}
to determine the numerical flux in the HDG scheme is always selected
as $\tau = \frac{p}{\kappa h}$ (see (\ref{numerical-flux})). Our
numerical results show that the predicted convergence rates are
observed.

The organization of the paper is as follows: We precisely define the
HDG method for the Helmholtz equation and give some notations in the next
section. Section 3 is dedicated to the characterization of the
surface degrees $\hat u_h $. In section 4, we derive the stability
estimates of the HDG method. The error estimates of the auxiliary
problem are carried out in section 5 while section 6 states the main
results of this paper, i.e., the error estimates of the HDG method
for the Helmholtz equation. In the final section, we give some
numerical results to confirm our theoretical analysis.

\section{The hybridizable discontinuous Galerkin method}
 The HDG scheme is based on a first order
formulation of the above Helmholtz equation
(\ref{Hem})-(\ref{Robin}) which can be rewritten in  mixed form as finding
$({\Bq},u)$ such that
\begin{align}
\label{HS}
{\bf i}\kappa{\Bq}+\nabla u&=0 \qquad {\rm in }\ \Omega,\\
{\bf i}\kappa u+{\rm div}\, {\Bq}&=f  \qquad {\rm in }\ \Omega,\\
\label{HE} -{\Bq}\cdot {\Bn}+u&=g \qquad {\rm on }\
\partial\Omega.
\end{align}
Existence and uniqueness of solutions to (\ref{HS})-(\ref{HE}) is
well known and it is proved in \cite{Feng} that they satisfy the
following  regularity result:
\begin{align}
\label{regularity}
\kappa^{-1}\|u\|_{2,\Omega}+\|u\|_{1,\Omega}+\|{\Bq}\|_{1,\Omega}\lesssim
M(\tilde f,\tilde g).
\end{align}

We consider a subdivision of $\Omega $ into a finite element mesh of
shape regular triangle $T$ in $\mathbb R^2$ (or tetrahedron $T$ in
$\mathbb R^3$) and denote the collection of triangles (tetrahedra)
by $\mathcal T_h$, the collection of edges (faces) by $\mathcal
E_h$, while the collection of interior edges (faces) by $\mathcal
E_h^0$ and the collection of element boundaries by $\partial
\mathcal T_h:=\{\partial T| T\in \mathcal T_h\}$. Throughout this paper we use the standard notations and definitions
for Sobolev spaces (see, e.g.,Adams\cite{Adams}).

On each element $T$ and each edge (face) $F$, we define the local
spaces of polynomials of degree $p\geq 1$:
$${\BV}(T):=(\mathcal P_p(T))^d,\qquad W(T):=\mathcal P_p(T),\qquad M(F):=\mathcal P_p(F),$$
where $\mathcal P_p(S)$ denotes the space of polynomials of total
degree at most $p$ on $S$. The corresponding global finite element
spaces are given by
\begin{align*}
{\BV}_h^p:&=\{{\Bv}\in {\BL}^2(\Omega)\  |\ {\Bv}|_T\in {\BV}(T)\ {\rm for\  all }\ T\in \mathcal T_h\},\\
W_h^p:&=\{w\in L^2(\Omega)\ |\ w|_T\in W(T)\ {\rm for\ all }\ T\in \mathcal T_h\},\\
M_h^p:&=\{ \mu\in L^2(\mathcal E_h)\ |\  \mu|_F\in M(F)\ {\rm for\ all }\ F\in \mathcal E_h\},
\end{align*}
where ${\BL}^2(\Omega):=( L^2(\Omega))^d$ and $L^2(\mathcal
E_h):=\Pi_{F\in \mathcal E_h}L^2(F)$. On these spaces we define the
bilinear forms
\begin{align*}
({\Bv},{\Bw})_{\mathcal T_h}:=\sum_{T\in \mathcal
T_h}({\Bv},{\Bw})_{T}, \ (v,w)_{\mathcal T_h}:=\sum_{T\in \mathcal
T_h}(v,w)_T, \ {\rm and} \ \langle v,w \rangle_{\partial \mathcal
T_h}:=\sum_{T\in \mathcal T_h}\langle v,w\rangle_{\partial T},
\end{align*}
with $({\Bv},{\Bw})_{T}:=\int_T {\Bv}\cdot {\Bw}dx$,
$(v,w)_T:=\int_T { v} { w}dx$ and $\langle v,w\rangle_{\partial
T}:=\int_{\partial T} { v} { w}ds$.

The hybridizable discontinuous Galerkin method yields finite element
approximations $({\Bq}_h, u_h,\hat u_h)\in {\BV}_h^p\times
W_h^p\times M_h^p$ which satisfy
\begin{align}
\label{P1}
({\bf i}\kappa {\Bq}_h,\overline{\Br})_{\mathcal T_h}-(u_h, \overline{{\rm div \, {\Br}}})_{\mathcal T_h}+\langle \hat u_h,\overline {{\Br}\cdot {\Bn}}\rangle_{\partial \mathcal T_h}&=0,\\
\label{P2}
({\bf i}\kappa u_h,\overline w)_{\mathcal T_h}-({\Bq}_h,\overline{\nabla w})_{\mathcal T_h}+\langle \hat {\Bq}_h\cdot {\Bn} ,\overline w\rangle_{\partial \mathcal T_h}&=(f,\overline w)_{\mathcal T_h} ,\\
\label{P3}
\langle -\hat{\Bq}_h\cdot {\Bn}+ \hat u_h,\overline \mu\rangle_{\partial \Omega}&=\langle g, \overline \mu\rangle_{\partial \Omega},\\
\label{P4} \langle \hat{\Bq}_h\cdot {\Bn},\overline
\mu\rangle_{\partial \mathcal T_h \backslash
 \partial \Omega}&=0,
\end{align}
for all ${\Br}\in {\BV}_h^p$, $w\in W_h^p$, and $\mu\in M_h^p$,
where the overbar denotes complex conjugation. The numerical flux
$\hat {\Bq}_h$ is given by
\begin{align}\label{numerical-flux}
\hat {\Bq}_h={\Bq}_h+\tau (u_h-\hat u_h){\Bn} \qquad {\rm on } \
\partial \mathcal T_h,
\end{align}
where the parameter $\tau$ is the so-called {\it local stabilization parameter} which has an important effect on both the stability of the solution and the accuracy of the HDG scheme.
We always choose $\tau = \frac{p}{\kappa h}$ in this paper. The error analysis is based on projection operators which are defined as follows
$${\bf \Pi_h}: {\BL}^2(\Omega)\to {\BV}_h^p \qquad \Pi_h: L^2(\Omega)\to  W_h^p$$
 for any $T\in \mathcal T_h$, they satisfy
\begin{align}
({\bf \Pi}_h{\Bq},{\Bv})_T&=({\Bq},{\Bv})_T \qquad {\rm for \ all\ }{\Bv}\in {\BV}(T),\\
\label{l2}
(\Pi_hu,w)_T&=(u,w)_T  \qquad {\rm for \ all\ }w\in W(T).
\end{align}
We conclude the introduction by setting some notations used
throughout this paper. Let the broken space $H^1(\Omega_h)$ be
defined by
$$H^1(\Omega_h):=\{v: v|_T\in H^1(T),\ \forall T\in \mathcal T_h\},$$
the seminorm of which is
$$|v|_{1,\Omega_h}^2:=\sum_{T\in \mathcal T_h}|v|_{1,T}^2.$$
The trace of functions in $H^1(\Omega_h)$ belong to
$T(\Gamma):=\Pi_{T\in \mathcal T_h}L^2(\partial T)$. For any
$\phi\in T(\Gamma)$, and ${\Bv}\in (T(\Gamma))^d$, if $e\in \mathcal
E_h^0,e=\partial T^+\cap \partial T^-$, we set
\begin{align*}
\{\hspace{-0.1cm}\{\phi\}\hspace{-0.1cm}\}:=\frac{1}{2}(\phi^++\phi^-),\quad
\llbracket\phi\rrbracket:=\phi^+\cdot {\Bn}^++\phi^-\cdot {\Bn}^-
\end{align*}
and
\begin{align*}
\{\hspace{-0.1cm}\{{\Bv}\}\hspace{-0.1cm}\}:=\frac{1}{2}({\Bv}^++{\Bv}^-),\quad
\llbracket{\Bv}\rrbracket:={\Bv}^+\cdot {\Bn}^++{\Bv}^-\cdot
{\Bn}^-.
\end{align*}
For $e\in\partial \mathcal T_h\cap \partial \Omega$, we define
$$\{\hspace{-0.1cm}\{\phi\}\hspace{-0.1cm}\}:=\phi,\ \llbracket\phi\rrbracket:=\phi\cdot{\Bn},\ \{\hspace{-0.1cm}\{{\Bv}\}\hspace{-0.1cm}\}:={\Bv},\ \llbracket{\Bv}\rrbracket:={\Bv}\cdot{\Bn}.$$

\section{The characterization of $\hat u_h $}

One of the advantages of hybridizable discontinuous Galerkin methods
is the elimination of both ${\Bq}_h$ and $u_h$ from the equation and
obtain a formulation in terms of $\hat u_h$ only. In this section,
we show that $\hat u_h$ can be characterized by a simple weak
formulation in which none of the other variables appear.

First we define the discrete solutions of the local problems, for
each function $\lambda\in M_h^p$, $({\bf \mathcal Q}_\lambda, \mathcal {U}_\lambda)\in
{\BV}(T)\times W(T)$ satisfies the following formulation
\begin{align}
\label{local1} ({\bf i}\kappa\mathcal
{Q}_\lambda,\overline{{\Br}})_T-(\mathcal {U}_\lambda,\overline{{\rm
div}\,{\Br}})_T&=-\langle \lambda,
\overline{{\Br}\cdot {\Bn}}\rangle_{\partial T} \quad {\rm for \ all \ } {\Br}\in {\BV}(T),\\
\label{local2} ({\bf i}\kappa \mathcal
{U}_\lambda,\overline{w})_T-(\mathcal {Q}_\lambda,\overline{\nabla
w})_T+\langle\hat {\mathcal {Q}}_\lambda\cdot
{\Bn},\overline{w}\rangle_{\partial T}&=0\quad {\rm for \ all \ }
{\Bw}\in W(T),
\end{align}
where $\hat {\mathcal {Q}}_\lambda\cdot {\Bn}=\mathcal {Q}_\lambda\cdot
{\Bn}+\tau(\mathcal {U}_\lambda-\lambda)$.  For $f\in L^2(\Omega)$, $({\bf
\mathcal Q}_f, \mathcal {U}_f)\in {\BV}(T)\times W(T)$ is defined as
follows
\begin{align}
\label{local3}
({\bf i}\kappa\mathcal {Q}_f,\overline{{\Br}})_T-(\mathcal {U}_f,\overline{{\rm div}\,{\Br}})_T&=0 \quad {\rm for \ all \ } {\Br}\in {\BV}(T),\\
\label{local4} ({\bf i}\kappa \mathcal
{U}_f,\overline{w})_T-(\mathcal {Q}_f,\overline{\nabla
w})_T+\langle\hat {\mathcal {Q}}_f\cdot
{\Bn},\overline{w}\rangle_{\partial T}&=(f,\overline{w})_T\quad {\rm
for \ all \ } {\Bw}\in W(T),
\end{align}
where $\hat {\mathcal {Q}}_f\cdot {\Bn}=\mathcal {Q}_f\cdot
{\Bn}+\tau\mathcal {U}_f$. Next we show that the local problem
(\ref{local1})-(\ref{local2}) is well posed. The uniqueness of
(\ref{local3})-(\ref{local4}) can be deduced similarly.
\begin{lem}
There exist a unique solution  $({\bf \mathcal Q}_\lambda, \mathcal
{U}_\lambda)\in {\BV}(T)\times W(T)$  to the local problem
(\ref{local1})-(\ref{local2}).
\end{lem}
\begin{proof}
Since it is a square system, to prove the existence and uniqueness of its solution, it is enough to show that if $\lambda=0$, we have that ${\bf \mathcal Q}_\lambda=0$, $\mathcal {U}_\lambda=0$.
Taking $r={\bf \mathcal Q}_\lambda$ in (\ref{local1}), $w=\mathcal {U}_\lambda$ in (\ref{local2}), we get
$${\bf i}\kappa(\mathcal {U}_\lambda,\overline{\mathcal {U}_\lambda})_T-{\bf i}\kappa({\bf \mathcal Q}_\lambda,\overline{{\bf \mathcal Q}_\lambda})_T+\tau\langle \mathcal {U}_\lambda,\overline{\mathcal {U}_\lambda}\rangle_{\partial T}=0,$$
which means
$\mathcal {U}_\lambda|_{\partial T}=0$. Back to (\ref{local1}), we derive that
$$-{\bf i}\kappa{\bf \mathcal Q}_\lambda=\nabla \mathcal {U}_\lambda.$$
 Inserting the above expression into (\ref{local2}), $$\triangle \mathcal {U}_\lambda=-\kappa^2\mathcal {U}_\lambda$$ is deduced, which implies that $\mathcal {U}_\lambda=0$, ${\bf \mathcal Q}_\lambda=0$.
\end{proof}

It is worth noting that the solution $({\Bq}_h, u_h)$ in
(\ref{P1})-(\ref{P4}) is exactly correspond to the following relationship
\begin{align*}
{\Bq}_h={\bf \mathcal Q}_{\hat u_h}+{\bf \mathcal Q}_f,\quad
u_h=\mathcal {U}_{\hat u_h}+\mathcal {U}_f.
\end{align*}
And $\hat u_h$ is the solution of the following formulation
$$a_h(\hat u_h,\mu)=b_h(\mu)\quad {\rm for \ all} \  \mu\in M_h^p,$$
where
\begin{align*}
a_h(\lambda,\mu):&=-\langle\llbracket\hat {\mathcal {Q}}_\lambda\rrbracket,\overline{\mu}\rangle_{\partial\mathcal T_h}+\langle \lambda,\overline{\mu}\rangle_{\partial \Omega},\\
b_h(\mu):&=\langle\llbracket\hat {\mathcal
{Q}}_f\rrbracket,\overline{\mu}\rangle_{\partial\mathcal
T_h}+\langle g,\overline{\mu}\rangle_{\partial \Omega}.
\end{align*}

\section{The stability of the hybridizable discontinuous Galerkin method}
The goal of this section is to derive stability estimates. We first cite the following lemma which provides
some approximation results that will play an important role later. A proof of the lemma can be found in \cite{Babuska, Schwab}.

\begin{lem}
Let $\hat T$ be a standard square or triangle. Then there exists an
operator  $\hat \pi_p:H^1(\hat T)\to \mathcal P_p(\hat T)$ such that
for any $\hat u\in H^1(\hat T)$
\begin{align}
\label{pip0-1} \|\hat u-\hat \pi_p\hat u\|_{0,\hat T}&\lesssim
p^{-1}|\hat u|_{1,\hat T}.
\end{align}
Moreover, if $\hat u\in H^2(\hat T) $,
\begin{align}
\label{pip0}
\|\hat u-\hat \pi_p\hat u\|_{0,\hat T}&\lesssim p^{-2}|\hat u|_{2,\hat T}, \\
\label{pip1}
|\hat u-\hat \pi_p\hat u|_{1,\hat T}&\lesssim p^{-1}|\hat u|_{2,\hat T}.
\end{align}
\end{lem}
Using the standard scaling technique, we can get the following approximation results.
\begin{lem}\label{imp-lemma-pre}
For any $ T\in \mathcal T_h$, there exists an operator $\pi_h^p
:H^1( T)\to \mathcal P_p(T)$ such that for any $u\in H^1(T)$ there
holds
\begin{align}
\label{hpip0-1} \| u- \pi_h^p u\|_{0,T}&\lesssim \frac{h}{p}| u|_{1,
T},
\end{align}
Moreover, if $ u\in H^2(T) $,
\begin{align}
\label{hpip0}
\| u- \pi_h^p u\|_{0,T}&\lesssim \Big(\frac{h}{p}\Big)^2| u|_{2, T},\\
\label{hpip1}
| u- \pi_h^p u|_{1,T}&\lesssim \frac{h}{p}| u|_{2,T}.
\end{align}
\end{lem}
We also need the following trace inequality and refer to \cite{Schwab} for the proof.
\begin{lem}\label{lemma33}
For any $ T\in \mathcal T_h$ and $v\in \mathcal P_p(T)$
\begin{align*}
\|v\|_{0,\partial T}\lesssim ph^{-\frac{1}{2}}\|v\|_{0,T}.
\end{align*}
\end{lem}

Now we derive the following approximation properties of the
projection operator $\Pi_h$ which is defined in (\ref{l2}). For the
sake of simplicity, the proof is restricted to 2-d case.
\begin{lem}\label{lemma34}
For any $T\in \mathcal T_h$, $u\in H^1(T)$, the projection operator
$\Pi_h$ satisfies
\begin{align}
\label{l2p-1}
\|u-\Pi_hu\|_{0,T}&\lesssim \frac{h}{p}|u|_{1,T},\\
\label{l21} \|u-\Pi_hu\|_{0,\partial T}&\lesssim
\Big(\frac{h}{p}\Big)^{\frac{1}{2}}|u|_{1,T}.
\end{align}
Moreover, if $u\in H^2(T)$
\begin{align}
\label{l2p}
\|u-\Pi_hu\|_{0,T}&\lesssim \Big(\frac{h}{p}\Big)^2|u|_{2,T},\\
\label{l22} \|u-\Pi_hu\|_{0,\partial T}&\lesssim
\Big(\frac{h}{p}\Big)^{\frac{3}{2}}|u|_{2,T}.
\end{align}
\end{lem}
\begin{proof}
An important property of $L^2(\Omega)$ projection operator $\Pi_h$ is that
\begin{align*}
\|u-\Pi_hu\|_{0,T}\leq {\rm inf}_{v\in W(T)}\|u-v\|_{0,T},
\end{align*}
hence (\ref{hpip0-1}) and (\ref{hpip0}) imply (\ref{l2p-1}) and
(\ref{l2p}). Let $\hat T_1$ and $\hat T_2$ be the standard triangles
with linear mappings $x=F_1(\hat x)$ and $x=F_2(\hat x)$
respectively, see Figure \ref{figref} for illustration. For any
$v\in W(T)$, define $\hat v(\hat x):=v\circ F_1(\hat x)$ and $\tilde
v(\hat x):=v\circ F_2(\hat x)$.

\begin{figure}[htbp]
\centering
    \includegraphics[width=3.0in]{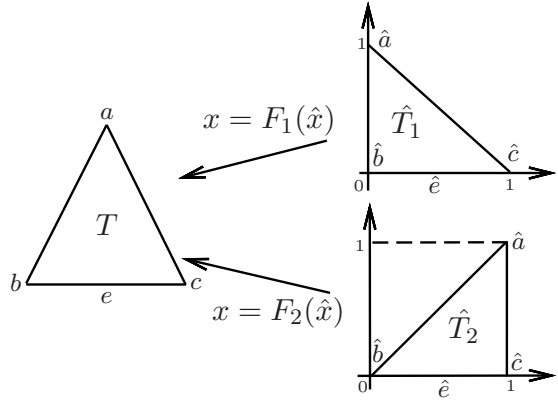}
    \caption{\small The triangle $T$ and its reference triangles $\hat{T}_1$ and $\hat{T}_2$.}\label{figref}
\end{figure}
For any $\hat v(\hat x)\in H^1(\hat T_1)$, we have
\begin{align*}
|\hat v(\hat {x})|^2&=\hat v(\hat {x})\cdot \overline{\hat v(\hat
{x})}= \hat v(\hat x_1,0)\cdot\overline{\hat v(\hat
x_1,0)}+\int_0^{\hat x_2}\frac{\partial (\hat v(\hat
x_1,\eta)\cdot\overline{\hat v(\hat x_1,\eta)})}{\partial \eta}\\&=
|\hat v(\hat x_1,0)|^2+2{\rm Re}\int_0^{\hat x_2}\hat v(\hat
x_1,\eta)\overline{\frac{\partial \hat v(\hat x_1,\eta)}{\partial
\eta}}d\eta.
\end{align*}
Hence
$$
|\hat v(\hat x_1,0)|^2=|\hat v(\hat x_1,\hat x_2)|^2-2{\rm Re}\int_0^{\hat x_2}\hat v(\hat x_1,\eta)\overline{\frac{\partial \hat v(\hat x_1,\eta)}{\partial \eta}}d\eta.
$$
By using integration on $\hat T_1$ we can deduce
\begin{align*}
&\int_0^1|\hat v(\hat x_1,0)|^2(1-\hat x_1)d\hat x_1=\int_0^1\int_0^{1-\hat x_1}|\hat v(\hat x_1,0)|^2d\hat x_2d\hat x_1\\
&=\int_{\hat T_1}|\hat v(\hat x_1,\hat x_2)|^2d\hat x_2d\hat x_1-2{\rm Re}\int_0^1\int_0^{1-\hat x_1}\int_0^{\hat x_2}\hat v(\hat x_1,\eta)\overline{\frac{\partial \hat v(\hat x_1,\eta)}{\partial \eta}}d\eta d\hat x_2d\hat x_1\\
&=\int_{\hat T_1}|\hat v(\hat x_1,\hat x_2)|^2d\hat x_2d\hat x_1-2{\rm Re}\int_0^1\int_0^{1-\hat x_1}\int_{\eta}^{1-\hat x_1}\hat v(\hat x_1,\eta)\overline{\frac{\partial \hat v(\hat x_1,\eta)}{\partial \eta}}d\hat x_2 d\eta d\hat x_1\\
&=\int_{\hat T_1}|\hat v(\hat x_1,\hat x_2)|^2d\hat x_2d\hat x_1-2{\rm Re}\int_0^1\int_0^{1-\hat x_1}(1-\hat x_1-\eta)\hat v(\hat x_1,\eta)\overline{\frac{\partial \hat v(\hat x_1,\eta)}{\partial \eta}} d\eta d\hat x_1\\
&=\int_{\hat T_1}|\hat v(\hat x_1,\hat x_2)|^2d\hat x_2d\hat x_1-2{\rm Re}\int_{\hat T_1}(1-\hat x_1-\hat x_2)\hat v(\hat x_1,\hat x_2)\overline{\frac{\partial \hat v(\hat x_1,\hat x_2)}{\partial \hat x_2}} d\hat x_2 d\hat x_1.
\end{align*}
Take $\hat v=\hat u-\hat \Pi\hat u$, where $\hat \Pi\hat u\in \mathcal P_p(\hat T_1)$ satisfies $(\hat \Pi\hat u, \hat w)_{\hat T_1}=(\hat u, \hat w)_{\hat T_1},\ \forall \hat w\in \mathcal P_p(\hat T_1)$, and note that
$$(u-\Pi_h u,v)_T=\frac{|T|}{|\hat T_1|}(\hat u-\widehat{\Pi_h u},\hat v)_{\hat T_1}=0,$$
which means $\hat \Pi\hat u=\widehat{\Pi_hu}$. By Lemma
\ref{imp-lemma-pre} and scaling technique we get
\begin{equation}\begin{split}
\label{e21}
&\int_{\hat e}\big((\hat u-\hat \Pi\hat u)(\hat x_1,0)\big)^2(1-\hat x_1)d\hat x_1\\
&=\int_{\hat T_1}(\hat u-\hat \Pi\hat u)^2d\hat x_2d\hat x_1-2{\rm Re}\int_{\hat T_1}(\hat u-\hat \Pi\hat u)\overline{\frac{\partial (\hat u-\hat \Pi\hat u)}{\partial \hat x_2}}(1-\hat x_1-\hat x_2) d\hat x_2 d\hat x_1\\
&=\int_{\hat T_1}(\hat u-\hat \Pi\hat u)^2d\hat x_2d\hat x_1-2{\rm Re}\int_{\hat T_1}(\hat u-\hat \Pi\hat u)\overline{\frac{\partial (\hat u-\hat \pi_p\hat u)}{\partial \hat x_2}}(1-\hat x_1-\hat x_2) d\hat x_2 d\hat x_1\\
&\leq \|\hat u-\hat \Pi\hat u\|_{0,\hat T_1}^2+2\|\hat u-\hat
\Pi\hat u\|_{0,\hat T_1} |\hat u-\hat \pi_p\hat u|_{1,\hat T_1}\\
&\lesssim
h^{-2}\|u-\Pi_hu\|_{0,T}^2+h^{-1}\|u-\Pi_hu\|_{0,T}|u-\pi_h^pu|_{1,T}
\lesssim \frac{h^2}{p^3}|u|_{2,T}^2,
\end{split}
\end{equation}
and
\begin{equation}\begin{split}
\label{e11}
&\int_{\hat e}\big((\hat u-\hat \Pi\hat u)(\hat x_1,0)\big)^2(1-\hat x_1)d\hat x_1\\&=\int_{\hat T_1}(\hat u-\hat \Pi\hat u)^2d\hat x_2d\hat x_1-2{\rm Re}\int_{\hat T_1}(\hat u-\hat \Pi\hat u)\overline{\frac{\partial \hat u}{\partial \hat x_2}}(1-\hat x_1-\hat x_2) d\hat x_2 d\hat x_1\\
&\lesssim \|\hat u-\hat \Pi\hat u\|_{0,\hat T_1}^2+\|\hat u-\hat
\Pi\hat u\|_{0,\hat T_1}
|\hat u|_{1,\hat T_1}\\
&\lesssim
h^{-2}\|u-\Pi_hu\|_{0,T}^2+h^{-1}\|u-\Pi_hu\|_{0,T}|u|_{1,T}
\lesssim p^{-1}|u|_{1,T}^2.
\end{split}
\end{equation}
Now we map $T$ to $\hat T_2$ and similarly we can derive
\begin{align}
\label{e22} \int_{\hat e}\big((\tilde u-\tilde \Pi\tilde u)(\hat
x_1,0)\big)^2\hat x_1d\hat x_1\lesssim \frac{h^2}{p^3}|u|_{2,T}^2,
\end{align}
and
\begin{align}
\label{e12} \int_{\hat e}\big((\tilde u-\tilde \Pi\tilde u)(\hat
x_1,0)\big)^2\hat x_1d\hat x_1\lesssim p^{-1}|u|_{1,T}^2,
\end{align}
where $\tilde \Pi\tilde u\in \mathcal P_p(\hat T_2)$ satisfies
$(\tilde \Pi\tilde  u, \tilde w)_{\hat T_2}=(\tilde u, \tilde
w)_{\hat T_2},\ \forall \tilde w\in \mathcal P_p(\hat T_2)$ and
$\tilde \Pi\tilde u=\widetilde{\Pi_hu}$. Since $\hat v(\hat
x)=\tilde v(\hat x)$ on $\hat e$, summing up (\ref{e21}),
(\ref{e22}), and (\ref{e11}), (\ref{e12}) respectively and noting
that $e$ is not particularly chosen, the lemma is proved.
\end{proof}

\begin{remark}
In this paper, we only deal with meshes consisting of triangles or
tetrahedra, but we should note that Lemma \ref{lemma34} can be
extended to the meshes constituted with rectangles or hexahedra. The
proof is similar with Lemma \ref{lemma34} and can also be found in
\cite{CIP-D}.
\end{remark}

\begin{lem}\label{lemma35}
Let $({\Bq}_h, u_h,\hat u_h)\in {\BV}_h^p\times W_h^p\times M_h^p$
be the solutions of (\ref{P1})-(\ref{P4}). There hold
\begin{align}
\label{trace-H}
\tau\|u_h-\hat u_h\|_{0,\partial \mathcal T_h}^2&\leq \|f\|_{0,\Omega}\|u_h\|_{0,\Omega}+\|g\|_{0,\partial\Omega}^2,\\
\label{q_h-H} \kappa \|{\Bq}_h\|^2_{0,\Omega}&\lesssim
\kappa\|u_h\|^2_{0,\Omega}+\|f\|^2_{0,\Omega}+\|g\|^2_{0,\partial\Omega}.
\end{align}
\end{lem}
 \begin{proof}
 We choose ${\Br}={\Bq}_h$, $w=u_h$, $\mu=\hat u_h$ in (\ref{P1})-(\ref{P4}) and get
 \begin{align}
 \label{r}
 ({\bf i}\kappa {\Bq}_h,\overline{\Bq}_h)_{\mathcal T_h}-(u_h, \overline{{\rm div \, {\Bq}_h}})_{\mathcal T_h}+\langle \hat u_h,\overline {{\Bq}_h\cdot {\Bn}}\rangle_{\partial \mathcal T_h}&=0,\\
 \label{w}
 ({\bf i}\kappa u_h+{\rm div \, {\Bq}_h},\overline u_h)_{\mathcal T_h}+\langle \tau(u_h-\hat u_h) ,\overline u_h\rangle_{\partial \mathcal T_h}&=(f,\overline u_h)_{\mathcal T_h} ,\\
 \label{mu}
 \langle \hat{\Bq}_h\cdot {\Bn} ,\overline {\hat u_h}\rangle_{\partial \mathcal T_h}&=\langle \hat u_h,\overline {\hat u_h }\rangle_{\partial \Omega} -\langle g, \overline {\hat u_h}\rangle_{\partial \Omega}.
\end{align}
Using (\ref{mu}), the complex conjugation of (\ref{r}) can be rewritten as
\begin{align}
\label{cr} - ({\bf i}\kappa {\Bq}_h,\overline{\Bq}_h)_{\mathcal
T_h}-({\rm div \, {\Bq}_h},\overline{u_h})_{\mathcal T_h}-\langle
\tau(u_h-\hat u_h),\overline{\hat u_h}\rangle_{\partial \mathcal
T_h}+\langle \hat u_h,\overline {\hat u_h }\rangle_{\partial \Omega}
=\langle g, \overline {\hat u_h}\rangle_{\partial \Omega}.
\end{align}
Adding (\ref{cr}) and (\ref{w}) together, the following identity is deduced
\begin{align*}
({\bf i}\kappa u_h,\overline u_h)_{\mathcal T_h}- ({\bf i}\kappa
{\Bq}_h,\overline{\Bq}_h)_{\mathcal T_h}+\langle \tau(u_h-\hat
u_h),\overline{(u_h-\hat u_h)}\rangle_{\partial \mathcal T_h}
+\langle\hat u_h,\overline{\hat u_h}\rangle_{\partial \Omega}\\=(
f,\overline u_h)_{\mathcal T_h}+\langle g,\overline{\hat
u_h}\rangle_{\partial \Omega},
\end{align*}
which implies the lemma.
\end{proof}

Next we use a duality argument to estimate the stability of $u_h$.
Given $u_h\in L^2(\Omega)$, we introduce the dual problem
\begin{align}
 \label{D1}
 -{\bf i}\kappa{\bf \Phi}+\nabla \Psi&=0\qquad \ {\rm in\ }\Omega,\\
  \label{D2}
 {\rm div\, }{\bf \Phi}-{\bf i}\kappa\Psi&=u_h \qquad {\rm in\ }\Omega,\\
  \label{D3}
 {\bf \Phi}\cdot {\Bn}&=\Psi \qquad \ {\rm on\ }\partial \Omega.
\end{align}
In the following lemma, we give some explicit bounds for $\Psi$ and ${\bf \Phi}$.
\begin{lem}\label{lemma46}
For $\Psi$ and ${\bf \Phi}$ defined above, they admit the following
estimate:
\begin{align}
\label{Regularity}
\|\Psi\|_{0,\Omega}+\kappa^{-2}\|\Psi\|_{2,\Omega}+\kappa^{-1}\|\Psi\|_{1,\Omega}+\|\Psi\|_{0,\partial\Omega}
+\kappa^{-1}\|{\bf \Phi}\|_{1,\Omega}\lesssim \|u_h\|_{0,\Omega}.
\end{align}
\end{lem}
\begin{proof}
In fact, $\Psi$ satisfies the following equation
\begin{align*}
\triangle \Psi+\kappa^2\Psi&={\bf i}\kappa u_h \quad {\rm in\  }\Omega\\
\nabla \Psi\cdot {\Bn}&={\bf i}\kappa\Psi \quad {\rm on \
}\partial\Omega.
\end{align*}
In \cite{Feng}, it is proved that
\begin{align*}
\|\Psi\|_{0,\Omega}+\kappa^{-2}\|\Psi\|_{2,\Omega}
+\kappa^{-1}\|{\bf \Phi}\|_{1,\Omega}\lesssim \|u_h\|_{0,\Omega}.
\end{align*}
Since $\Psi$ satisfies the following weak formulation
$a(\Psi, v)=({\bf i}\kappa u_h,v)$, where
$$a(\Psi, v):=-(\nabla \Psi, \overline{\nabla v})+\kappa^2(\Psi,\overline{v})+{\bf i}\kappa\langle\Psi, \overline{v}\rangle_{\partial \Omega}.
$$
Testing the above formulation by $v=\Psi$ and taking the imaginary
part yields
$$\kappa\|\Psi\|_{0,\partial\Omega}^2\leq \kappa\|u_h\|_{0,\Omega}\|\Psi\|_{0,\Omega}\lesssim \kappa\|u_h\|_{0,\Omega}^2,
$$
which finishes the proof of this lemma.
\end{proof}

Now we are ready to derive the stability of $u_h$, which plays an important role in the error analysis for the Helmholtz equation.
\begin{thm}\label{theorem31}
 Let $({\Bq}_h, u_h,\hat u_h)\in {\BV}_h^p\times W_h^p\times M_h^p$ be the solutions of (\ref{P1})-(\ref{P4}). Then
 \begin{align}
 \label{u_h}
 \|u_h\|_{0,\Omega}&\lesssim \Big(1+\frac{\kappa^3h^2}{p^2}\Big)\|f\|_{0,\Omega}+\Big(1+\frac{\kappa^{\frac{3}{2}}h}{p}\Big)\|g\|_{0,\partial\Omega},\\
 \label{q_h}
 \|{\Bq}_h\|_{0,\Omega}&\lesssim \Big(1+\frac{\kappa^3h^2}{p^2}\Big)\|f\|_{0,\Omega}+\Big(1+\frac{\kappa^{\frac{3}{2}}h}{p}\Big)\|g\|_{0,\partial\Omega},\\
 \label{trace}
 \|\hat u_h\|_{0,\partial\mathcal T_h}&\lesssim
 \Big(\big(\frac{\kappa h}{p}\big)^{\frac{1}{2}}+ph^{-\frac{1}{2}}\Big)\Big(\big(1+\frac{\kappa^3h^2}{p^2}\big)\|f\|_{0,\Omega}
 +\big(1+\frac{\kappa^{\frac{3}{2}}h}{p}\big)\|g\|_{0,\partial\Omega}\Big).
\end{align}
\end{thm}
 \begin{proof}
Using (\ref{D2}), Green formulation and the definition of projection
operators, we obtain
 \begin{align*}
 &(u_h,\overline u_h)_{\mathcal T_h}=(u_h, \overline{{\rm div \, }{\bf \Phi}-{\bf i}\kappa\Psi})_{\mathcal T_h}=(u_h, \overline{{\rm div \, }{\bf \Phi}})_{\mathcal T_h}+{\bf i}\kappa(u_h, \overline \Psi)_{\mathcal T_h}\\
 &=-(\nabla u_h,\overline{{\bf \Phi}})_{\mathcal T_h}+\langle u_h, \overline{{\bf \Phi}\cdot {\Bn}}\rangle_{\partial\mathcal T_h}+{\bf i}\kappa(u_h, \overline \Psi)\\
 &=-(\nabla u_h,\overline{{\bf \Pi}_h{\bf \Phi}})_{\mathcal T_h}+\langle u_h, \overline{{\bf \Pi}_h{\bf \Phi}\cdot {\Bn}}\rangle_{\partial\mathcal T_h}+\langle u_h, \overline{({\bf \Phi}-{\bf \Pi}_h{\bf \Phi})\cdot {\Bn}}\rangle_{\partial\mathcal T_h}+{\bf i}\kappa(u_h, \overline \Psi)_{\mathcal T_h}\\
 &=(u_h, \overline{{\rm div \, }{\bf \Pi}_h{\bf \Phi}})_{\mathcal T_h}+\langle u_h, \overline{({\bf \Phi}-{\bf \Pi}_h{\bf \Phi})\cdot {\Bn}}\rangle_{\partial\mathcal T_h}+{\bf i}\kappa(u_h, \overline {\Pi_h\Psi})_{\mathcal T_h}.
  \end{align*}
 Hence using (\ref{P1}) and the fact that ${\bf \Phi}\cdot {\Bn}$ is continuous across the inner edges, the above equality can be rewritten as
 \begin{align*}
 &(u_h,\overline u_h)_{\mathcal T_h}=({\bf i}\kappa{\Bq}_h,\overline{{\bf \Pi}_h{\bf \Phi}})_{\mathcal T_h}+\langle \hat u_h, \overline{{\bf \Pi}_h{\bf \Phi}\cdot {\Bn}}\rangle_{\partial\mathcal T_h}-\langle \hat u_h, \overline{{\bf \Phi}\cdot {\Bn}}\rangle_{\partial\mathcal T_h}+\langle \hat u_h, \overline{{\bf \Phi}\cdot {\Bn}}\rangle_{\partial\mathcal T_h}\\
 &+\langle u_h, \overline{({\bf \Phi}-{\bf \Pi}_h{\bf \Phi})\cdot {\Bn}}\rangle_{\partial\mathcal T_h}+{\bf i}\kappa(u_h, \overline {\Pi_h\Psi})_{\mathcal T_h}\\
 &=({\bf i}\kappa{\Bq}_h,\overline{{\bf \Phi}})_{\mathcal T_h}+\langle u_h-\hat u_h, \overline{({\bf \Phi}-{\bf \Pi}_h{\bf \Phi})\cdot {\Bn}}\rangle_{\partial\mathcal T_h}+\langle \hat u_h, \overline{{\bf \Phi}\cdot {\Bn}}\rangle_{\partial\Omega}+{\bf i}\kappa(u_h, \overline {\Pi_h\Psi})_{\mathcal T_h}.
  \end{align*}
Green formulation and (\ref{D1}) indicate that
\begin{align*}
({\bf i}\kappa{\Bq}_h,\overline{{\bf \Phi}})_{\mathcal
T_h}=({\Bq}_h,\overline{-\nabla \Psi})_{\mathcal T_h}=({\rm
div}\,{\Bq}_h,\overline{ \Psi})_{\mathcal T_h}- \langle{\Bq}_h\cdot
{\Bn}, \overline \Psi\rangle_{\partial\mathcal T_h}=({\rm
div}\,{\Bq}_h,\overline{\Pi_h \Psi})_{\mathcal T_h}-
\langle{\Bq}_h\cdot {\Bn}, \overline \Psi\rangle_{\partial\mathcal
T_h}.
\end{align*}
Combining (\ref{P2})-(\ref{P4}) and (\ref{D3}) gives
 \begin{align*}
&(u_h,\overline u_h)_{\mathcal T_h}=({\bf i}\kappa u_h+{\rm div\, }{\Bq}_h,\overline {\Pi_h\Psi})_{\mathcal T_h}-\langle{\Bq}_h\cdot {\Bn}, \overline \Psi\rangle_{\partial\mathcal T_h}\\&+\langle u_h-\hat u_h, \overline{({\bf \Phi}-{\bf \Pi}_h{\bf \Phi})\cdot {\Bn}}\rangle_{\partial\mathcal T_h}+\langle \hat u_h, \overline{{\bf \Phi}\cdot {\Bn}}\rangle_{\partial\Omega}\\
 &=(f,\overline {\Pi_h\Psi})_{\mathcal T_h}-\langle\tau(u_h-\hat u_h),\overline {\Pi_h\Psi}\rangle_{\partial\mathcal T_h}-\langle{\Bq}_h\cdot {\Bn}, \overline \Psi\rangle_{\partial\mathcal T_h}+\langle\hat{\Bq}_h\cdot {\Bn}, \overline \Psi\rangle_{\partial\mathcal T_h}\\&-\langle\hat{\Bq}_h\cdot {\Bn}, \overline \Psi\rangle_{\partial\Omega}
 +\langle u_h-\hat u_h, \overline{({\bf \Phi}-{\bf \Pi}_h{\bf \Phi})\cdot {\Bn}}\rangle_{\partial\mathcal T_h}+\langle \hat u_h, \overline{{\bf \Phi}\cdot {\Bn}}\rangle_{\partial\Omega}\\
 &=(f,\overline {\Pi_h\Psi})_{\mathcal T_h}+\langle\tau(u_h-\hat u_h),\overline {\Psi-\Pi_h\Psi}\rangle_{\partial\mathcal T_h}+\langle u_h-\hat u_h, \overline{({\bf \Phi}-{\bf \Pi}_h{\bf \Phi})\cdot {\Bn}}\rangle_{\partial\mathcal T_h}+\langle P_Mg,\Psi\rangle_{\partial\Omega},
 \end{align*}
 where we have used that the normal component
of  $\hat {{\Bq}}_h$ across interelement boundaries is continuous
and $P_Mg\in \mathcal P_p(\partial \mathcal T_h\cap \partial
\Omega), \ \langle P_Mg, v\rangle_{\partial \Omega}=\langle g,
v\rangle_{\partial \Omega},\ \forall v \in \mathcal P_p(\partial
\mathcal T_h\cap \partial \Omega)$.
 So we can get
\begin{align*}
\|u_h\|_{0,\Omega}^2&\lesssim \|f\|_{0,\Omega}\|\Psi\|_{0,\Omega}+
\tau\|(u_h-\hat u_h)\|_{0,\partial\mathcal
T_h}\|{\Psi-\Pi_h\Psi}\|_{0,\partial\mathcal T_h}\\&+ \|u_h-\hat
u_h\|_{0,\partial\mathcal T_h} \|{\bf \Phi}-{\bf \Pi}_h{\bf
\Phi}\|_{0,\partial\mathcal T_h}+\|g\|_{0,\partial
\Omega}\|\Psi\|_{0,\partial \Omega}.
\end{align*}
Applying Lemmas \ref{lemma34}-\ref{lemma35} and the regularity estimate (\ref{Regularity}), we get
\begin{align*}
\|u_h\|_{0,\Omega}^2&\lesssim
\|f\|_{0,\Omega}\|u_h\|_{0,\Omega}+\|g\|_{0,\partial
\Omega}\|u_h\|_{0,\Omega}\\&+\Big(\tau^{\frac{1}{2}}\kappa^2\big(\frac{h}{p}\big)^{\frac{3}{2}}+\tau^{-\frac{1}{2}}\kappa\big(\frac{h}{p}\big)^{\frac{1}{2}}\Big)
\Big(\|f\|^{\frac{1}{2}}_{0,\Omega}\|u_h\|^{\frac{1}{2}}_{0,\Omega}+\|g\|_{0,\partial
\Omega}\Big)\|u_h\|_{0,\Omega}.
\end{align*}
Note that we choose $\tau=\frac{p}{\kappa h}$ to get the minimum of
the term
$\tau^{\frac{1}{2}}\kappa^2(\frac{h}{p})^{\frac{3}{2}}+\tau^{-\frac{1}{2}}\kappa(\frac{h}{p})^{\frac{1}{2}}$.
Eliminating $\|u_h\|_{0,\Omega}$ from both sides of the equation, we
can get
\begin{align*}
\|u_h\|_{0,\Omega}&\leq
C\|f\|_{0,\Omega}+C\|g\|_{0,\partial
\Omega}+C\kappa^{\frac{3}{2}}\frac{h}{p}
(\|f\|^{\frac{1}{2}}_{0,\Omega}\|u_h\|^{\frac{1}{2}}_{0,\Omega}+\|g\|_{0,\partial
\Omega})\\
&\leq C(1+\frac{\kappa^3
h^2}{p^2})\|f\|_{0,\Omega}+C(1+\frac{\kappa^{\frac{3}{2}}h}{p})\|g\|_{0,\partial\Omega}+\delta\|u_h\|_{0,\Omega}
.
\end{align*}
Choosing $\delta \leq \frac{1}{2}$, (\ref{u_h}) is obtained. Using
(\ref{q_h-H}), the bound for ${\Bq}_h$ is deduced. According to
Lemma \ref{lemma33},
$$\|u_h\|_{0,\partial\Omega}\lesssim ph^{-{\frac{1}{2}}}\|u_h\|_{0,\Omega},
$$
which combined with (\ref{u_h}), (\ref{trace-H}) and the triangle inequality yields (\ref{trace}).
\end{proof}

\section{Error estimates of an auxiliary problem}
In this section, we derive the error estimates of the solutions for the auxiliary problem
\begin{align}
 {\bf i}\kappa{\BQ}+\nabla U &=0\qquad \qquad \quad \ {\rm in\ }\Omega, \nn \\
 {\rm div\, }{\BQ} -{\bf i}\kappa U &=f - 2 {\bf i} \kappa u \qquad {\rm in\ }\Omega,\nn \\
 -{\BQ}\cdot {\Bn}+U&=g \qquad \qquad \quad \ {\rm on\ }\partial \Omega,\nn
\end{align}
where $u,f$ and $g$ are determined by the problem
(\ref{HS})-(\ref{HE}). The HDG scheme of this problem is to find
$({\BQ}_h, U_h,\hat \lambda_h)\in {\BV}_h^p\times W_h^p\times M_h^p$
such that
\begin{align}
\label{AP1}
({\bf i}\kappa {\BQ}_h,\overline{\Br})_{\mathcal T_h}-(U_h, \overline{{\rm div \, {\Br}}})_{\mathcal T_h}+\langle  \lambda_h,\overline {{\Br}\cdot {\Bn}}\rangle_{\partial \mathcal T_h}&=0,\\
\label{AP2}
-({\bf i}\kappa U_h,\overline w)_{\mathcal T_h}-({\BQ}_h,\overline{\nabla w})_{\mathcal T_h}+\langle \hat {\BQ}_h\cdot {\Bn} ,\overline w\rangle_{\partial \mathcal T_h}&=(f-2{\bf i}\kappa u,\overline w)_{\mathcal T_h} ,\\
\label{AP3}
\langle -\hat{\BQ}_h\cdot {\Bn}+ \ \lambda_h,\overline \mu\rangle_{\partial \Omega}&=\langle g, \overline \mu\rangle_{\partial \Omega},\\
\label{AP4} \langle \hat{\BQ}_h\cdot {\Bn},\overline
\mu\rangle_{\partial \mathcal T_h \backslash
 \partial \Omega}&=0,
\end{align}
for all ${\Br}\in {\BV}_h^p$, $w\in W_h^p$ and $\mu\in M_h^p$, where
\begin{align}
\hat {\BQ}_h={\BQ}_h+\tau (U_h-\lambda_h){\Bn} \qquad {\rm on } \
\partial \mathcal T_h.
\end{align}
Inserting the expression of $\hat {\BQ}_h$ into (\ref{AP3}) and
(\ref{AP4}), we obtain that on the edge $e\in \partial \mathcal
T_h\backslash\partial \Omega$
\begin{align*}
\lambda_h&=\{\hspace{-0.1cm}\{U_h\}\hspace{-0.1cm}\}+\frac{1}{2\tau}\llbracket{\BQ}_h\rrbracket,\\
\hat{{\BQ}}_h\cdot {\Bn} &= \{\hspace{-0.1cm}\{ {\BQ}_h
\}\hspace{-0.1cm}\} \cdot {\Bn}+ \frac{\tau}{2} \llbracket
U_h\rrbracket {\Bn},
\end{align*}
and on the boundary edge $e\in \partial \mathcal T_h\cap\partial \Omega$
\begin{align*}
\lambda_h&=\frac{{\BQ}_h\cdot {\Bn}+\tau U_h}{1+\tau}+\frac{P_Mg}{1+\tau},\\
\hat{{\BQ}}_h\cdot {\Bn}-{{\BQ}}_h\cdot
{\Bn}&=\frac{\tau(U_h-{\BQ}_h\cdot {\Bn}-P_Mg)}{1+\tau}.
\end{align*}
We can substitute the above expressions into
(\ref{AP1})-(\ref{AP4}), and get the equivalent formulations of
${{\BQ}}_h$ and $U_h$ as follows:
\begin{equation}\begin{split}
\label{M1} {\BB}_1({\BQ}_h,{U}_h; {\Br}, w) & := ({\bf i}\kappa
{\BQ}_h,\overline{\Br})_{\mathcal T_h}-(U_h, \overline{{\rm div \,
{\Br}}})_{\mathcal T_h}+\sum_{e\in\mathcal E_h^0}\big(\langle
\{\hspace{-0.1cm}\{U_h\}\hspace{-0.1cm}\},\llbracket\overline{{\Br}}\rrbracket\rangle_e+
\langle\frac{1}{2\tau}\llbracket{\BQ}_h\rrbracket,\llbracket\overline{{\Br}}\rrbracket\rangle_e\big)\\
&+\langle \frac{{\BQ}_h\cdot {\Bn}}{1+\tau},\overline{{\Br}\cdot
{\Bn}}\rangle_{\partial \Omega}+ \langle
\frac{\tau}{1+\tau}U_h,\overline{{\Br}\cdot {\Bn}}\rangle_{\partial
\Omega}= -\langle \frac{1}{1+\tau}g,\overline{{\Br}\cdot
{\Bn}}\rangle_{\partial \Omega},
\end{split}
\end{equation}
\begin{equation}\begin{split}
\label{M2} {\BB}_2({\BQ}_h,{U}_h; {\Br}, w) &:=  -({\bf i}\kappa
U_h,\overline w)_{\mathcal T_h}+({\rm div\, }{\BQ}_h,\overline{
w})_{\mathcal T_h}-\sum_{e\in\mathcal
E_h^0}\langle\llbracket{\BQ}_h\rrbracket,\{\hspace{-0.1cm}\{\overline{w}\}\hspace{-0.1cm}\}\rangle_e+
\sum_{e\in\mathcal E_h^0}\langle\frac{\tau}{2}\llbracket
U_h\rrbracket,\llbracket\overline{w}\rrbracket\rangle_e\\&+
\big\langle\frac{\tau}{1+\tau}(U_h-{\BQ}_h\cdot{\Bn}), \overline
w\big\rangle_{\partial \Omega}=(f-2{\bf i}\kappa u,\overline
w)_{\mathcal T_h}+\big\langle\frac{\tau}{1+\tau}g, \overline
w\big\rangle_{\partial \Omega},
\end{split}
\end{equation}
for all ${\Br}\in {\BV}_h^p$, $w\in W_h^p$. Define
\begin{align*}
&\mathcal A({\BQ}_h,U_h;{\Br},w):={\BB}_1({\BQ}_h,{U}_h; {\Br}, w) +
\overline{{\BB}_2({\BQ}_h,{U}_h; {\Br}, w)},
\end{align*}
and
\begin{align*}
&\mathcal A_1({\BQ}_h,U_h;{\Br},w):={\BB}_1({\BQ}_h,{U}_h; {\Br}, w)
+{\BB}_2({\BQ}_h,{U}_h; {\Br}, w).
\end{align*}
An obvious observation is that
\begin{equation}\begin{split}
\label{bilinear} \mathcal A({\BQ}_h,U_h;{\BQ}_h,U_h)&={\bf
i}\kappa\big(\|{\BQ}_h\|_{0,\Omega}^2+\|U_h\|_{0,\Omega}^2\big)+
\sum_{e\in\mathcal
E_h^0}\big(\frac{1}{2\tau}\|\llbracket{\BQ}_h\rrbracket\|_{0,e}^2+\frac{\tau}{2}\|\llbracket
U_h\rrbracket\|_{0,e}^2\big)\\&
+\frac{1}{1+\tau}\|{\BQ}_h\cdot{\Bn}\|_{0,\partial\Omega}^2
+\frac{\tau}{1+\tau}\|U_h\|_{0,\partial\Omega}^2.
\end{split}
\end{equation}
Since the formulation is consistent, we have
\begin{align}
\label{consistent} \mathcal A({\Bq}-{\BQ}_h, u-U_h;{\Br},w)=\mathcal
A_1({\Bq}-{\BQ}_h, u-U_h;{\Br},w)=0,
\end{align}
for all ${\Br}\in {\BV}_h^p$, $w\in W_h^p$. We first give the error
estimation of the flux ${\BQ}_h$, and then use the duality argument
to bound the $L^2$-error of the discrete solution $U_h$.
\begin{thm}\label{theorem41}
Let ${{\BQ}}_h$ and $U_h$ be the solution of (\ref{M1})-(\ref{M2}).
Denote $e_{\Bq}:={\Bq}-{\BQ}_h={\Bq}-{\bf \Pi}_h{\Bq}+{\bf
\Pi}_he_{\Bq}$, $e_u:=u-U_h=u-\Pi_hu+\Pi_he_u$ and
\begin{align*}
E :=&  \ \kappa\|{\bf \Pi}_he_{\Bq}\|_{0,\Omega}^2+\kappa\|
\Pi_he_u\|_{0,\Omega}^2+
\sum_{e\in\mathcal E_h^0}(\frac{1}{2\tau}\|\llbracket{\bf \Pi}_he_{\Bq}\rrbracket\|_{0,e}^2+\frac{\tau}{2}\|\llbracket\Pi_he_u\rrbracket\|_{0,e}^2)\\
&+\frac{1}{1+\tau}\|{\bf
\Pi}_he_{\Bq}\cdot{\Bn}\|_{0,\partial\Omega}^2
+\frac{\tau}{1+\tau}\|\Pi_he_u\|_{0,\partial\Omega}^2.
\end{align*}
Then there hold the following estimates:
\begin{align}
\label{E}
E&\lesssim \frac{\kappa h^2}{p^2}M^2(\tilde f,\tilde g),\\
\label{Q_h} \|{\Bq}-{\BQ}_h\|_{0,\Omega}&\lesssim
\frac{h}{p}M(\tilde f,\tilde g).
\end{align}
\end{thm}
\begin{proof}
Direct calculation shows that
\begin{align}
\label{eq} \sum_{T\in \mathcal T_h}\int_{\partial T}\xi {\bf
\Theta}\cdot{\Bn}ds=\sum_{e\in \mathcal E_h}\int_e\llbracket
\xi\rrbracket\{\hspace{-0.1cm}\{{\bf
\Theta}\}\hspace{-0.1cm}\}ds+\sum_{e\in \mathcal
E_h^0}\int_e\{\hspace{-0.1cm}\{\xi \}\hspace{-0.1cm}\}\llbracket{\bf
\Theta}\rrbracket ds,
\end{align}
for all ${\bf \Theta}\in \big(H^1(\Omega_h)\big)^d$ and $\xi \in
H^1(\Omega_h)$. According to (\ref{consistent}) and (\ref{eq}) we
have
\begin{align*}
&\mathcal A({\bf \Pi}_he_{\Bq},\Pi_he_u;{\bf
\Pi}_he_{\Bq},\Pi_he_u)=
\mathcal A({\bf \Pi}_h{\Bq}-{\BQ}_h,\Pi_hu-U_h;{\bf \Pi}_he_{\Bq},\Pi_he_u)\\
&=\mathcal A({\bf \Pi}_h{\Bq}-{\Bq},\Pi_hu-u;{\bf
\Pi}_he_{\Bq},\Pi_he_u)+
\mathcal A({\Bq}-{\BQ}_h,u-U_h;{\bf \Pi}_he_{\Bq},\Pi_he_u)\\
&=\mathcal A({\bf \Pi}_h{\Bq}-{\Bq},\Pi_hu-u;{\bf \Pi}_he_{\Bq},\Pi_he_u)\\
&=\sum_{e\in\mathcal E_h^0}\big(\langle
\{\hspace{-0.1cm}\{\Pi_hu-u\}\hspace{-0.1cm}\},\llbracket\overline{{\bf
\Pi}_he_{\Bq}}\rrbracket\rangle_e+
\langle\frac{1}{2\tau}\llbracket{\bf
\Pi}_h{\Bq}-{\Bq}\rrbracket,\llbracket\overline{{\bf
\Pi}_he_{\Bq}}\rrbracket\rangle_e\big)\\&+\big\langle \frac{({\bf
\Pi}_h{\Bq}-{\Bq})\cdot {\Bn}}{1+\tau},\overline{{\bf
\Pi}_he_{\Bq}\cdot {\Bn}}\big\rangle_{\partial \Omega}+ \big\langle
\frac{\tau}{1+\tau}(\Pi_hu-u),\overline{{\bf \Pi}_he_{\Bq}\cdot
{\Bn}}\big\rangle_{\partial \Omega}\\&+\sum_{e\in\mathcal
E_h^0}\langle \llbracket{ \Pi_he_u
}\rrbracket,\overline{\{\hspace{-0.1cm}\{{\bf
\Pi}_h{\Bq}-{\Bq}\}\hspace{-0.1cm}\}}\rangle_e+ \sum_{e\in\mathcal
E_h^0}\langle\frac{\tau}{2}\llbracket{\Pi_he_u}\rrbracket,\overline{\llbracket\Pi_hu-u\rrbracket}\rangle_e
\\&+\frac{\tau}{1+\tau}\langle{\Pi_he_u},\overline{\Pi_hu-u}\rangle_{\partial \Omega}+\frac{1}{1+\tau}\langle { \Pi_he_u },\overline{({\bf \Pi}_h{\Bq}-{\Bq})\cdot{\Bn}}\rangle_{\partial \Omega}.
\end{align*}
Using (\ref{bilinear}) and the Young's inequality we obtain
\begin{align*}
\frac{1}{2}E&\leq C_{\delta}\tau\sum_{e\in\mathcal E_h^0}\|\Pi_hu-u\|_{0,e}^2+\delta\sum_{e\in\mathcal E_h^0}\frac{1}{2\tau}\|\llbracket{\bf \Pi}_he_{\Bq}\rrbracket\|_{0,e}^2+C_{\delta}\frac{1}{2\tau}\sum_{e\in\mathcal E_h^0}\|{\bf \Pi}_h{\Bq}-{\Bq}\|_{0,e}^2\\
&+\frac{C_{\delta}}{1+\tau}\|{\bf \Pi}_h{\Bq}-{\Bq}\|_{0,\partial
\Omega}^2+\frac{\delta}{1+\tau}\|{\bf
\Pi}_he_{\Bq}\cdot{\Bn}\|_{0,\partial\Omega}^2+
C_{\delta}\tau\|\Pi_hu-u\|_{0,\partial\Omega}^2\\
&+C_{\delta}\tau^{-1}\sum_{e\in\mathcal E_h^0}\|{\bf
\Pi}_h{\Bq}-{\Bq}\|_{0,e}^2+\delta \sum_{e\in\mathcal
E_h^0}\frac{\tau}{2}\|\llbracket\Pi_he_u\rrbracket\|_{0,e}^2+
\delta\frac{\tau}{1+\tau}\|\Pi_he_u\|_{0,\partial\Omega}^2\\&+\frac{C_{\delta}}{\tau(1+\tau)}\|{\bf
\Pi}_h{\Bq}-{\Bq}\|_{0,\partial \Omega}^2.
\end{align*}
Choosing $\delta\leq \frac{1}{4}$ and taking advantage of Lemma \ref{lemma34}, we get
\begin{align*}
{\frac{1}{4}}E \lesssim \tau\sum_{e\in\mathcal
E_h}\|\Pi_hu-u\|_{0,e}^2+\tau^{-1}\sum_{e\in\mathcal E_h}\|{\bf
\Pi}_h{\Bq}-{\Bq}\|_{0,e}^2  \lesssim \frac{\kappa
h^2}{p^2}M^2(\tilde f,\tilde g),
\end{align*}
and hence (\ref{E}) is deduced. Then (\ref{Q_h}) follows from Lemma \ref{lemma34} and (\ref{E}).
\end{proof}

Next we establish an error estimate for $U_h$, we perform an analogue of the Aubin-Nitsche duality argument to get the convergence rate. First we begin by introducing the dual problem
\begin{align*}
 -{\bf i}\kappa{\bf \Phi}+\nabla \Psi&=0\qquad \ {\rm in\ }\Omega,\\
 {\rm div\, }{\bf \Phi}+{\bf i}\kappa\Psi&=e_u \qquad {\rm in\ }\Omega,\\
 {\bf \Phi}\cdot {\Bn}&=\Psi \qquad \ {\rm on\ }\partial \Omega,
\end{align*}
and prove its regularity estimations.
\begin{lem}\label{lemma42}
Let ${\bf \Phi}$ and $\Psi$ be defined above, then they admit the following estimate
\begin{align}
\label{Aregularity} \|{\bf
\Phi}\|_{1,\Omega}+\|\Psi\|_{1,\Omega}+\kappa^{-1}\|\Psi\|_{2,\Omega}\lesssim
\|e_u\|_{0,\Omega}.
\end{align}
\end{lem}
\begin{proof}
Direct calculation shows that $\Psi$ satisfies the equation as follows
\begin{align*}
\triangle \Psi-\kappa^2\Psi&={\bf i}\kappa e_u \quad {\rm in\  }\Omega\\
\nabla \Psi\cdot {\Bn}&={\bf i}\kappa\Psi \quad {\rm on \
}\partial\Omega.
\end{align*}
It is well known that $\Psi$ is the solution of the following weak problem, for all $v\in H^1(\Omega)$
$$\hat a(\Psi,v):=(\nabla\Psi,\overline{\nabla v})+\kappa^2(\Psi,\overline{v})-{\bf i}\kappa\langle\Psi,\overline{v}\rangle_{\partial\Omega}=(-{\bf i}\kappa e_u,\overline{v}).
$$
Taking $v=\Psi$, we get
$$\|\Psi\|_{0,\Omega}\leq \kappa^{-1}\|e_u\|_{0,\Omega},$$
and
$$\|\Psi\|_{1,\Omega}^2\lesssim |\Psi|_{1,\Omega}^2+\|\Psi\|_{0,\partial\Omega}^2\lesssim \kappa\|e_u\|_{0,\Omega}\|\Psi\|_{0,\Omega}\lesssim \|e_u\|_{0,\Omega}^2,
$$
where we have used Poincar\'{e} inequality. The regularity theory
for the Laplace problem (see Chap 2 of \cite{PG}) gives the bound
for $|\Psi|_{2,\Omega}$,
\begin{align*}
|\Psi|_{2,\Omega}&\lesssim \|{\bf i}\kappa e_u-\kappa^2\Psi\|_{0,\Omega}+\|{\bf i}\kappa\Psi\|_{H^{\frac{1}{2}}(\partial \Omega)}\\
&\lesssim \kappa\|e_u\|_{0,\Omega}+
\kappa\|\Psi\|_{1,\Omega}\lesssim  \kappa\|e_u\|_{0,\Omega}.
\end{align*}
Combining the definition of ${\bf
\Phi}$ and the above estimates completes the proof of (\ref{Aregularity}).
\end{proof}

Now for any ${{\bf \Theta},\Br} \in \big( H^1(\Omega_h) \big)^d$ and
$\xi,w \in H^1(\Omega_h)$,  we define the following bilinear form
\begin{align*}
&\tilde {\mathcal A}({{\bf \Theta}},\xi;{\Br},w):=-({\bf i}\kappa
{\bf \Theta} ,\overline{\Br})_{\mathcal T_h}-(\xi , \overline{{\rm
div \, {\Br}}})_{\mathcal T_h}+\sum_{e\in\mathcal E_h^0}\big(\langle
\{\hspace{-0.1cm}\{\xi
\}\hspace{-0.1cm}\},\llbracket\overline{{\Br}}\rrbracket\rangle_e+
\langle\frac{1}{2\tau}\llbracket{\bf \Theta} \rrbracket,\llbracket\overline{{\Br}}\rrbracket\rangle_e\big)\\
&+\big\langle \frac{{\bf \Theta} \cdot
{\Bn}}{1+\tau},\overline{{\Br}\cdot {\Bn}}\big\rangle_{\partial
\Omega}+ \big\langle \frac{\tau}{1+\tau}\xi ,\overline{{\Br}\cdot
{\Bn}}\big\rangle_{\partial \Omega}+({\bf i}\kappa \xi ,\overline
w)_{\mathcal T_h}+({\rm div \, }{\bf \Theta} ,\overline{
w})_{\mathcal T_h}\\&-\sum_{e\in\mathcal E_h^0}\langle\llbracket{\bf
\Theta}
\rrbracket,\{\hspace{-0.1cm}\{\overline{w}\}\hspace{-0.1cm}\}\rangle_e+
\sum_{e\in\mathcal E_h^0}\langle\frac{\tau}{2}\llbracket \xi
\rrbracket,\llbracket\overline{w}\rrbracket\big\rangle_e+
\langle\frac{\tau}{1+\tau}(\xi -{\bf \Theta}\cdot {\Bn} ), \overline
w\big\rangle_{\partial \Omega} .
\end{align*}
Direct calculation shows that
\begin{align}
\label{dual} \tilde {\mathcal A}({\bf \Theta},\xi;\Br,w)=\overline{
\mathcal A_1(-\Br,w;-{\bf \Theta},\xi)}.
\end{align}
Moreover, the consistency of the bilinear form implies that
\begin{align}
\label{Aconsistent} \tilde {\mathcal A}({\bf
\Phi},\Psi;-e_{\Bq},e_u)=(e_u,\overline{e_u}),
\end{align}
where $e_{\Bq},e_u$ and ${\bf \Phi},\Psi$ are defined in Theorem
\ref{theorem41} and Lemma \ref{lemma42} respectively.
\begin{thm}\label{theorem42}
Let ${{\Bq}}_h$ and $U_h$ be the solution of (\ref{M1})-(\ref{M2}).
There holds
\begin{align}
\label{U_h} \|u-U_h\|_{0,\Omega}\lesssim \frac{\kappa
h^2}{p^2}M(\tilde f,\tilde g).
\end{align}
\end{thm}
\begin{proof}
Using  (\ref{Aconsistent}), (\ref{dual}) and (\ref{consistent}), we have
\begin{equation}\begin{split}
\label{e_u} &\|e_u\|_{0,\Omega}^2=\tilde {\mathcal A}({\bf
\Phi},\Psi;-e_{\Bq},e_u)
=\overline{\mathcal A_1(e_{\Bq},e_u;-{\bf \Phi},\Psi)}\\
&=\overline{\mathcal A_1(e_{\Bq},e_u;-{\bf \Phi}+{\bf \Pi}_h{\bf
\Phi},\Psi-\Pi_h\Psi)}
=\overline{\mathcal A_1({\bf \Pi}_he_{\Bq},\Pi_he_u;-{\bf \Phi}+{\bf \Pi}_h{\bf \Phi},\Psi-\Pi_h\Psi)}\\
&+\overline{\mathcal A_1({\Bq}-{\bf \Pi}_h{\Bq},u-\Pi_hu;-{\bf
\Phi}+{\bf \Pi}_h{\bf \Phi},\Psi-\Pi_h\Psi)}.
\end{split}
\end{equation}
Denote $$T_1:=\mathcal A_1({\bf \Pi}_he_{\Bq},\Pi_he_u;-{\bf
\Phi}+{\bf \Pi}_h{\bf \Phi},\Psi-\Pi_h\Psi)$$ and $$T_2:=\mathcal
A_1({\Bq}-{\bf \Pi}_h{\Bq},u-\Pi_hu;-{\bf \Phi}+{\bf \Pi}_h{\bf
\Phi},\Psi-\Pi_h\Psi).$$ Then we estimate the above two terms
respectively. By (\ref{eq}) and the property of the projection
operators we can rewrite $T_1$ as
\begin{align*}
T_1&=-\sum_{e\in\mathcal E_h^0}\langle
\llbracket\Pi_he_u\rrbracket,\{\hspace{-0.1cm}\{\overline{{\bf
\Pi}_h{\bf \Phi}-{\bf
\Phi}}\}\hspace{-0.1cm}\}\rangle_e+\sum_{e\in\mathcal
E_h^0}\langle\frac{1}{2\tau}\llbracket{\bf \Pi}_he_{\Bq}
\rrbracket,\llbracket\overline{{\bf \Pi}_h{\bf \Phi}-{\bf
\Phi}}\rrbracket\rangle_e\\& -\frac{1}{1+\tau}\langle\Pi_he_u,({\bf
\Pi}_h{\bf \Phi}-{\bf \Phi})\cdot{\Bn}\rangle_{\partial
\Omega}+\frac{1}{1+\tau}\langle{\bf \Pi}_he_{\Bq}\cdot{\Bn},({\bf
\Pi}_h{\bf \Phi}-{\bf \Phi})\cdot{\Bn}\rangle_{\partial \Omega}\\&
-\sum_{e\in\mathcal E_h^0}\langle \llbracket\Pi_he_{\Bq}
\rrbracket,\{\hspace{-0.1cm}\{\overline{{ \Psi}-{ \Pi}_h{
\Psi}}\}\hspace{-0.1cm}\}\rangle_e+\sum_{e\in\mathcal
E_h^0}\frac{\tau}{2}\langle \llbracket\Pi_he_u
\rrbracket,\llbracket\overline{{ \Psi}-{ \Pi}_h{
\Psi}}\rrbracket\rangle_e\\& +\frac{\tau}{1+\tau}\langle \Pi_he_u
,\overline{{ \Psi}-{ \Pi}_h{ \Psi}}\rangle_{\partial\Omega}-
\frac{\tau}{1+\tau}\langle {\bf \Pi}_he_{\Bq}\cdot{\Bn} ,\overline{{
\Psi}-{ \Pi}_h{ \Psi}}\rangle_{\partial\Omega}.
\end{align*}
Applying the Cauchy-Schwarz inequality, we have
\begin{align*}
|T_1|&\leq \Big(\sum_{e\in\mathcal E_h^0}\tau
\|\llbracket\Pi_he_u\rrbracket\|_{0,e}^2\Big)^{\frac{1}{2}}\cdot
\Big({\tau}^{-1}\sum_{e\in\mathcal E_h^0}\|{\bf \Pi}_h{\bf
\Phi}-{\bf
\Phi}\|_{0,e}^2\Big)^{\frac{1}{2}}\\
& +\Big(\sum_{e\in\mathcal E_h^0}\frac{1}{2\tau} \|\llbracket{\bf
\Pi}_he_{\Bq}\rrbracket\|_{0,e}^2\Big)^{\frac{1}{2}}\cdot
\Big(\frac{1}{2\tau}\sum_{e\in\mathcal E_h^0}\|{\bf \Pi}_h{\bf
\Phi}-{\bf \Phi}\|_{0,e}^2\Big)^{\frac{1}{2}}\\
& +\Big(\frac{\tau}{(1+\tau)^2}
\|\Pi_he_u\|_{0,\partial\Omega}^2\Big)^{\frac{1}{2}}\cdot
\Big({\tau}^{-1}\|{\bf \Pi}_h{\bf \Phi}-{\bf \Phi}\|_{0,\partial\Omega}^2\Big)^{\frac{1}{2}}\\
&+\Big(\frac{1}{1+\tau} \|{\bf
\Pi}_he_{\Bq}\cdot{\Bn}\|_{0,\partial\Omega}^2\Big)^{\frac{1}{2}}\cdot
\Big(\frac{1}{1+\tau}\|{\bf \Pi}_h{\bf \Phi}-{\bf \Phi}\|_{0,\partial\Omega}^2\Big)^{\frac{1}{2}}\\
&+\Big(\sum_{e\in\mathcal E_h^0}\frac{1}{\tau} \|\llbracket{\bf \Pi}_he_{\Bq}\rrbracket\|_{0,e}^2\Big)^{\frac{1}{2}}\cdot\Big(\tau\sum_{e\in\mathcal E_h^0}\|{ \Psi}-{ \Pi}_h{ \Psi}\|_{0,e}^2\Big)^{\frac{1}{2}}\\
&+(\sum_{e\in\mathcal E_h^0}\frac{\tau}{2}\|\llbracket\Pi_he_u\rrbracket\|_{0,e}^2\Big)^{\frac{1}{2}} \cdot\Big(\frac{\tau}{2}\sum_{e\in\mathcal E_h^0}\|{ \Psi}-{ \Pi}_h{ \Psi}\|_{0,e}^2\Big)^{\frac{1}{2}}\\
&+\Big(\frac{\tau}{1+\tau}\|\Pi_he_u\|_{0,\partial\Omega}^2)^{\frac{1}{2}}\cdot
\Big(\frac{\tau}{1+\tau}\|{ \Psi}-{ \Pi}_h{ \Psi}\|_{0,\partial\Omega}^2\Big)^{\frac{1}{2}}\\
&+\Big(\frac{\tau}{(1+\tau)^2} \|{\bf
\Pi}_he_{\Bq}\cdot{\Bn}\|_{0,\partial\Omega}^2\Big)^{\frac{1}{2}}\cdot
\Big(\tau\|{ \Psi}-{ \Pi}_h{
\Psi}\|_{0,\partial\Omega}^2\Big)^{\frac{1}{2}}.
\end{align*}
Then the upper bound for $T_1$ follows from Lemma \ref{lemma34}, Theorem \ref{theorem41} and the regularity estimation (\ref{Aregularity}) that
\begin{align}
\label{T_1} |T_1|\lesssim \frac{\kappa h^2}{p^2}M(\tilde f,\tilde
g)\|e_u\|_{0,\Omega}.
\end{align}
Similarly we use the property of the projection operators and get
\begin{align*}
T_2&={\bf i}\kappa({\Bq}-{\bf \Pi}_h{\Bq},\overline{{\bf \Pi}_h{\bf
\Phi}-{\bf \Phi}})_{\mathcal T_h}
+(\nabla(u-\pi_h^pu),\overline{{\bf \Pi}_h{\bf \Phi}-{\bf
\Phi}})_{\mathcal T_h}\\&
-\sum_{e\in\mathcal E_h}\langle \llbracket u-\Pi_hu\rrbracket,\{\hspace{-0.1cm}\{\overline{{\bf \Pi}_h{\bf \Phi}-{\bf \Phi}}\}\hspace{-0.1cm}\}\rangle_e+\sum_{e\in\mathcal E_h^0}\frac{1}{2\tau}\langle \llbracket{\Bq}-{\bf \Pi}_h{\Bq}\rrbracket,\llbracket\overline{{\bf \Pi}_h{\bf \Phi}-{\bf \Phi}}\rrbracket\rangle_e\\
&+\frac{1}{1+\tau}\langle ({\Bq}-{\bf
\Pi}_h{\Bq})\cdot{\Bn},\overline{({\bf \Pi}_h{\bf \Phi}-{\bf
\Phi})\cdot{\Bn}}\rangle_{\partial \Omega}+
\frac{\tau}{1+\tau}\langle (u-\Pi_hu),\overline{({\bf \Pi}_h{\bf
\Phi}-{\bf \Phi})\cdot{\Bn}}\rangle_{\partial \Omega}\\&
-{\bf i}\kappa(u-\Pi_hu,\overline{\Psi-\Pi_h\Psi})_{\mathcal T_h}-({\Bq}-{\bf \Pi}_h{\Bq},\overline{\nabla(\Psi-\pi_h^p\Psi)})_{\mathcal T_h}\\
&+\sum_{e\in\mathcal E_h}\langle\{\hspace{-0.1cm}\{{\Bq}-{\bf
\Pi}_h{\Bq}\}\hspace{-0.1cm}\},\llbracket\overline{\Psi-\Pi_h\Psi}\rrbracket\rangle_e
+\sum_{e\in\mathcal E_h^0}\frac{\tau}{2}\langle\llbracket u-\Pi_hu\rrbracket,\llbracket\overline{\Psi-\Pi_h\Psi}\rrbracket\rangle_e\\
&+\frac{\tau}{1+\tau}\langle
u-\Pi_hu,\overline{\Psi-\Pi_h\Psi}\rangle_{\partial \Omega}-
\frac{\tau}{1+\tau}\langle({\Bq}-{\bf
\Pi}_h{\Bq})\cdot{\Bn},\overline{(\Psi-\Pi_h\Psi)}\rangle_{\partial
\Omega}.
\end{align*}
Hence
\begin{align*}
|T_2|&\leq \kappa\|{\Bq}-{\bf \Pi}_h{\Bq}\|_{0,\Omega}\|{\bf \Pi}_h{\bf \Phi}-{\bf \Phi}\|_{0,\Omega}+|u-\pi_h^pu|_{1,\Omega_h}\|{\bf \Pi}_h{\bf \Phi}-{\bf \Phi}\|_{0,\Omega}\\
&+\Big(\sum_{e\in \mathcal
E_h}\|u-\Pi_hu\|_{0,e}^2\Big)^{\frac{1}{2}} \cdot\Big(\sum_{e\in
\mathcal E_h}\|{\bf \Pi}_h{\bf \Phi}-{\bf
\Phi}\|_{0,e}^2\Big)^{\frac{1}{2}}+\|{\Bq}-{\bf
\Pi}_h{\Bq}\|_{0,\Omega}|\Psi-\pi_h^p\Psi|_{1,\Omega_h}\\&+\tau^{-1}
(\sum_{e\in \mathcal E_h}\|{\Bq}-{\bf
\Pi}_h{\Bq}\|_{0,e}^2)^{\frac{1}{2}}
\cdot(\sum_{e\in \mathcal E_h}\|{\bf \Pi}_h{\bf \Phi}-{\bf \Phi}\|_{0,e}^2)^{\frac{1}{2}}\\
&+\Big(\sum_{e\in \mathcal E_h}\|{\Bq}-{\bf
\Pi}_h{\Bq}\|_{0,e}^2\Big)^{\frac{1}{2}} \cdot\Big(\sum_{e\in
\mathcal
E_h}\|\Psi-\Pi_h\Psi\|_{0,e}^2\Big)^{\frac{1}{2}}+\kappa\|u-\Pi_hu\|_{0,\Omega}\|\Psi-\Pi_h\Psi\|_{0,\Omega}
\\&+
\tau\Big(\sum_{e\in \mathcal
E_h}\|u-\Pi_hu\|_{0,e}^2\Big)^{\frac{1}{2}} \cdot\Big(\sum_{e\in
\mathcal E_h}\|\Psi-\Pi_h\Psi\|_{0,e}^2\Big)^{\frac{1}{2}}.
\end{align*}
Using Lemma \ref{lemma34}, Theorem \ref{theorem41} and (\ref{Aregularity}) again, we deduce
\begin{align}
\label{T_2} |T_2|\lesssim \frac{\kappa h^2}{p^2}M(\tilde f,\tilde
g)\|e_u\|_{0,\Omega}.
\end{align}
Taking (\ref{T_1}) and (\ref{T_2}) into (\ref{e_u}), the desired result (\ref{U_h}) is obtained.
\end{proof}

Finally we give the error estimate of the trace flux $\lambda_h$.
\begin{thm}
Let $\lambda_h$ be the solution of (\ref{AP1})-(\ref{AP4}). Then
\begin{align}
\label{Lam} \|u-\lambda_h\|_{0,\partial \mathcal T_h}\lesssim
\frac{\kappa h^{\frac{3}{2}}}{p}M(\tilde f,\tilde g).
\end{align}
\end{thm}
\begin{proof}
For $e\in \partial \mathcal T_h\backslash \partial \Omega$,
\begin{align*}
\lambda_h&=\{\hspace{-0.1cm}\{U_h\}\hspace{-0.1cm}\}+\frac{1}{2\tau}\llbracket{\BQ}_h\rrbracket.
\end{align*}
Note that ${\Bq}\cdot{\Bn}$ is continuous across inner edges(faces).
By Lemmas \ref{lemma33}-\ref{lemma34} we have
\begin{align*}
&\|u-\lambda_h\|_{0,\partial \mathcal T_h\backslash\partial \Omega}\leq \|u-U_h\|_{0,\partial \mathcal T_h\backslash\partial \Omega}+\frac{1}{2\tau}\|\llbracket{\Bq}-{\BQ}_h\rrbracket\|_{0,\partial \mathcal T_h\backslash\partial \Omega}\\
&\lesssim \|u-\Pi_hu\|_{0,\partial \mathcal T_h}+\|\Pi_hu-U_h\|_{0,\partial \mathcal T_h}+\tau^{-1}(\|{\Bq}-{\bf \Pi}_h{\Bq}\|_{0,\partial \mathcal T_h}+\|{\bf \Pi}_h{\Bq}-{\BQ}_h\|_{0,\partial \mathcal T_h})\\
&\lesssim \|u-\Pi_hu\|_{0,\partial \mathcal
T_h}+ph^{-\frac{1}{2}}(\|u-\Pi_hu\|_{0,\Omega}+\|u-U_h\|_{0,\Omega})
+\tau^{-1}\|{\Bq}-{\bf \Pi}_h{\Bq}\|_{0,\partial \mathcal T_h}\\
&+\tau^{-1}ph^{-\frac{1}{2}}(\|{\Bq}-{\bf
\Pi}_h{\Bq}\|_{0,\Omega}+\|{\Bq}-{\BQ}_h\|_{0,\Omega})\lesssim
\frac{\kappa h^{\frac{3}{2}}}{p}M(\tilde f,\tilde g),
\end{align*}
where we utilize Theorems \ref{theorem41}-\ref{theorem42} in the last inequality.

For $e\in \partial \mathcal T_h\cap \partial \Omega$,
\begin{align*}
u-\lambda_h=\frac{\tau(u-U_h)}{1+\tau}+\frac{({\Bq}-{\BQ}_h)\cdot{\Bn}}{1+\tau}+\frac{g-P_Mg}{1+\tau}.
\end{align*}
Since
\begin{align*}
\|u-\lambda_h\|_{0,\partial\Omega}^2=\|u-\Pi_hu\|_{0,\partial\Omega}^2-
\|\lambda_h-\Pi_hu\|_{0,\partial\Omega}^2- 2{\rm Re}\langle
u-\lambda_h, \overline{\lambda_h-\Pi_hu}\rangle_{\partial \Omega},
\end{align*}
the definition of $P_M$ implies that
\begin{align*}
\|u-\lambda_h\|_{0,\partial\Omega}^2&\leq
\|u-\Pi_hu\|_{0,\partial\Omega}^2-
\|\lambda_h-\Pi_hu\|_{0,\partial\Omega}^2+
\frac{2\tau}{1+\tau}\|u-U_h\|_{0,\partial\Omega}\|\lambda_h-\Pi_hu\|_{0,\partial\Omega}\\&+
\frac{2}{1+\tau}\|{\Bq}-{\BQ}_h\|_{0,\partial\Omega}\|\lambda_h-\Pi_hu\|_{0,\partial\Omega}\\
&\lesssim
\|u-\Pi_hu\|_{0,\partial\Omega}^2+\|u-U_h\|_{0,\partial\Omega}^2+
{\tau}^{-2}\|{\Bq}-{\BQ}_h\|_{0,\partial\Omega}^2.
\end{align*}
Similar to the proof for the case of inner edges, we can also show that
\begin{align*}
\|u-\lambda_h\|_{0,\partial\Omega}\lesssim \frac{\kappa
h^{\frac{3}{2}}}{p}M(\tilde f,\tilde g).
\end{align*}
Hence (\ref{Lam}) is derived.
\end{proof}

\section{The error estimates}
In this section, we shall derive error estimates for the scheme
(\ref{P1})-(\ref{P4}). This will be done by making use of the
stability estimates derived in Theorem \ref{theorem31} and the error
estimates of the auxiliary problem established in the previous
section.
\begin{thm}
Let ${{\Bq}}_h$, $u_h$ and $\hat u_h$ be the solution of
(\ref{P1})-(\ref{P4}). We have
\begin{align}
\|u-u_h\|_{0,\Omega}&\lesssim \Big(\frac{\kappa h^2}{p^2}+\frac{\kappa^2h^2}{p^2}+\frac{\kappa^5h^4}{p^4}\Big)M(\tilde f,\tilde g),\\
\kappa\|{\Bq}-{\Bq}_h\|_{0,\Omega}&\lesssim \Big(\frac{\kappa h}{p}+\frac{\kappa^3h^2}{p^2}+\frac{\kappa^6h^4}{p^4}\Big)M(\tilde f,\tilde g),\\
\|u-\hat u_h\|_{0,\partial \mathcal T_h}&\lesssim \Big(\frac{\kappa
h^{\frac{3}{2}}}{p} +\big(\frac{\kappa
h}{p}\big)^{\frac{5}{2}}+\big(\frac{\kappa
h}{p}\big)^{\frac{5}{2}}\frac{\kappa^3h^2}{p^2}
+\frac{\kappa^2h^{\frac{3}{2}}}{p}+\frac{\kappa^5h^{\frac{7}{2}}}{p^3}\Big)M(\tilde
f,\tilde g).
\end{align}
Moreover, if $\frac{\kappa^3h^2}{p^2}\lesssim 1$, then
\begin{align}
\|u-u_h\|_{0,\Omega}&\lesssim \Big(\frac{\kappa h^2}{p^2}+\frac{\kappa^2h^2}{p^2}\Big)M(\tilde f,\tilde g),\\
\kappa\|{\Bq}-{\Bq}_h\|_{0,\Omega}&\lesssim \Big(\frac{\kappa h }{p}+\frac{\kappa^3h^2}{p^2}\Big)M(\tilde f,\tilde g),\\
\|u-\hat u_h\|_{0,\partial \mathcal T_h}&\lesssim \Big(\frac{\kappa
h^{\frac{3}{2}}}{p}+ \frac{\kappa^2h^{\frac{3}{2}}}{p} +\frac{\kappa
h^{\frac{3}{2}}}{p^{\frac{3}{2}}} \Big)M(\tilde f,\tilde g).
\end{align}
\end{thm}
\begin{proof}
Denote ${\bf \epsilon}_h^{\Bq}:={\Bq}_h-{\BQ}_h$,
$\epsilon_h^u:=u_h-U_h$, and $\epsilon_h^{\hat u}:=\hat
u_h-\lambda_h$, according to the formulation (\ref{P1})-(\ref{P4})
and (\ref{AP1})-(\ref{AP4}), we have $({\bf
\epsilon}_h^{\Bq},\epsilon_h^u,\epsilon_h^{\hat u})\in
{\BV}_h^p\times W_h^p\times M_h^p$ and they satisfy
\begin{align*}
({\bf i}\kappa {\bf \epsilon}_h^{\Bq},\overline{\Br})_{\mathcal T_h}-(\epsilon_h^u, \overline{{\rm div \, {\Br}}})_{\mathcal T_h}+\langle \epsilon_h^{\hat u},\overline {{\Br}\cdot {\Bn}}\rangle_{\partial \mathcal T_h}&=0,\\
({\bf i}\kappa \epsilon_h^u,\overline w)_{\mathcal T_h}-({\bf \epsilon}_h^{\Bq},\overline{\nabla w})_{\mathcal T_h}+\langle \hat {\Bq}_h\cdot {\Bn} ,\overline w\rangle_{\partial \mathcal T_h}&=2{\bf i}\kappa(U_h-u,\overline w)_{\mathcal T_h} ,\\
\langle -\hat{\bf \epsilon}_h^{\Bq}\cdot {\Bn}+ \epsilon_h^{\hat u},\overline \mu\rangle_{\partial \Omega}&=0,\\
\langle \hat{\bf \epsilon}_h^{\Bq}\cdot {\Bn},\overline
\mu\rangle_{\partial \mathcal T_h \backslash
 \partial \Omega}&=0,
\end{align*}
for all ${\Br}\in {\BV}_h^p$, $w\in W_h^p$, and $\mu\in M_h^p$. The
numerical flux $\hat {\bf \epsilon}_h^{\Bq}$ is given by
\begin{align}
\hat {\bf \epsilon}_h^{\Bq}={\bf \epsilon}_h^{\Bq}+\tau
(\epsilon_h^u-\epsilon_h^{\hat u}){\Bn} \qquad {\rm on } \
\partial \mathcal T_h.
\end{align}
The stability estimates in Theorem \ref{theorem31} imply that
\begin{align*}
\|\epsilon_h^u\|_{0,\Omega}&\lesssim \Big(1+\frac{\kappa^3h^2}{p^2}\Big)\kappa\|u-U_h\|_{0,\Omega},\\
\|{\bf \epsilon}_h^{\Bq}\|_{0,\Omega}&\lesssim \Big(1+\frac{\kappa^3h^2}{p^2}\Big)\kappa\|u-U_h\|_{0,\Omega},\\
\|\epsilon_h^{\hat u}\|_{0,\partial \mathcal T_h}&\lesssim
\Big(\big(\frac{\kappa
h}{p}\big)^{\frac{1}{2}}+ph^{-\frac{1}{2}}\Big)\Big(1+\frac{\kappa^3h^2}{p^2}\Big)\kappa\|u-U_h\|_{0,\Omega}.
\end{align*}
Using the triangle inequality and Theorem \ref{theorem42}, the proof is finished.
\end{proof}
Next we demonstrate the improved convergence results for the coarse meshes under the condition $\frac{\kappa^3h^2}{p^2}\gtrsim 1$. First we give the stability estimate for the following elliptic HDG scheme.
\begin{lem}\label{lemma61}
Let $(\mathcal Q_h$, $\mathcal U_h$, $\hat {\mathcal U}_h)\in {\BV}_h^p\times
W_h^p\times M_h^p$ be the solution of the following elliptic HDG scheme
\begin{align*}
({\bf i}\kappa \mathcal Q_h,\overline{\Br})_{\mathcal T_h}-(\mathcal U_h, \overline{{\rm div \, {\Br}}})_{\mathcal T_h}+\langle \hat {\mathcal U}_h,\overline {{\Br}\cdot {\Bn}}\rangle_{\partial \mathcal T_h}&=0,\\
-({\bf i}\kappa \mathcal U_h,\overline w)_{\mathcal T_h}-(\mathcal Q_h,\overline{\nabla w})_{\mathcal T_h}+\langle \hat {\mathcal Q}_h\cdot {\Bn} ,\overline w\rangle_{\partial \mathcal T_h}&=(f,\overline w)_{\mathcal T_h} ,\\
\langle -\hat{\mathcal Q}_h\cdot {\Bn}+ \hat {\mathcal U}_h,\overline \mu\rangle_{\partial \Omega}&=0,\\
\langle \hat{\mathcal Q}_h\cdot {\Bn},\overline
\mu\rangle_{\partial \mathcal T_h \backslash
 \partial \Omega}&=0,
\end{align*}
for all ${\Br}\in {\BV}_h^p$, $w\in W_h^p$, and $\mu\in M_h^p$,
where the numerical flux
$\hat {\mathcal Q}_h$ is given by
\begin{align*}\label{numerical-flux}
\hat {\mathcal Q}_h=\mathcal Q_h+\tau (\mathcal U_h-\hat {\mathcal U}_h){\Bn} \qquad {\rm on } \
\partial \mathcal T_h.
\end{align*}
Then there holds
\begin{align}
\|\mathcal Q_h\|_{0,\Omega}+\|{\mathcal{U}}_h\|_{0,\Omega}\lesssim \kappa^{-1}\|f\|_{0,\Omega}.
\end{align}
\end{lem}
\begin{proof}
Similar to the proof of Lemma \ref{lemma35}, we can deduce
\begin{align*}
({\bf i}\kappa \mathcal U_h,\overline {\mathcal U_h})_{\mathcal T_h}+({\bf i}\kappa
\mathcal Q_h,\overline{\mathcal Q_h})_{\mathcal T_h}+\langle \tau(\mathcal U_h-\hat {\mathcal U}_h),
\overline{(\mathcal U_h-\hat {\mathcal U}_h)}\rangle_{\partial \mathcal T_h}
+\langle\hat {\mathcal U}_h,\overline{\hat {\mathcal U}_h}\rangle_{\partial \Omega}=
(f,\overline {\mathcal U_h})_{\mathcal T_h}.
\end{align*}
Hence
\begin{align*}
\kappa \|{\mathcal{U}}_h\|_{0,\Omega}^2&\leq \|f\|_{0,\Omega}\|{\mathcal{U}}_h\|_{0,\Omega},\\
\kappa \|{\mathcal{Q}}_h\|_{0,\Omega}^2&\leq \|f\|_{0,\Omega}\|{\mathcal{U}}_h\|_{0,\Omega},
\end{align*}
which means
\begin{align*}
\kappa (\|{\mathcal{U}}_h\|_{0,\Omega}+\|{\mathcal{Q}}_h\|_{0,\Omega})\leq \|f\|_{0,\Omega}.
\end{align*}
\end{proof}

\begin{thm}\label{theorem61}
Let ${{\Bq}}_h$, $u_h$ and $\hat u_h$ be the solution of
(\ref{P1})-(\ref{P4}). Under the mesh condition
$\frac{\kappa^3h^2}{p^2}\gtrsim 1$, we have
\begin{align}
\|u-u_h\|_{0,\Omega}&\lesssim \frac{\kappa^2h^2}{p^2}M(\tilde f,\tilde g),\\
\kappa\|{\Bq}-{\Bq}_h\|_{0,\Omega}&\lesssim \Big(\frac{\kappa h }{p}+\frac{\kappa^3h^2}{p^2}\Big)M(\tilde f,\tilde g),\\
\|u-\hat u_h\|_{0,\partial \mathcal T_h}&\lesssim \Big(\frac{\kappa^2h^{\frac{3}{2}}}{p} +\frac{\kappa
h^{\frac{3}{2}}}{p^{\frac{3}{2}}} \Big)M(\tilde f,\tilde g).
\end{align}
\end{thm}
\begin{proof}
In this mesh condition, the stability estimates in Theorem
\ref{theorem31} and Lemma \ref{lemma35} indicate the following
inequality
\begin{align}
\label{diff}
\|u_h-\hat u_h\|_{0,\partial \mathcal T_h}\lesssim \frac{kh^{\frac{3}{2}}}{p^{\frac{3}{2}}}M(\tilde f,\tilde g).
\end{align}
The consistence of the HDG scheme implies that
\begin{align}
\label{EP1}
({\bf i}\kappa (\Bq-{\Bq}_h),\overline{\Br})_{\mathcal T_h}-(u-u_h, \overline{{\rm div \, {\Br}}})_{\mathcal T_h}+\langle u-\hat u_h,\overline {{\Br}\cdot {\Bn}}\rangle_{\partial \mathcal T_h}&=0,\\
\label{EP2}
({\bf i}\kappa (u-u_h),\overline w)_{\mathcal T_h}+({\rm div\ }(\Bq-{\Bq}_h),\overline{ w})_{\mathcal T_h}-\tau \langle u_h-\hat u_h ,\overline w\rangle_{\partial \mathcal T_h}&=0,\\
\label{EP3}
\langle -(\Bq-\hat{\Bq}_h)\cdot {\Bn}+ (u-\hat u_h),\overline \mu\rangle_{\partial \Omega}&=0,\\
\label{EP4} \langle (\Bq-\hat{\Bq}_h)\cdot {\Bn},\overline
\mu\rangle_{\partial \mathcal T_h \backslash
 \partial \Omega}&=0,
\end{align}
for all ${\Br}\in {\BV}_h^p$, $w\in W_h^p$, and $\mu\in M_h^p$. We
introduce the dual problem which replaces the right hand side of
(\ref{D1})-(\ref{D3}) by $u-u_h$ as follows:
\begin{align}
 \label{ED1}
 -{\bf i}\kappa{\bf \Phi}+\nabla \Psi&=0\qquad \ {\rm in\ }\Omega,\\
  \label{ED2}
 {\rm div\, }{\bf \Phi}-{\bf i}\kappa\Psi&=u-u_h \qquad {\rm in\ }\Omega,\\
  \label{ED3}
 {\bf \Phi}\cdot {\Bn}&=\Psi \qquad \ {\rm on\ }\partial \Omega.
\end{align}
Similar to Lemma \ref{lemma46}, we have the following regularity
estimate:
\begin{align}
\label{ERegularity}
\kappa^{-2}\|\Psi\|_{2,\Omega}+\kappa^{-1}\|\Psi\|_{1,\Omega}
+\kappa^{-1}\|{\bf \Phi}\|_{1,\Omega}\lesssim \|u-u_h\|_{0,\Omega}.
\end{align}
We denote by $P_M$ the $L^2$ projection onto $M_h^p$,
$$\langle P_M \Psi, \mu\rangle_{\partial \mathcal T_h}=\langle  \Psi, \mu\rangle_{\partial \mathcal T_h}\quad {\rm for \ all} \  \mu \in M_h^p.$$
From (\ref{EP1}) and (\ref{ED2}), we can easily get
\begin{align*}
 &(u-u_h,\overline {u-u_h})_{\mathcal T_h}=(u-u_h, \overline{{\rm div \, }{\bf \Phi}-{\bf i}\kappa\Psi})_{\mathcal T_h}=(u-u_h, \overline{{\rm div \, }{\bf \Phi}})_{\mathcal T_h}+{\bf i}\kappa(u-u_h, \overline \Psi)_{\mathcal T_h}\\
 &=(u-u_h, \overline{{\rm div \, }({\bf \Phi}-{\bf \Pi}_h{\bf \Phi})})_{\mathcal T_h}+({\bf i}\kappa(\Bq-{\Bq}_h),\overline{{\bf \Pi}_h{\bf \Phi}})_{\mathcal T_h}+\langle u-\hat u_h,\overline {{\bf \Pi}_h{\bf \Phi}\cdot {\Bn}}\rangle_{\partial \mathcal T_h}+{\bf i}\kappa(u-u_h, \overline \Psi)_{\mathcal T_h}.
  \end{align*}
According to (\ref{ED1}) and Green formulation, there holds
\begin{align*}
({\bf i}\kappa(\Bq-{\Bq}_h),\overline{{\bf \Phi}})_{\mathcal
T_h}=-(\Bq-{\Bq}_h,\overline{\nabla \Psi})_{\mathcal T_h} =({\rm div
\ }(\Bq-{\Bq}_h),\overline{\Psi})_{\mathcal T_h}-\langle
(\Bq-{\Bq}_h)\cdot {\Bn},\overline{\Psi}\rangle_{\partial \mathcal
T_h}.
\end{align*}
Hence, by (\ref{EP2}) we obtain
\begin{align}
\label{ex}
 (u-u_h,\overline {u-u_h})_{\mathcal T_h}=A_1+A_2,
 \end{align}
where
\begin{align*}
A_1&=(u-u_h, \overline{{\rm div \, }({\bf \Phi}-{\bf \Pi}_h{\bf \Phi})})_{\mathcal T_h}-({\bf i}\kappa(\Bq-{\Bq}_h),\overline{{\bf \Phi}-{\bf \Pi}_h{\bf \Phi}})_{\mathcal T_h}-\langle (\Bq-{\Bq}_h)\cdot {\Bn},\overline{\Psi-P_M\Psi}\rangle_{\partial \mathcal T_h}\\
&+({\bf i}\kappa(u-u_h)+{\rm div\ }(\Bq-{\Bq}_h),,\overline{\Psi-\Pi_h\Psi})_{\mathcal T_h}
-\tau\langle u-u_h,\overline{\Pi_h\Psi-P_M\Psi}\rangle_{\partial \mathcal T_h},
 \end{align*}
and
\begin{align*}
A_2&=\tau\langle u_h-\hat u_h,\overline{\Pi_h\Psi}\rangle_{\partial \mathcal T_h}-\langle (\Bq-{\Bq}_h)\cdot {\Bn},\overline{\Psi}\rangle_{\partial \mathcal T_h}+\langle u-\hat u_h,\overline{{\bf \Pi}_h{\bf \Phi}\cdot{\Bn}}\rangle_{\partial \mathcal T_h}\\
&+\langle (\Bq-{\Bq}_h)\cdot {\Bn},\overline{\Psi-P_M\Psi}\rangle_{\partial \mathcal T_h}
+\tau\langle u-u_h,\overline{\Pi_h\Psi-P_M\Psi}\rangle_{\partial \mathcal T_h}.
\end{align*}
Utilizing  (\ref{EP3}), (\ref{EP4}) and (\ref{ED3}), $A_2$ becomes
\begin{align*}
A_2&=-\langle (\Bq-{\Bq}_h)\cdot {\Bn},\overline{P_M\Psi}\rangle_{\partial \mathcal T_h}+\tau\langle u_h-\hat u_h,\overline{P_M\Psi}\rangle_{\partial \mathcal T_h}+\tau\langle u_h-\hat u_h,\overline{\Pi_h\Psi-P_M\Psi}\rangle_{\partial \mathcal T_h}\\
&+\langle u-\hat u_h,\overline{({\bf \Pi}_h{\bf \Phi}-{\bf \Phi})\cdot{\Bn}}\rangle_{\partial \mathcal T_h}
+\langle u-\hat u_h,\overline{{\bf \Phi}\cdot{\Bn}}\rangle_{\partial \Omega}
+\tau\langle u- u_h,\overline{\Pi_h\Psi-P_M\Psi}\rangle_{\partial \mathcal T_h}\\
&=-\langle u-\hat u_h,\overline{P_M\Psi}\rangle_{\partial \Omega}
+\tau\langle u-\hat u_h,\overline{\Pi_h\Psi-P_M\Psi}\rangle_{\partial \mathcal T_h}\\
&+\langle u-\hat u_h,\overline{({\bf \Pi}_h{\bf \Phi}-{\bf \Phi})\cdot{\Bn}}\rangle_{\partial \mathcal T_h}
+\langle u-\hat u_h,\overline{{\bf \Phi}\cdot{\Bn}}\rangle_{\partial \Omega}\\
&=\langle u-\hat u_h,\overline{\Psi-P_M\Psi}\rangle_{\partial \Omega}
+\tau\langle u-\hat u_h,\overline{\Pi_h\Psi-P_M\Psi}\rangle_{\partial \mathcal T_h}
+\langle u-\hat u_h,\overline{({\bf \Pi}_h{\bf \Phi}-{\bf \Phi})\cdot{\Bn}}\rangle_{\partial \mathcal T_h}.
\end{align*}
Taking the complex conjugation of (\ref{ex}) and making use of Green formulation and the definition of projection operator, we obtain
\begin{align*}
\|u-u_h\|_{0,\Omega}^2&={\bf i}\kappa({\bf \Phi}-{\bf \Pi}_h{\bf \Phi},\overline{{\Bq}-{\Bq}_h})_{\mathcal T_h}
+(\Psi-\Pi_h\Psi,\overline{{\rm div\ }(\Bq-{\Bq}_h)})_{\mathcal T_h}
-\langle \Psi-P_M\Psi,\overline{(\Bq-{\Bq}_h)\cdot\Bn}\rangle_{\partial \mathcal T_h}\\
&-{\bf i}\kappa(\Psi-\Pi_h\Psi,\overline{u-u_h})_{\mathcal T_h}
-({\bf \Phi}-{\bf \Pi}_h{\bf \Phi},\overline{\nabla(u-u_h)})_{\mathcal T_h}
+\langle ({\bf \Phi}-{\bf \Pi}_h{\bf \Phi})\cdot\Bn, u-u_h\rangle_{\partial \mathcal T_h}\\
&-\tau \langle \Pi_h\Psi-P_M\Psi,u-u_h\rangle_{\partial \mathcal T_h}
+\langle \Psi-P_M\Psi,\overline{u-\hat u_h}\rangle_{\partial \Omega}\\
&-\langle ({\bf \Phi}-{\bf \Pi}_h{\bf \Phi})\cdot\Bn, u-\hat u_h\rangle_{\partial \mathcal T_h}
+\tau \langle \Pi_h\Psi-P_M\Psi,u-\hat u_h\rangle_{\partial \mathcal T_h}\\
&={\bf i}\kappa({\bf \Phi}-{\bf \Pi}_h{\bf \Phi},\overline{{\Bq}-{\bf \Pi}_h{\Bq}})_{\mathcal T_h}
+(\Psi-\Pi_h\Psi,\overline{{\rm div\ }\Bq})_{\mathcal T_h}
-\langle \Psi-P_M\Psi,\overline{(\Bq-{\bf \Pi}_h{\Bq})\cdot\Bn}\rangle_{\partial \mathcal T_h}\\
&-{\bf i}\kappa(\Psi-\Pi_h\Psi,\overline{u-\Pi_hu})_{\mathcal T_h}
-({\bf \Phi}-{\bf \Pi}_h{\bf \Phi},\overline{\nabla(u-\pi_h^pu)})_{\mathcal T_h}
-\langle ({\bf \Phi}-{\bf \Pi}_h{\bf \Phi})\cdot\Bn, u_h-\hat u_h\rangle_{\partial \mathcal T_h}\\
&+\tau \langle \Pi_h\Psi-P_M\Psi,u_h-\hat u_h\rangle_{\partial \mathcal T_h}
+\langle \Psi-P_M\Psi,\overline{u-\Pi_h u}\rangle_{\partial \Omega}.
\end{align*}
Note that
\begin{align*}
\|\Pi_h\Psi-P_M\Psi\|_{0,\partial \mathcal T_h}\leq
\|\Psi-\Pi_h\Psi\|_{0,\partial \mathcal
T_h}+\|\Psi-P_M\Psi\|_{0,\partial \mathcal T_h} \lesssim
\|\Psi-\Pi_h\Psi\|_{0,\partial \mathcal T_h}.
\end{align*}
Using Lemma \ref{lemma34}, the regularity estimates
(\ref{regularity}) and (\ref{ERegularity}), the following inequality
is derived
\begin{align*}
\|u-u_h\|_{0,\Omega}^2\lesssim \frac{\kappa^2h^2}{p^2}\|u-u_h\|_{0,\Omega}M(\tilde f,\tilde g),
\end{align*}
which means
\begin{align*}
\|u-u_h\|_{0,\Omega}\lesssim \frac{\kappa^2h^2}{p^2}M(\tilde f,\tilde g).
\end{align*}
Since
$({\bf
\epsilon}_h^{\Bq},\epsilon_h^u,\epsilon_h^{\hat u})\in
{\BV}_h^p\times W_h^p\times M_h^p$  satisfy
\begin{align*}
({\bf i}\kappa {\bf \epsilon}_h^{\Bq},\overline{\Br})_{\mathcal T_h}-(\epsilon_h^u, \overline{{\rm div \, {\Br}}})_{\mathcal T_h}+\langle \epsilon_h^{\hat u},\overline {{\Br}\cdot {\Bn}}\rangle_{\partial \mathcal T_h}&=0,\\
-({\bf i}\kappa \epsilon_h^u,\overline w)_{\mathcal T_h}-({\bf \epsilon}_h^{\Bq},\overline{\nabla w})_{\mathcal T_h}+\langle \hat {\Bq}_h\cdot {\Bn} ,\overline w\rangle_{\partial \mathcal T_h}&=2{\bf i}\kappa(u-u_h,\overline w)_{\mathcal T_h} ,\\
\langle -\hat{\bf \epsilon}_h^{\Bq}\cdot {\Bn}+ \epsilon_h^{\hat u},\overline \mu\rangle_{\partial \Omega}&=0,\\
\langle \hat{\bf \epsilon}_h^{\Bq}\cdot {\Bn},\overline
\mu\rangle_{\partial \mathcal T_h \backslash
 \partial \Omega}&=0,
\end{align*}
for all ${\Br}\in {\BV}_h^p$, $w\in W_h^p$, and $\mu\in M_h^p$, it
follows from the stability estimate in Lemma \ref{lemma61} that
\begin{align*}
\kappa \|\Bq-{\Bq}_h\|_{0,\Omega}\leq \kappa \|\Bq-{\BQ}_h\|_{0,\Omega}+\kappa \|{\bf
\epsilon}_h^{\Bq}\|_{0,\Omega}\lesssim \big(\frac{\kappa h}{p}+\frac{\kappa^3 h^2}{p^2}\big)M(\tilde f,\tilde g).
\end{align*}
The triangle inequality and (\ref{diff}) imply that
\begin{align*}
\|u-\hat u_h\|_{0,\partial \mathcal T_h}&\leq \|u- \Pi_hu\|_{0,\partial \mathcal T_h}+\|u_h- \Pi_hu\|_{0,\partial \mathcal T_h}+\|u_h-\hat u_h\|_{0,\partial \mathcal T_h}\\
&\lesssim ph^{\frac{-1}{2}}(\|u-\Pi_hu\|_{0,\Omega}+\|u-u_h\|_{0,\Omega})+\kappa \big(\frac{h}{p}\big)^{\frac{3}{2}}M(\tilde f,\tilde g)\\
&\lesssim \big(\frac{\kappa^2h^{\frac{3}{2}}}{p}+\frac{\kappa h^{\frac{3}{2}}}{p^{\frac{3}{2}}}\big)M(\tilde f,\tilde g).
\end{align*}
The proof is completed.
\end{proof}

\section{Numerical results}
In this section, we present a detailed documentation of numerical
results of the HDG method for the following 2-d Helmholtz problem:
\begin{eqnarray}
&-\Delta u - \kappa^2 u = f := \frac{\sin{\kappa r}}{r} \qquad &{\rm
in} \ \Omega, \\
&\frac{\partial u}{\partial n} + \textbf{i} \kappa u = g \qquad
&{\rm on} \ \Gamma_R := \partial \Omega.
\end{eqnarray}
Here $\Omega$ is unit square $[-0.5,0.5]\times[-0.5,0.5]$, and $g$
is chosen such that the exact solution is given by
\begin{eqnarray}
u = \frac{\cos{\kappa r}}{\kappa} - \frac{\cos{\kappa} +
\textbf{i}\sin{\kappa}}{\kappa ( J_0(\kappa) + \textbf{i}J_1(\kappa)
)} J_0(\kappa r)
\end{eqnarray}
in polar coordinates, where $J_\nu (z)$ are Bessel functions of the
first kind.

\smallskip

In the numerical results of \cite{GM2011}, the optimal convergence
of the HDG method is observed when the parameter $\tau$ is chosen as
$O(1)$. In this work, when $u \in H^2(\Omega)$ we let $\tau =
\frac{p}{\kappa h}$, which is also used in the following experiment.
The HDG method is implemented for piecewise linear (HDG-P1),
piecewise quadratic (HDG-P2) and piecewise cubic (HDG-P3) finite
element spaces.

\smallskip

For the fixed wave number $\kappa$, we first show the dependence of
the convergence of $\|u-u_h\|_{0,\Omega}$, $\|\Bq -
\Bq_h\|_{0,\Omega}$ and $\|u-\hat{u}_h\|_{0,\partial \Ct_h}$ on
polynomial order $p$ and mesh size $h$. On one hand, the left graphs
of Figure \ref{fig1} display the above three kinds of errors for
$\kappa = 100$ by HDG-P1, HDG-P2 and HDG-P3 approximations. We find
that the pollution errors always appear on the coarse meshes, but
the errors of $\|u-u_h\|_{0,\Omega}$ almost converges in $O(\kappa
h^2/p^2)$ on the fine meshes, and $\|u-\hat{u}_h\|_{0,\partial
\Ct_h}$ nearly converges in $O(\kappa h^{\frac{3}{2}}/p)$ on the
fine meshes. The results support the theoretical analysis. We note
that the error of $\|\Bq - \Bq_h\|_{0,\Omega}$ also almost converges
in $O(\kappa h^2/p^2)$ on the fine meshes, which is a little better
than our theoretical prediction. On the other hand, for the case of
$\kappa = 300$, the right graphs of Figure \ref{fig1} show that the
errors of $\|u-\hat{u}_h\|_{0,\partial \Ct_h}$,
$\|u-u_h\|_{0,\Omega}$ and $\|\Bq - \Bq_h\|_{0,\Omega}$ always
decrease for high order polynomial approximations.

Figure \ref{fig2} displays the surface plots of the imaginary parts
of the HDG-P1, HDG-P2, HDG-P3 solutions of $u_h$ and the exact
solution for $\kappa = 100$ with mesh size $h \approx 0.022$. It is
shown that the HDG-P2 and HDG-P3 solutions have correct shapes and
amplitudes as the exact solution, while the HDG-P1 solution has a
correct shape but its amplitude is not very accurate near the center
of the domain.

\begin{figure}[htbp]
\centering
    \includegraphics[width=2.7in]{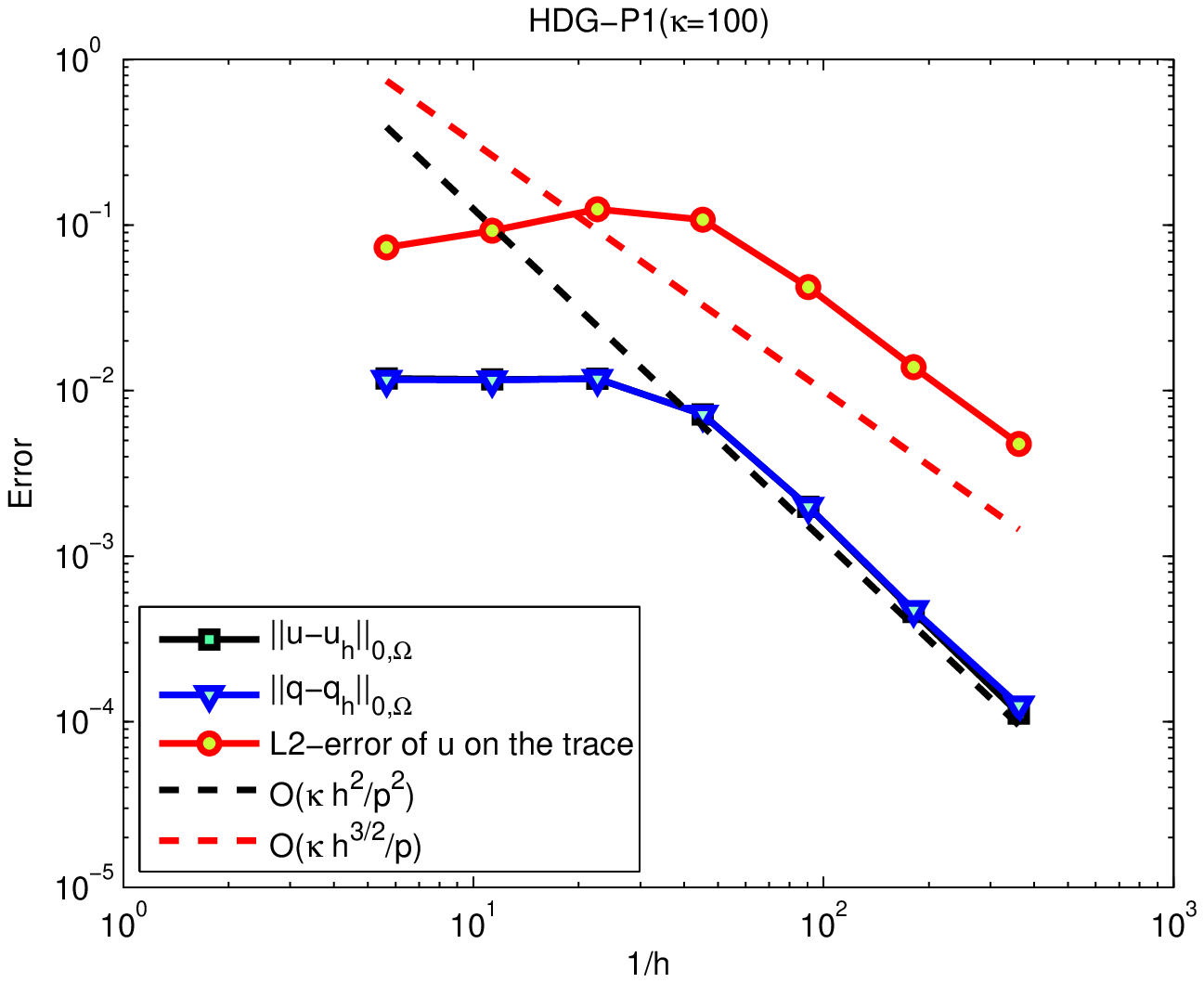}
    \includegraphics[width=2.7in]{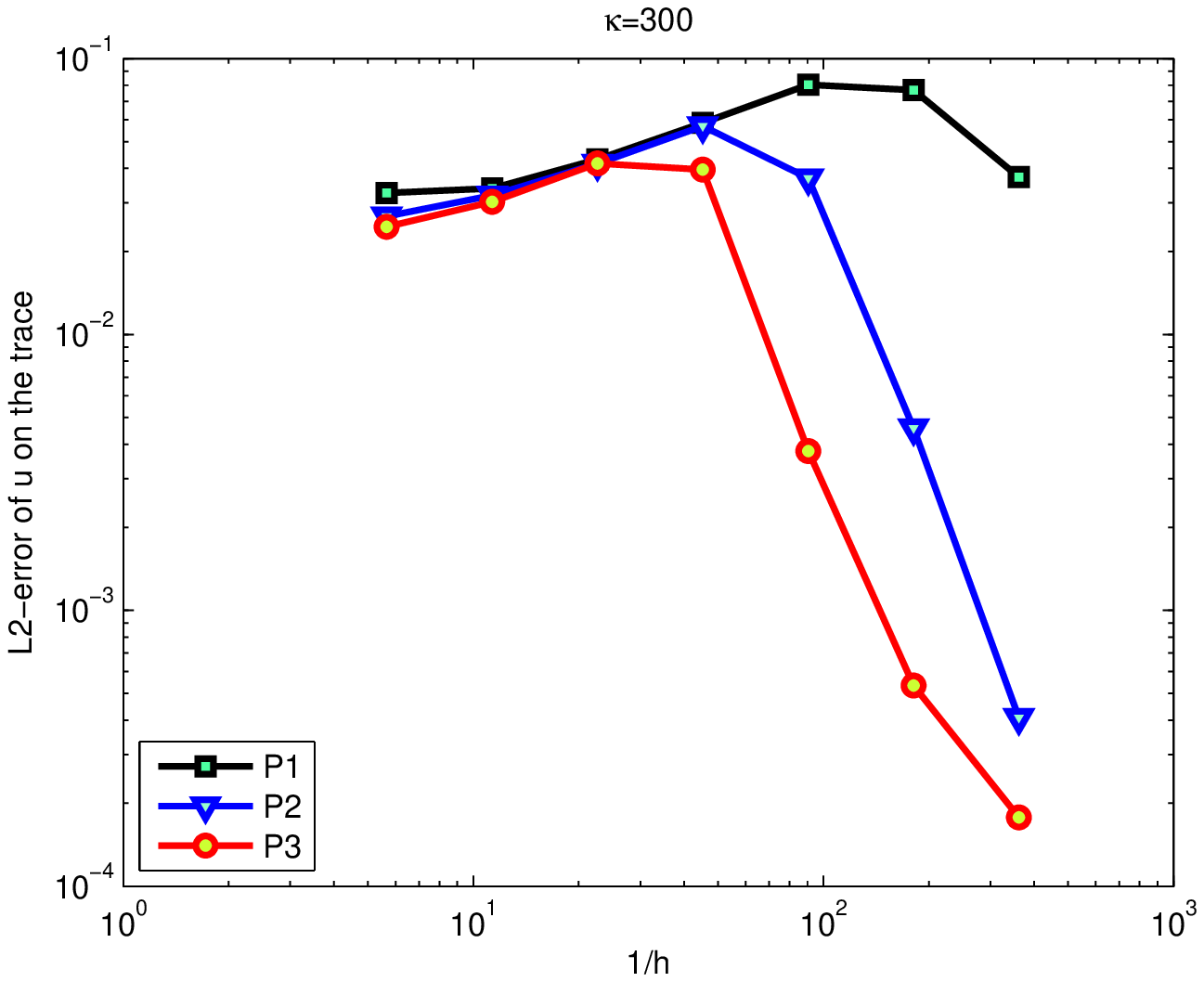}
    \includegraphics[width=2.7in]{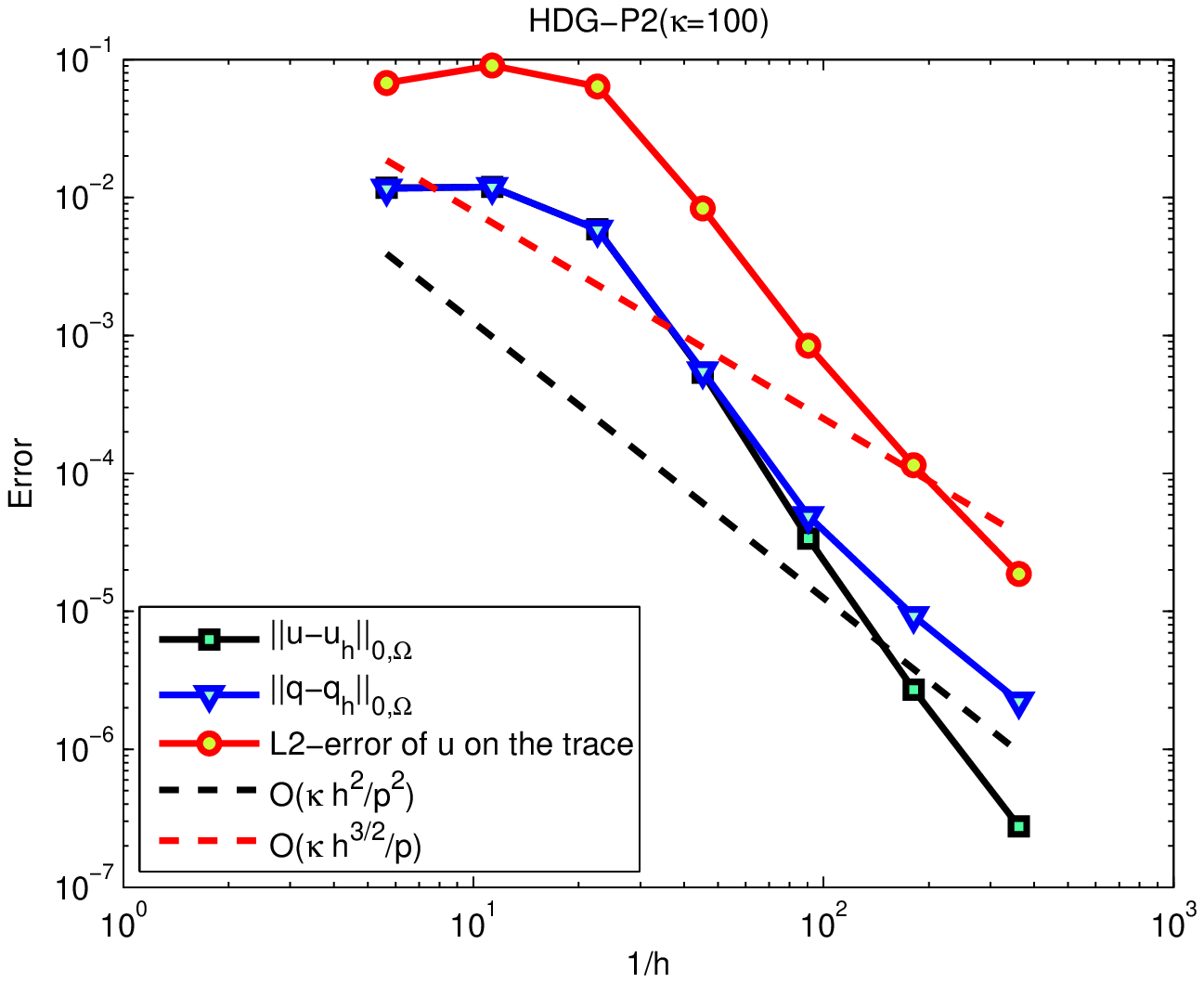}
    \includegraphics[width=2.7in]{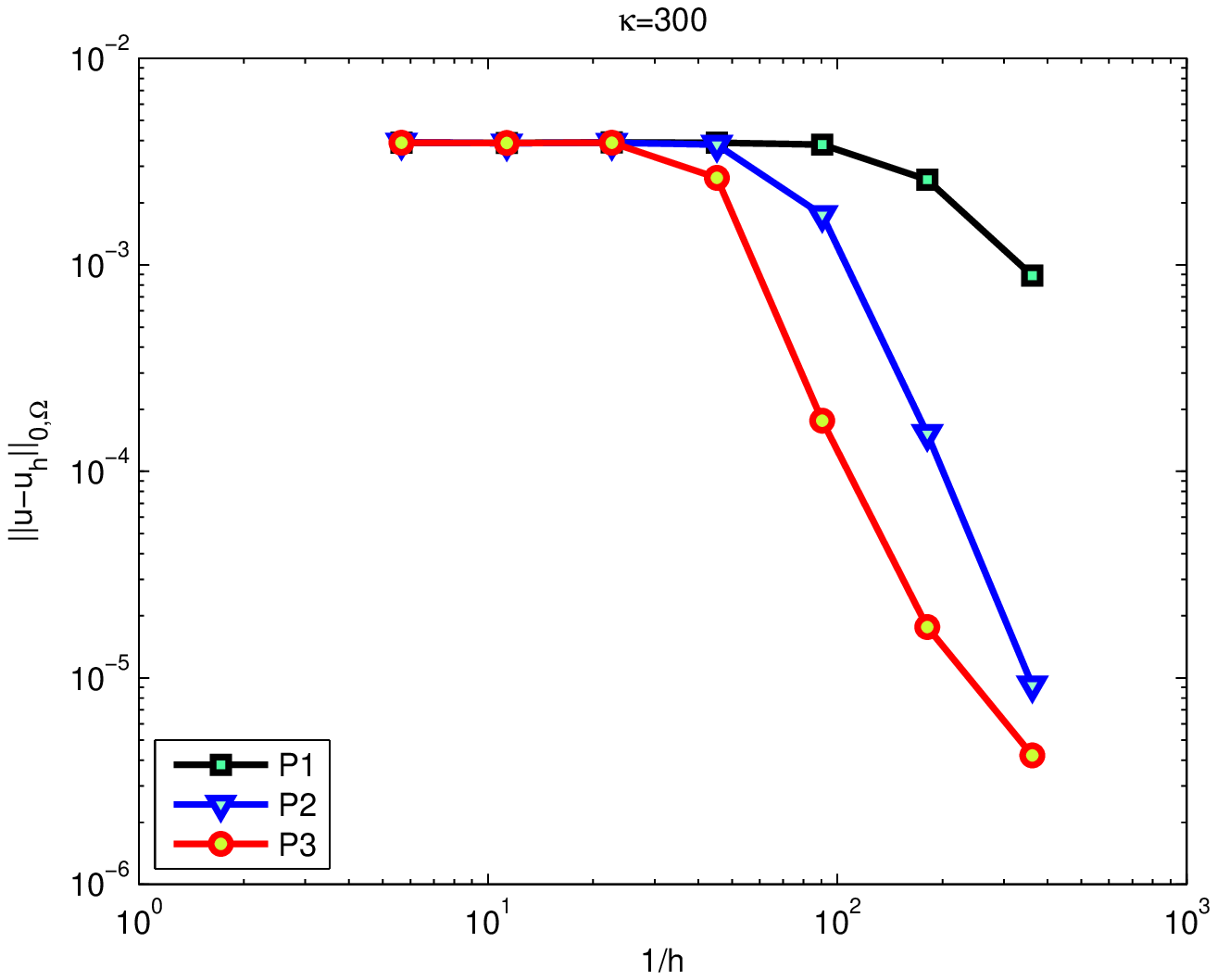}
    \includegraphics[width=2.7in]{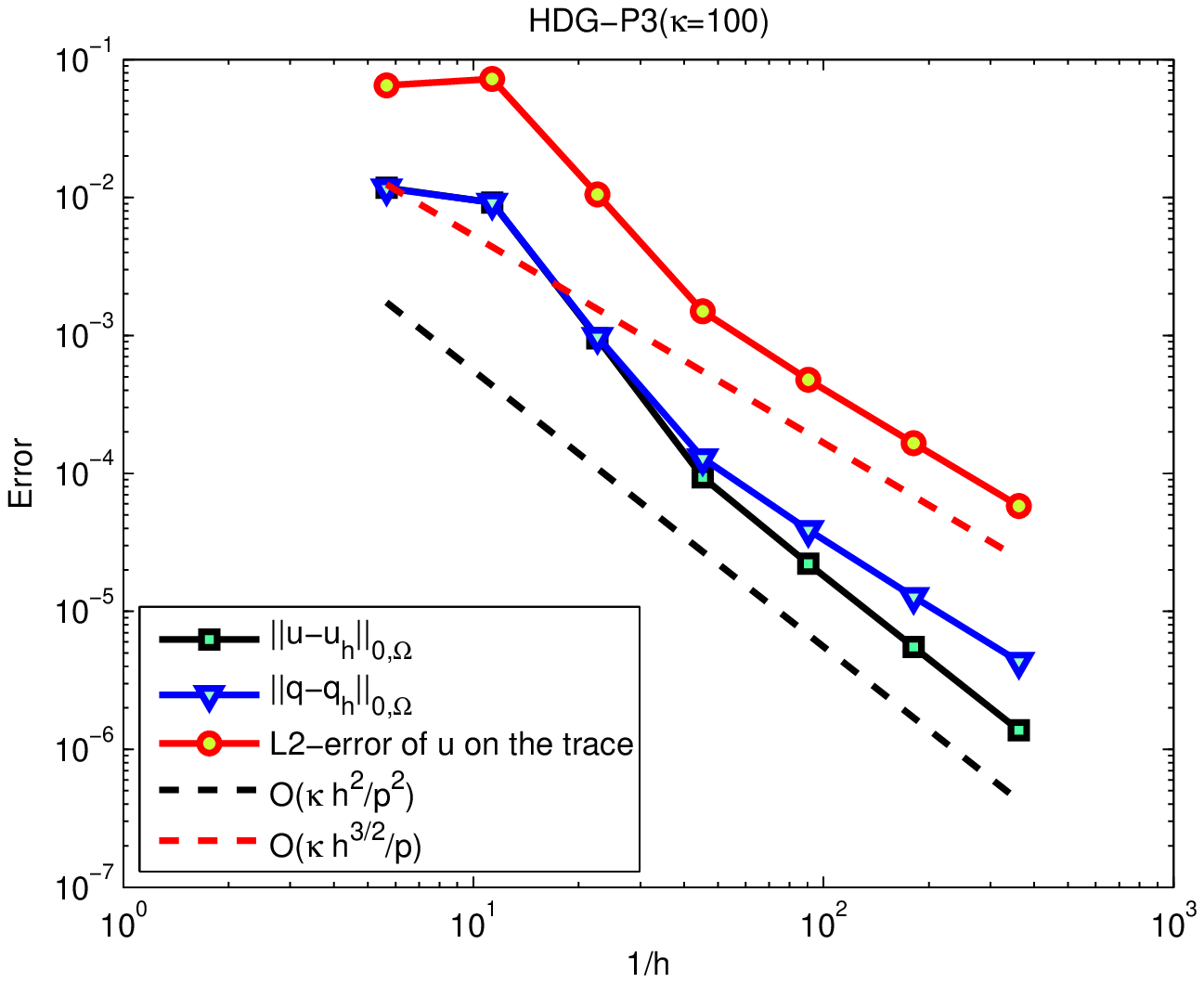}
    \includegraphics[width=2.7in]{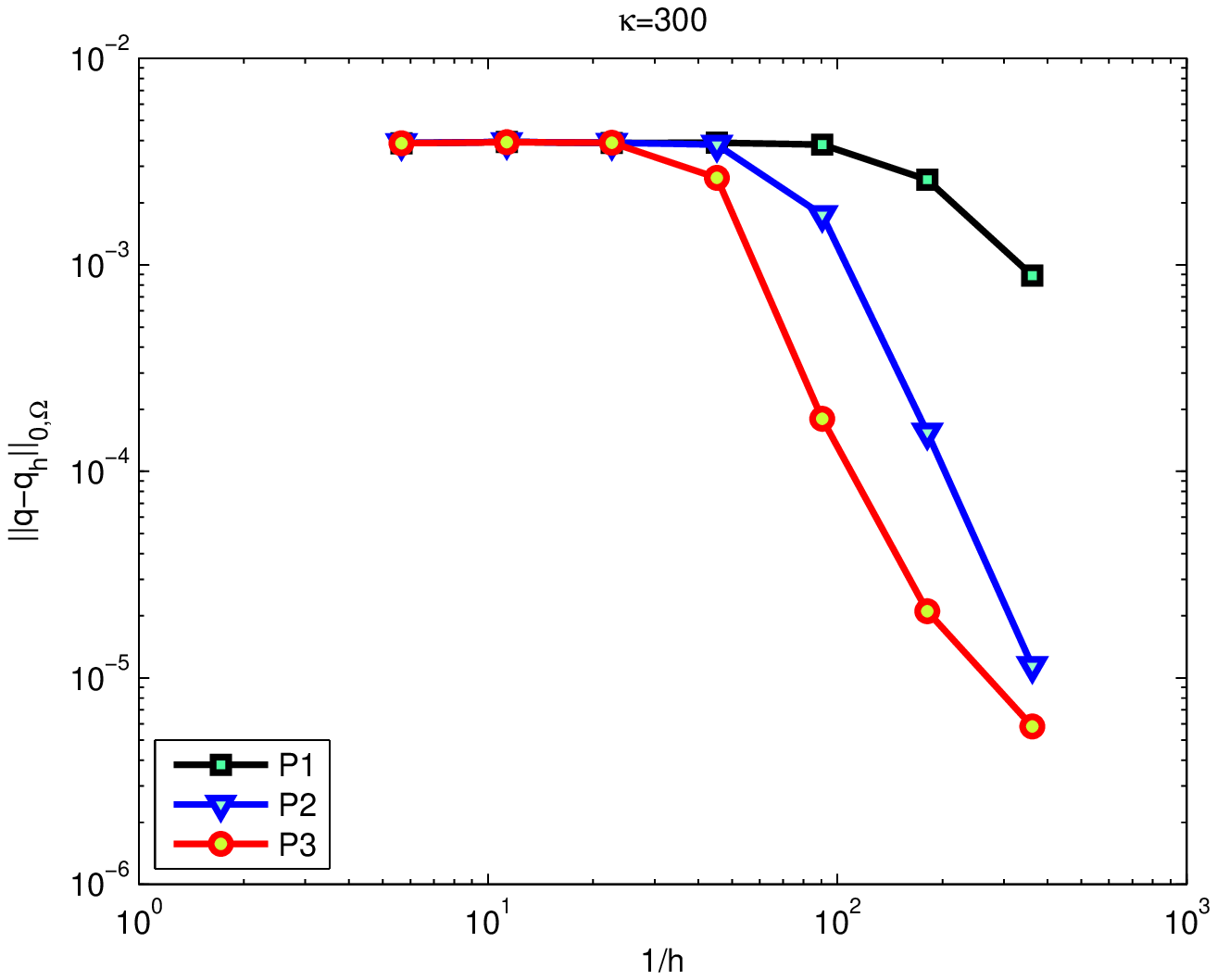}
    \caption{\small Errors of $\|u-u_h\|_{0,\Omega}$, $\|\Bq -
\Bq_h\|_{0,\Omega}$ and $\|u-\hat{u}_h\|_{0,\partial \Ct_h}$ for
$\kappa = 100$ by HDG-P1, HDG-P2 and HDG-P3 approximations (left,
top bottom); Errors of $\|u-\hat{u}_h\|_{0,\partial \Ct_h}$,
$\|u-u_h\|_{0,\Omega}$ and $\|\Bq - \Bq_h\|_{0,\Omega}$ for $\kappa
= 300$ by different polynomial approximations (right, top
bottom).}\label{fig1}
\end{figure}

\begin{figure}[htbp]
\centering
    \includegraphics[width=2.7in]{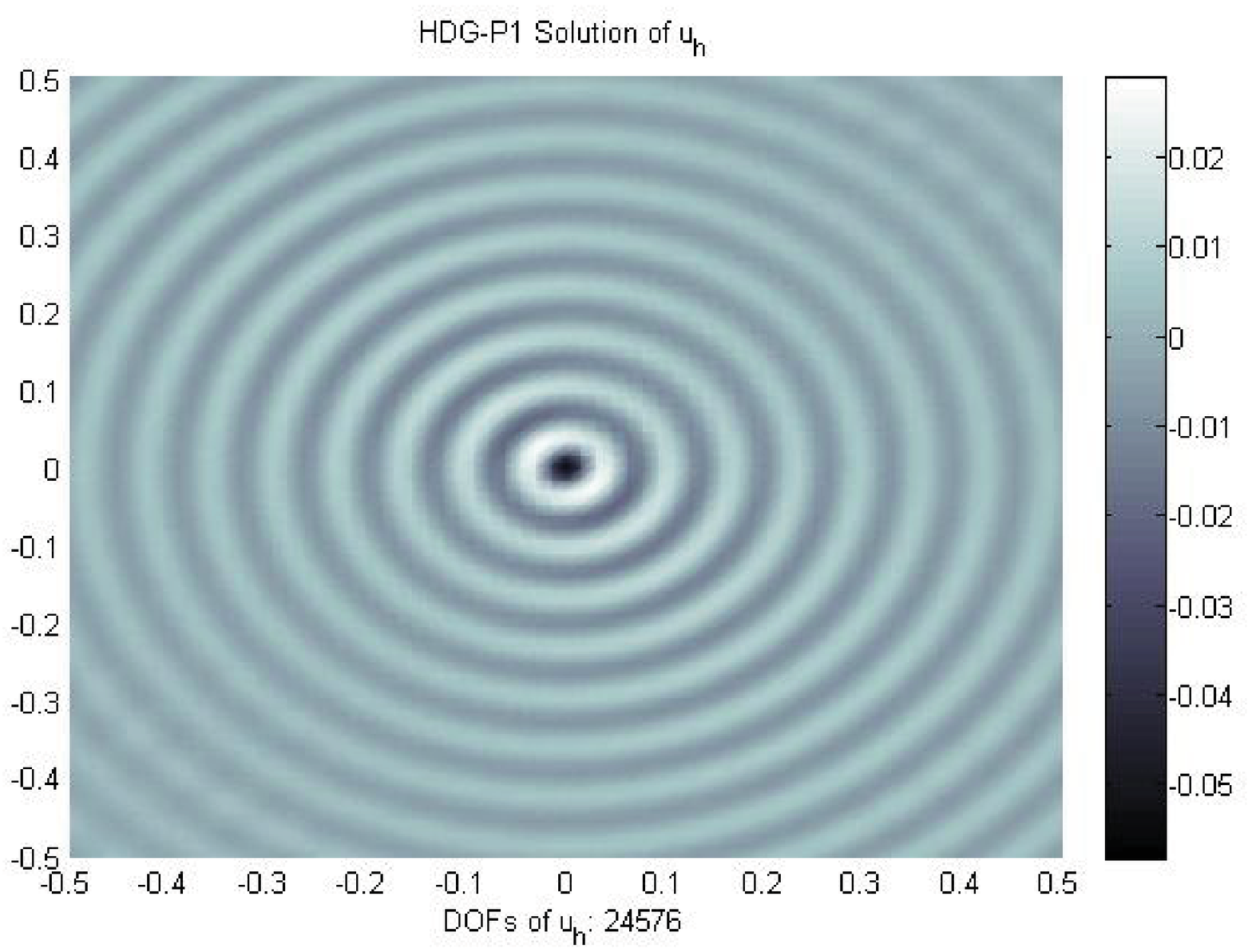}
    \includegraphics[width=2.7in]{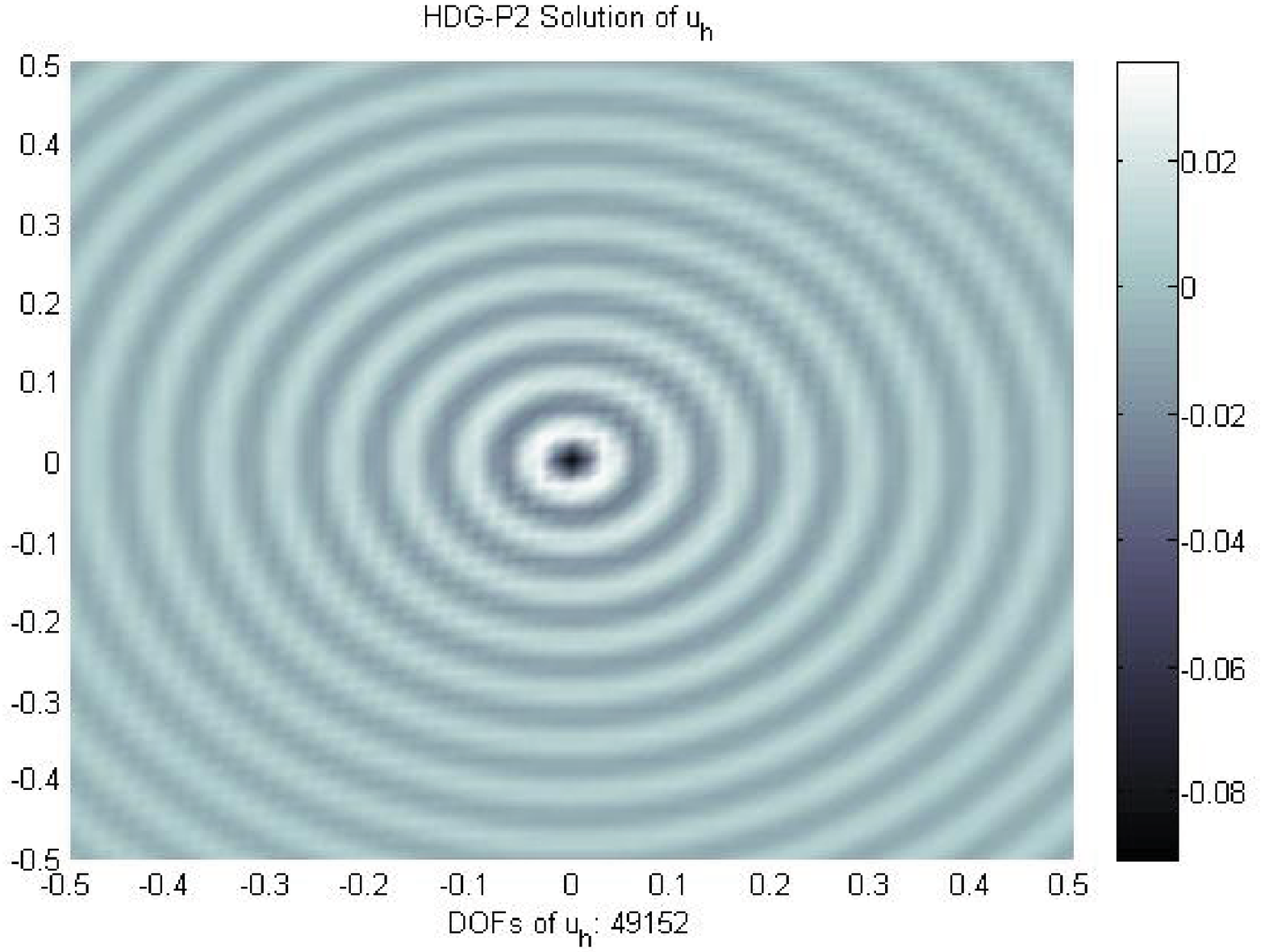}
    \includegraphics[width=2.7in]{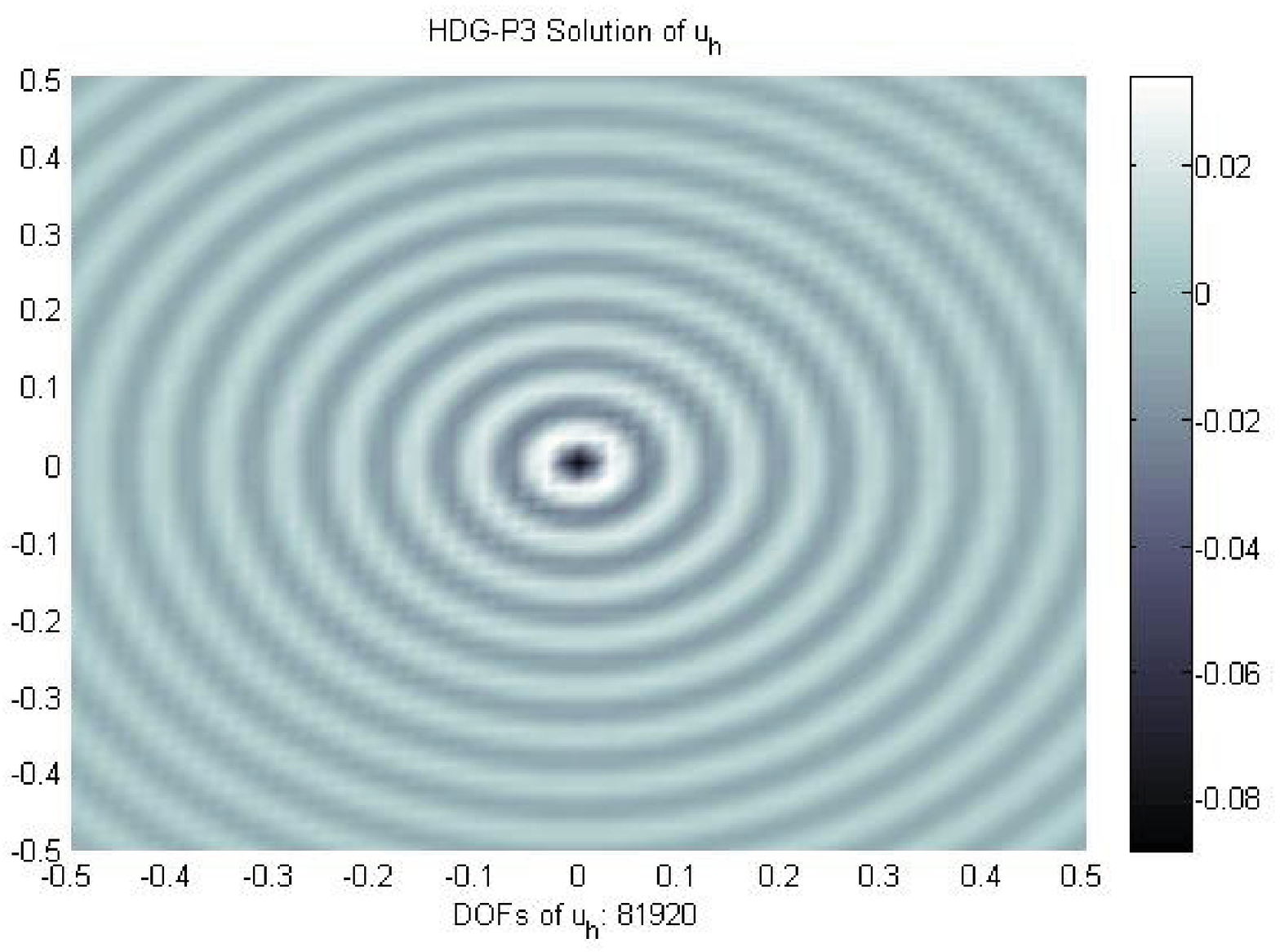}
    \includegraphics[width=2.7in]{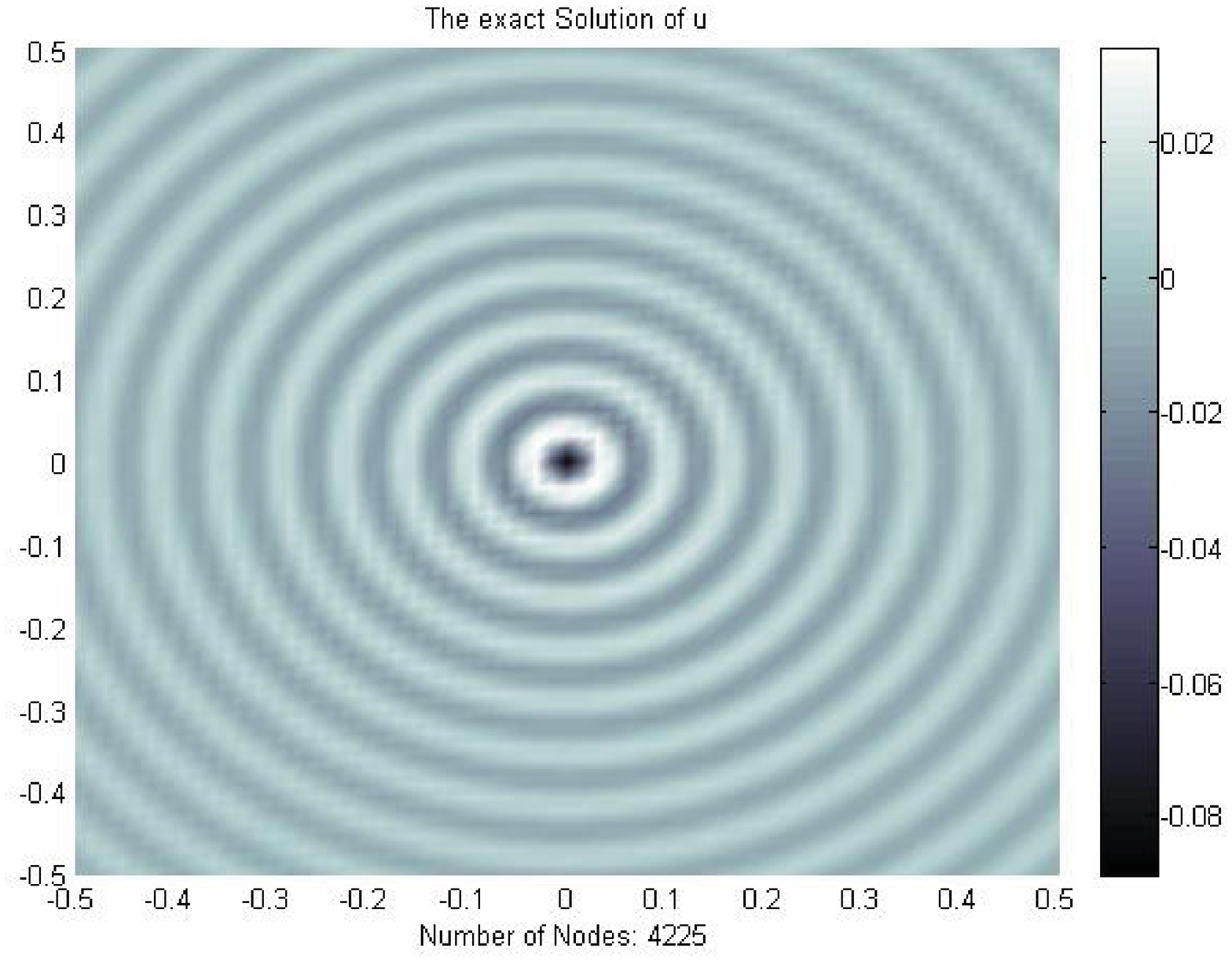}
    \caption{\small Surface plot of the imaginary parts of the HDG-P1, HDG-P2, HDG-P3 solutions of $u_h$ and the exact solution for $\kappa = 100$ with mesh size
    $h \approx 0.022$.}\label{fig2}
\end{figure}


\begin{figure}[htbp]
\centering
    \includegraphics[width=2.7in]{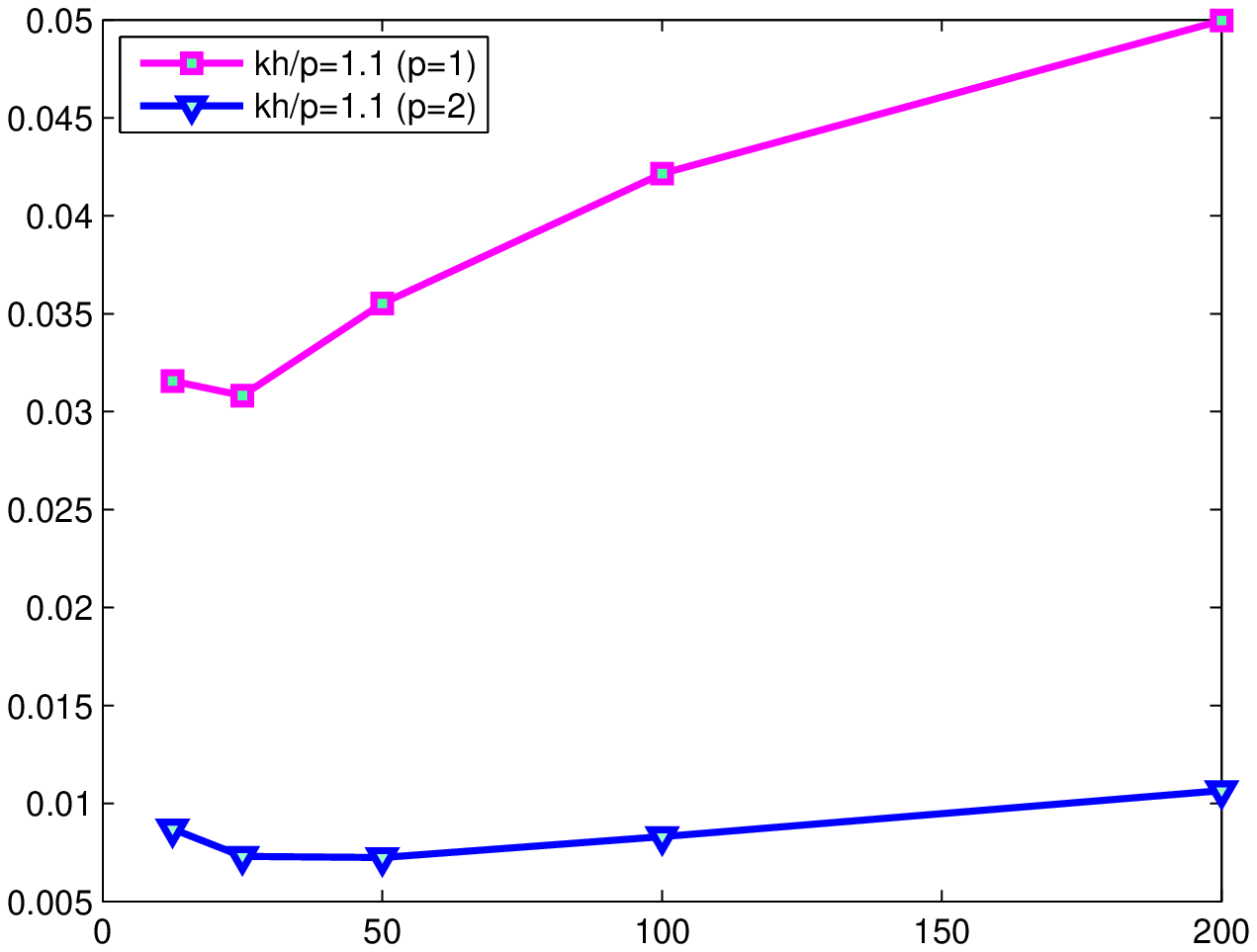}
    \includegraphics[width=2.7in]{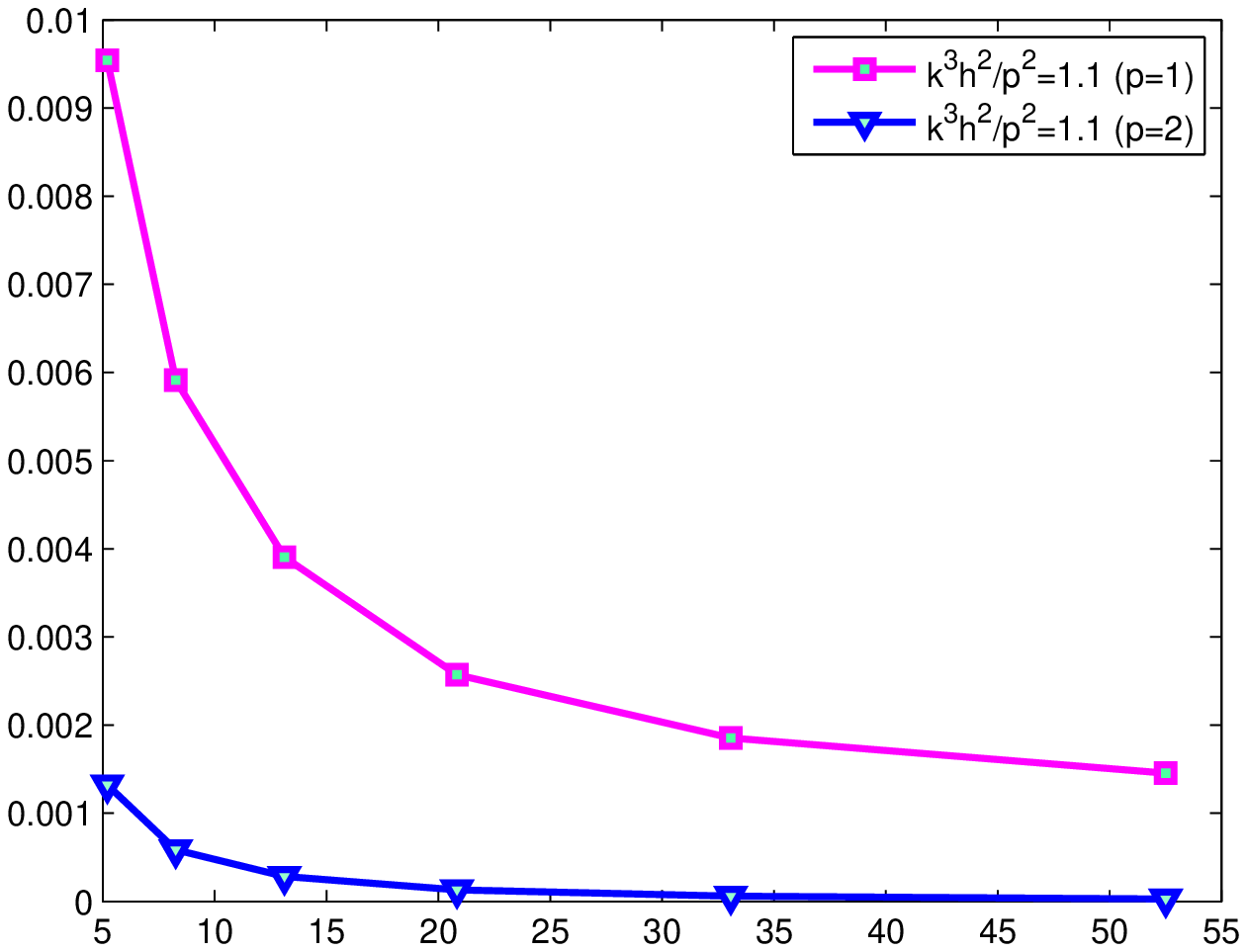}
    \caption{\small Errors of $\|u-\hat{u}_h\|_{0,\partial \Ct_h}$. Left: $\frac{\kappa h}{p} =1.1$. Right: $\frac{\kappa^3 h^2}{p^2} =1.1$.}\label{fig-kh}
\end{figure}

\begin{figure}[htbp]
\centering
    \includegraphics[width=2.7in]{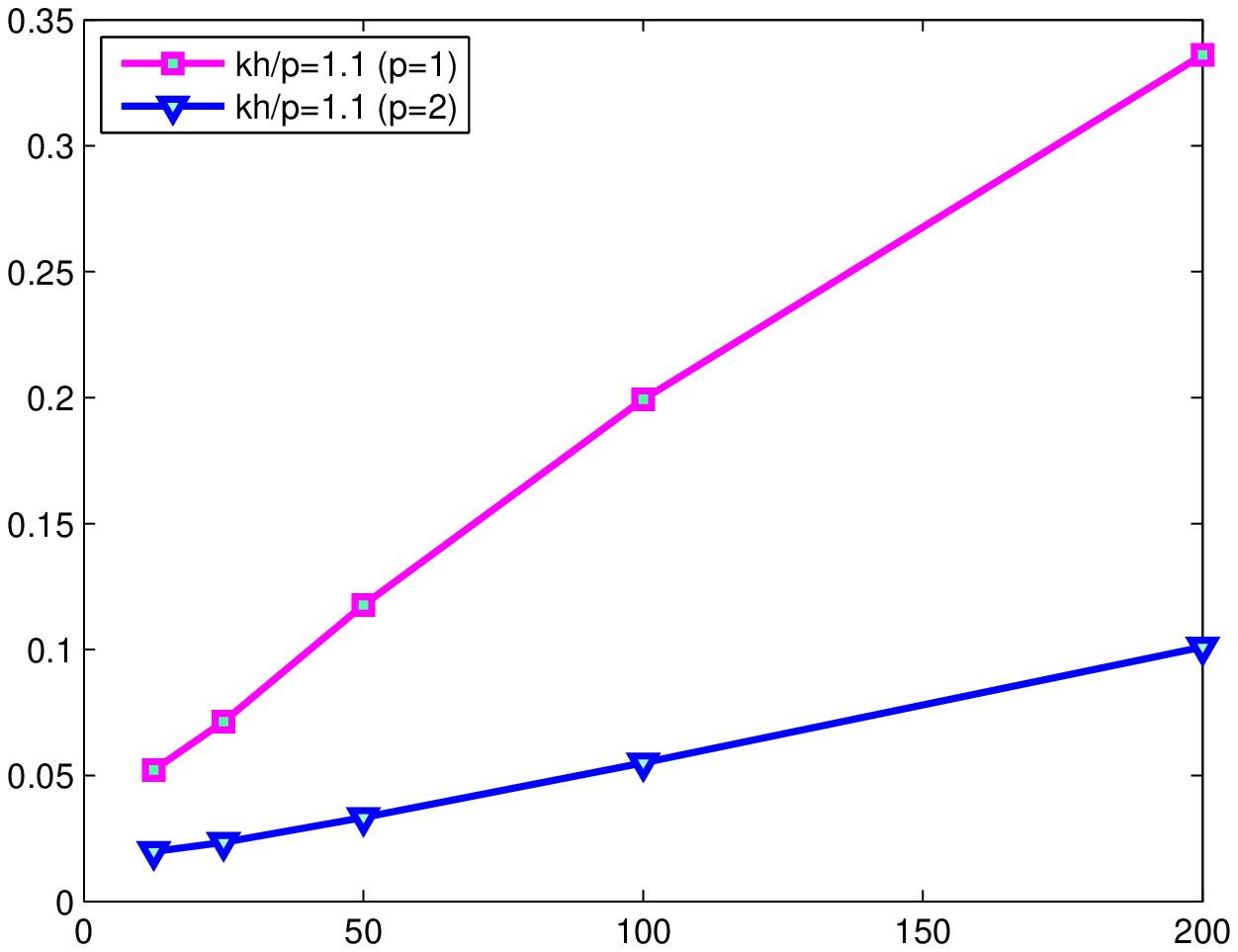}
    \includegraphics[width=2.7in]{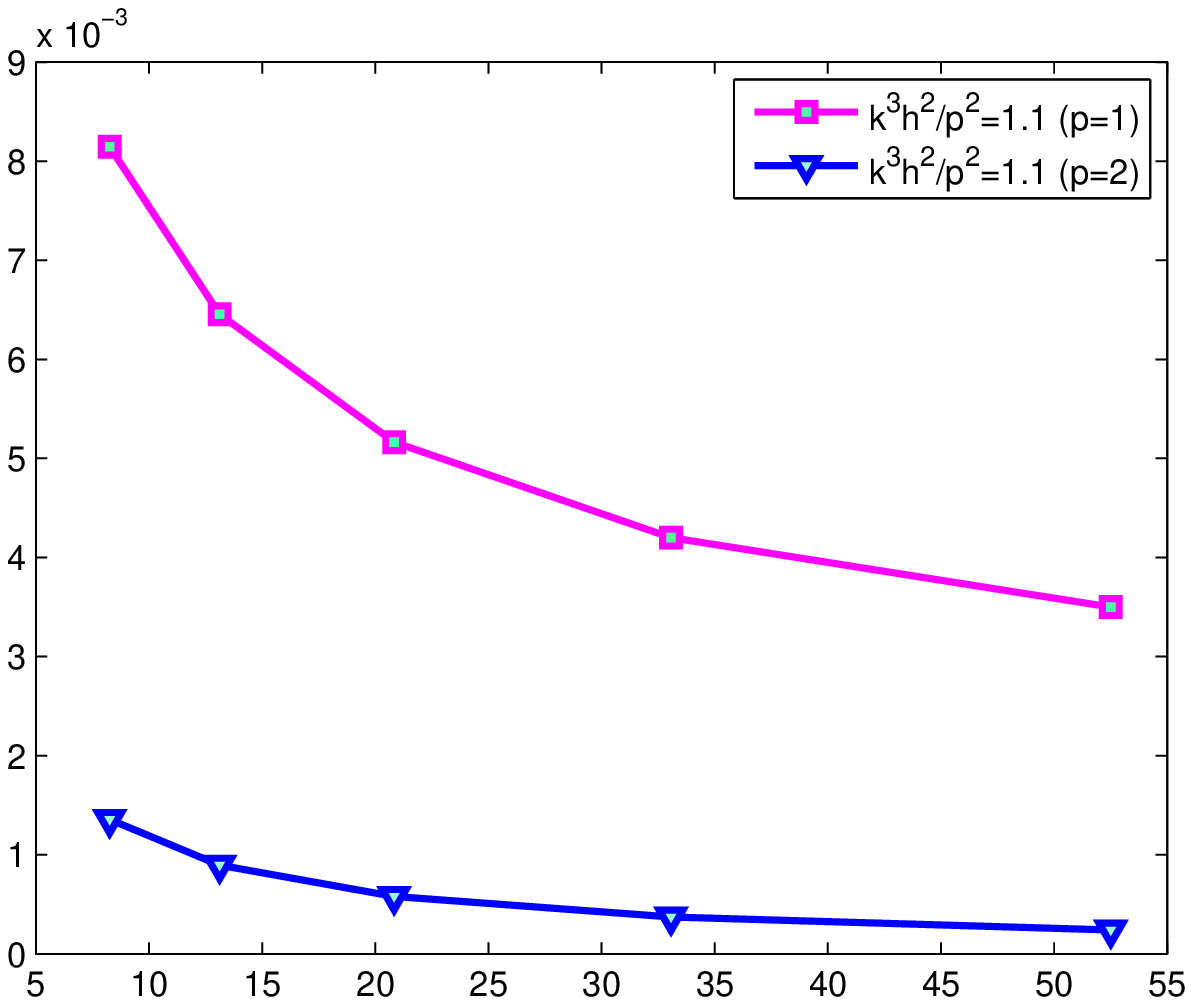}
    \includegraphics[width=2.7in]{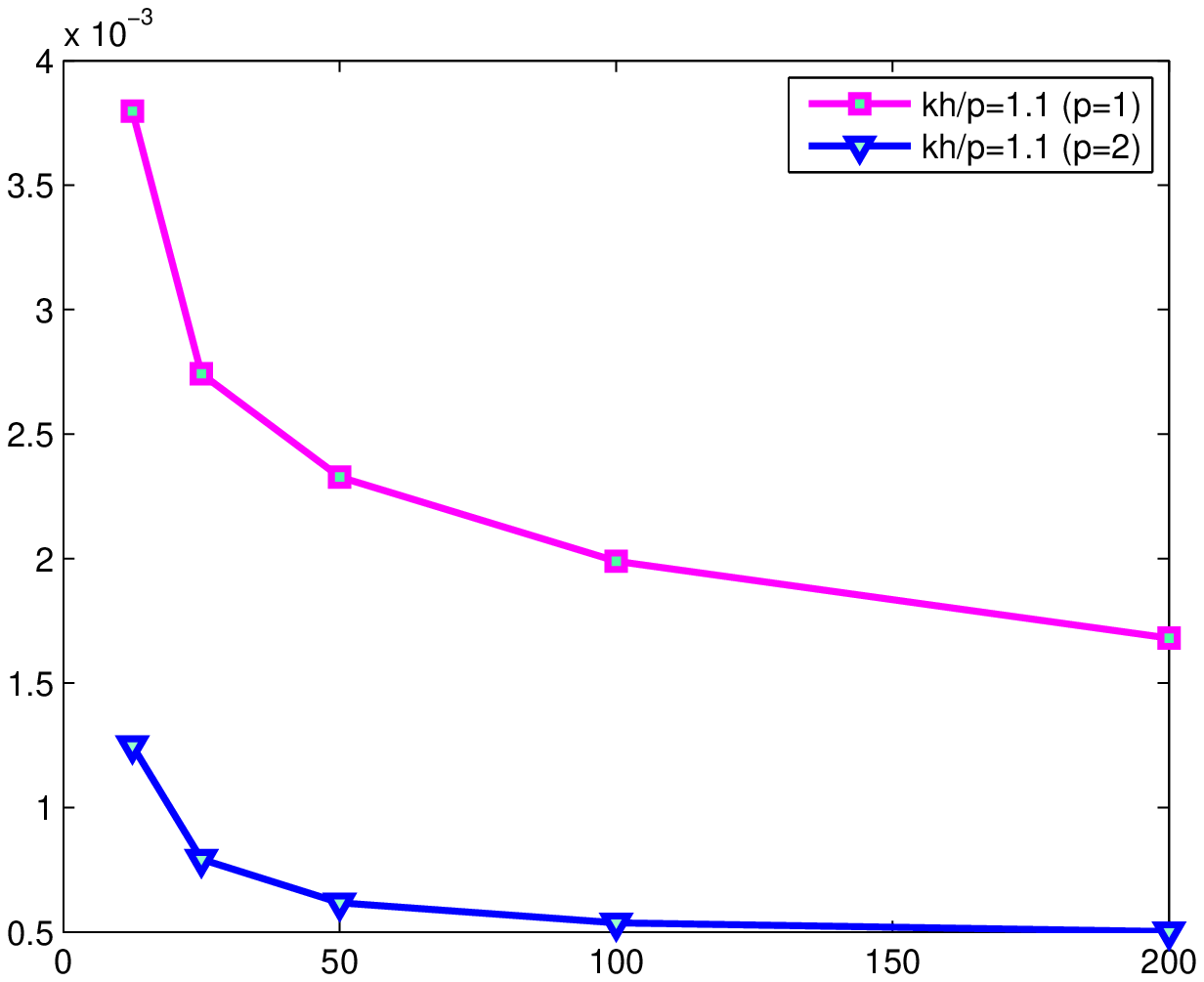}
    \caption{\small $\kappa \|\Bq-\Bq_h\|_{0,\Omega}$: $\frac{\kappa h}{p} =1.1$ (top left), $\frac{\kappa^3 h^2}{p^2} =1.1$ (top right).
    $\|u-u_h\|_{0,\Omega}$: $\frac{\kappa h}{p} =1.1$ (bottom).}\label{fig-kh-q}
\end{figure}

\begin{figure}[htbp]
\centering
    \includegraphics[width=2.7in]{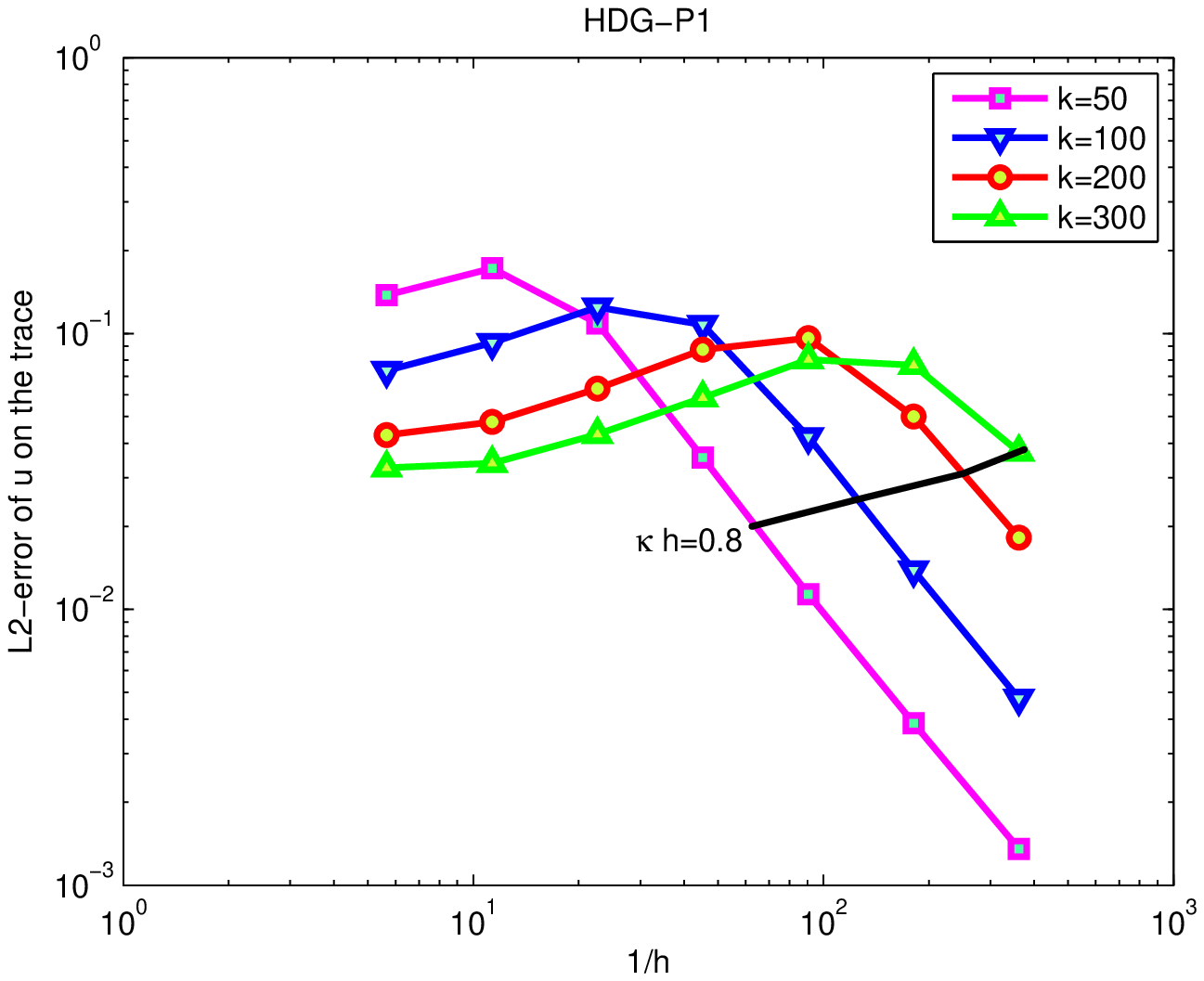}
    \includegraphics[width=2.7in]{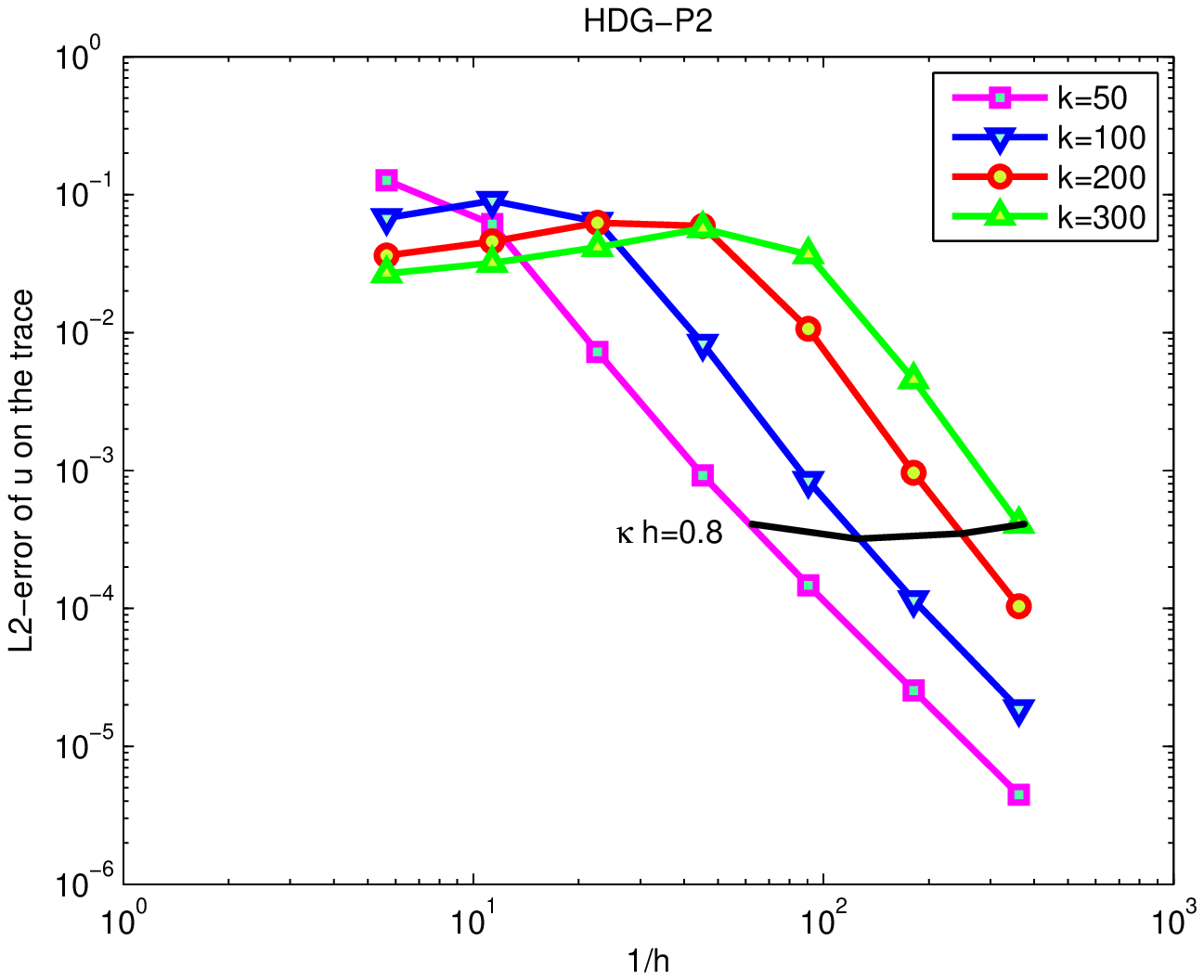}
    \includegraphics[width=2.7in]{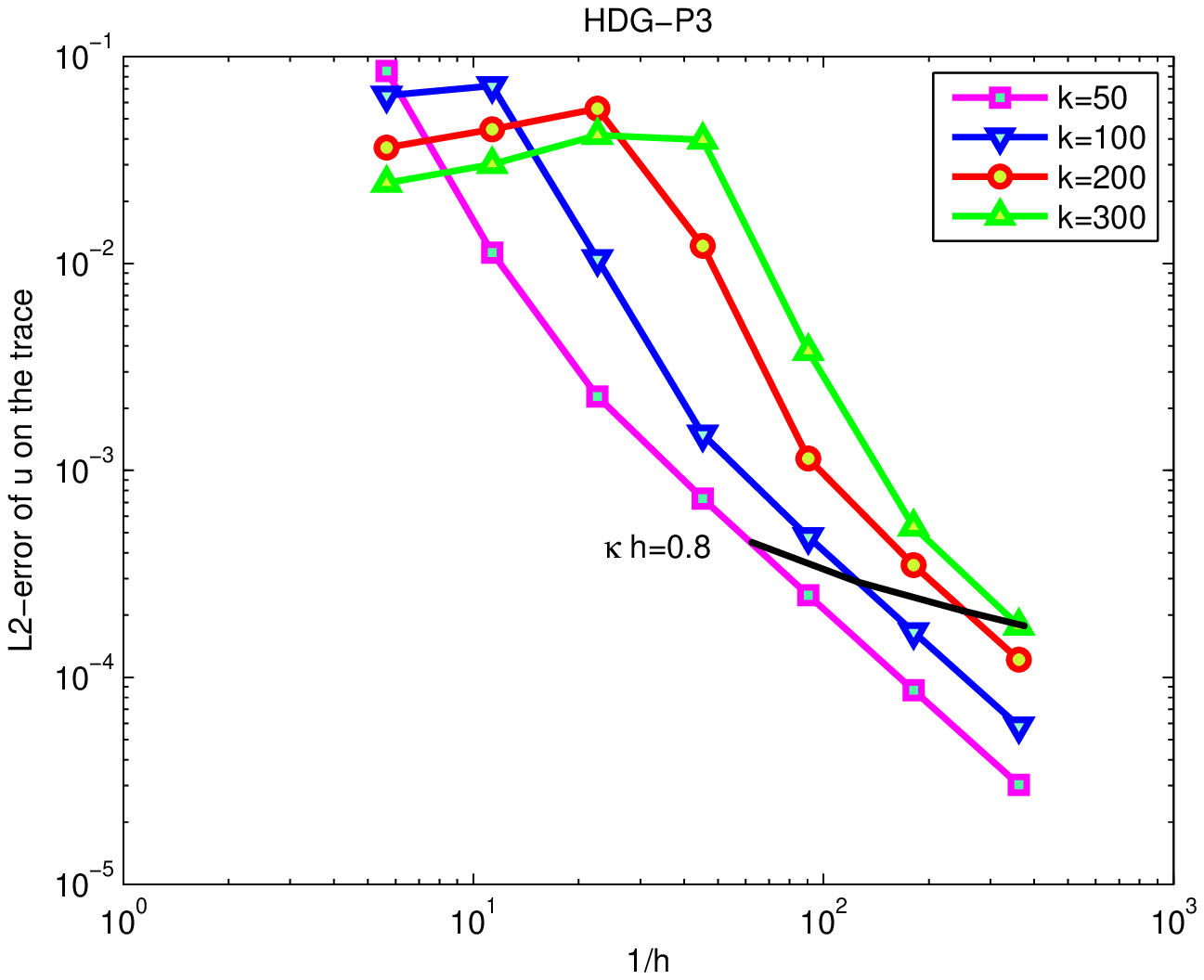}
    \caption{\small The error of $\|u-\hat{u}_h\|_{0,\partial \Ct_h}$ for
$\kappa = 50, 100, 200, 300$ by HDG-P1, HDG-P2 and
HDG-P3.}\label{fig3}
\end{figure}

The left graph of Figure \ref{fig-kh} displays the error
$\|u-\hat{u}_h\|_{0,\partial \Ct_h}$ for fixed $\frac{\kappa h}{p}
=1.1$. It shows that the error $\|u-\hat{u}_h\|_{0,\partial \Ct_h}$
can not be controlled by $\frac{\kappa h}{p}$ and increases with
$\kappa$, which indicates that there is pollution error in the total
error. The right graph of Figure \ref{fig-kh} displays the same
error with the mesh size satisfying $\frac{\kappa^3 h^2}{p^2} =1.1$
for different $\kappa, h$ and $p$. We observe that under this mesh
condition, the error $\|u-\hat{u}_h\|_{0,\partial \Ct_h}$ does not
increase with $\kappa$. The top two graphs of Figure \ref{fig-kh-q}
shows the similar property for the error $\kappa
\|\Bq-\Bq_h\|_{0,\Omega}$. From the bottom graph of Figure
\ref{fig-kh-q} we can also find that the error
$\|u-u_h\|_{0,\Omega}$ does not increase with $\kappa$ only under
the mesh condition $\frac{\kappa h}{p} =1.1$.

Next we verify the convergence properties of the HDG method for
different wave numbers by piecewise P1, P2 and P3 approximations
respectively. In Figure \ref{fig3}, the error
$\|u-\hat{u}_h\|_{0,\partial \Ct_h}$ of HDG-P1, HDG-P2 and HDG-P3
solutions for different wave numbers always oscillates on the coarse
meshes, and then decays on fine meshes. For HDG-P1 solution,
$\|u-\hat{u}_h\|_{0,\partial \Ct_h}$ grows with $\kappa$ along the
line $\kappa h = 0.8$ for $\kappa \leq 300$. By contrast, for
$\kappa \leq 300$, this error does not increase significantly along
the line $\kappa h = 0.8$ by HDG-P2, and decreases along the line
$\kappa h = 0.8$ by HDG-P3. This also means that the pollution error
can be reduced by high order polynomial approximations. Figure
\ref{fig4} shows the convergence property of $\|u-u_h\|_{0,\Omega}$
for different wave numbers. For $\kappa \leq 300$, along the line
$\kappa h =0.8$,  the error $\|u-u_h\|_{0,\Omega}$ of HDG-P1 and
HDG-P2 solutions stays stable, and for HDG-P3 solution, this error
decreases as $\|u-\hat{u}_h\|_{0,\partial \Ct_h}$. Similar
phenomenon can also be observed for the error $\|\Bq -
\Bq_h\|_{0,\Omega}$ by different polynomial approximations.

\begin{figure}[htbp]
\centering
    \includegraphics[width=2.7in]{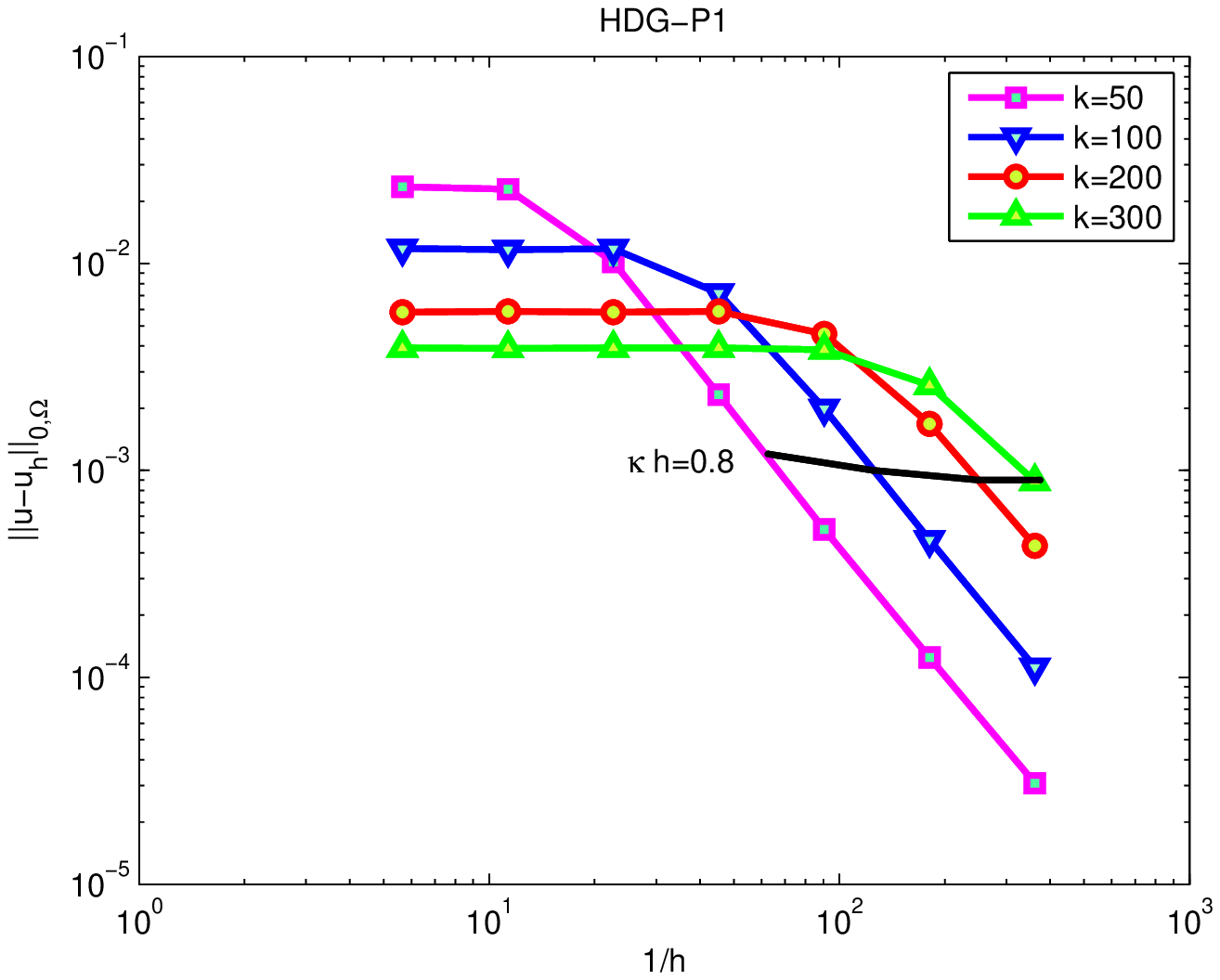}
    \includegraphics[width=2.7in]{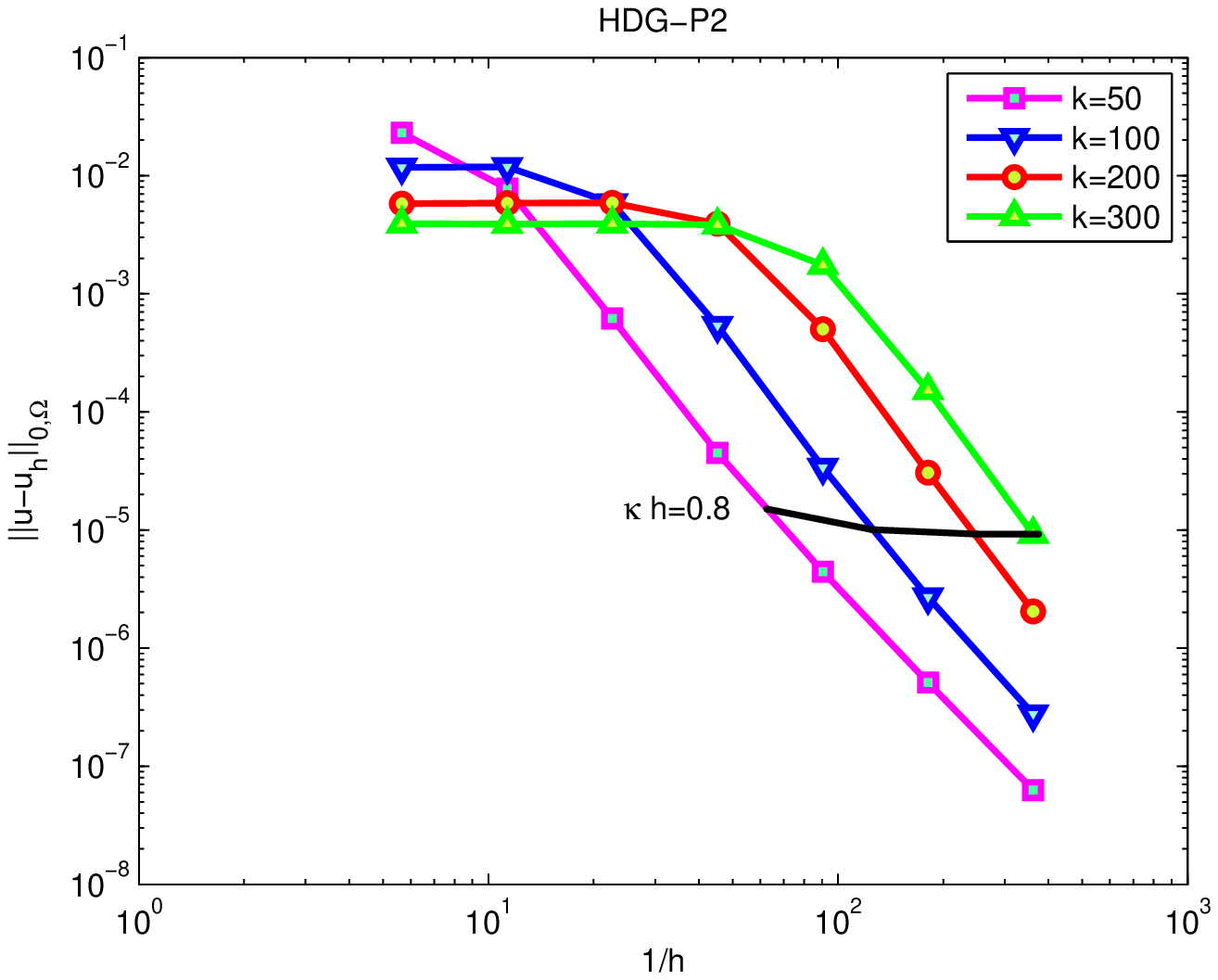}
    \includegraphics[width=2.7in]{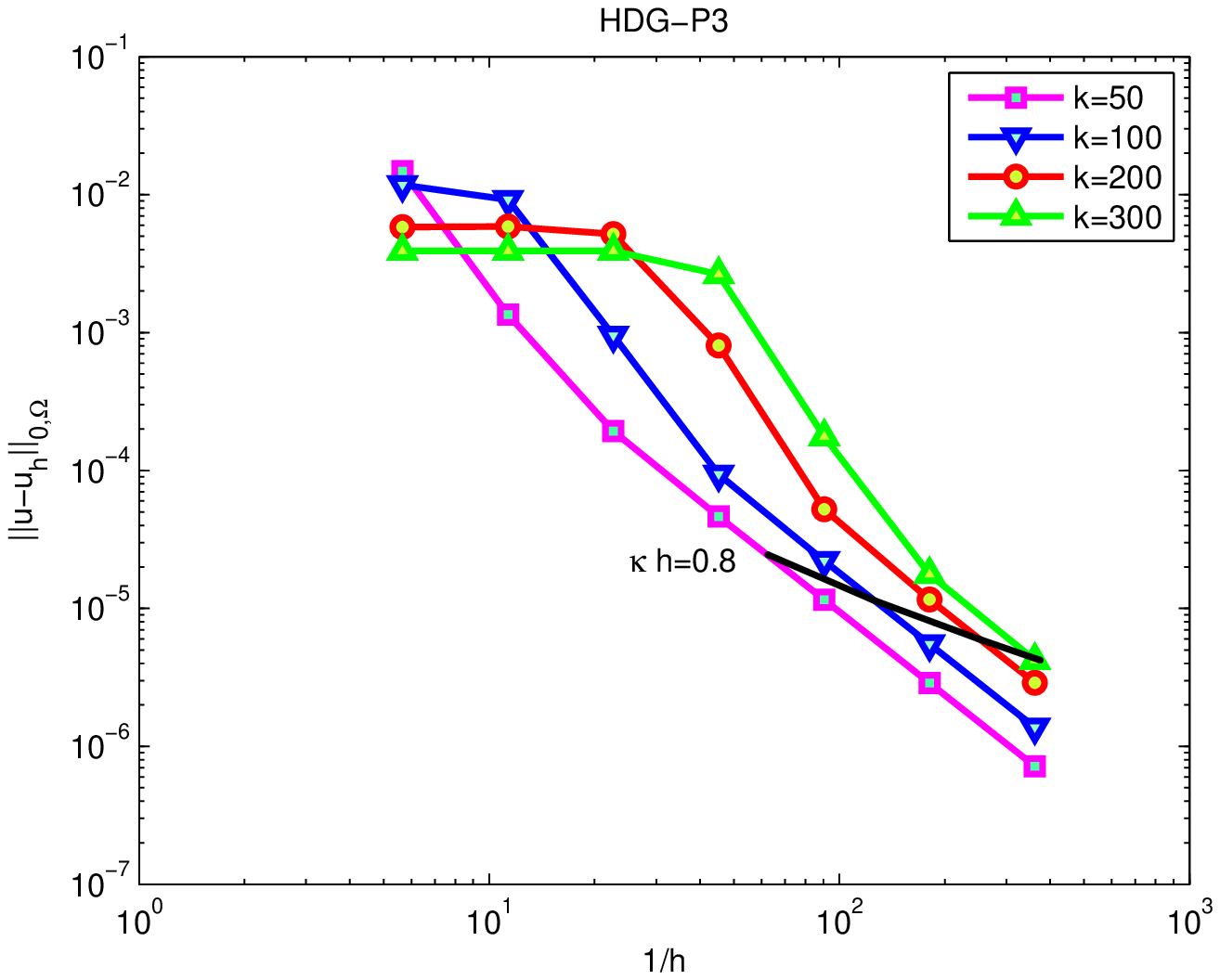}
    \caption{\small The error of $\|u-u_h\|_{0,\Omega}$ for
$\kappa = 50, 100, 200, 300$ by HDG-P1, HDG-P2 and
HDG-P3.}\label{fig4}
\end{figure}

\smallskip

For more detailed comparison between HDG methods with different
polynomial order approximations. We consider the problem with wave
number $\kappa = 200$. The traces of imaginary part of the HDG
solution $\hat{u}_h$ with piecewise P1, P2 and P3 approximations in
the $xz$-plane with mesh sizes $h \approx 0.022, 0.0055$, and the
trace of imaginary part of the exact solution, are both plotted in
Figure \ref{fig5}. On the coarse mesh with $h\approx  0.022$, the
shapes of HDG-P2 and HDG-P3 solutions are roughly the same as the
exact solution, while the shape of HDG-P1 solution does not match
the exact solution well. But on the fine mesh with $h\approx
0.0055$, the shapes of HDG solutions match well and even better for
high order polynomial approximations. Then we can observe that
although the phase error appears in the case of coarse meshes and
low order polynomial approximation, it can be reduced in the fine
meshes or by high order polynomial approximations.

\begin{figure}[htbp]
\centering
    \includegraphics[width=2.7in]{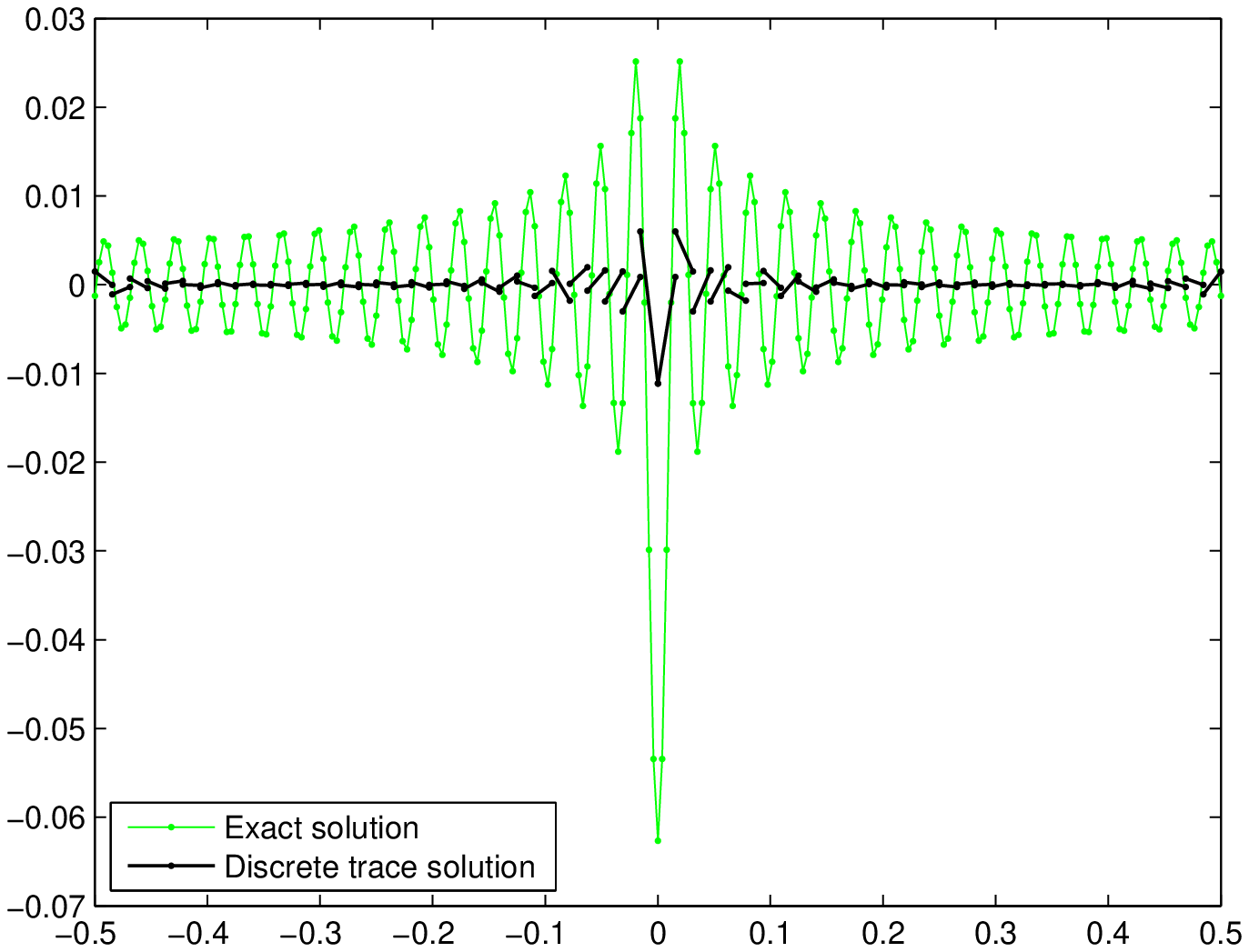}
    \includegraphics[width=2.7in]{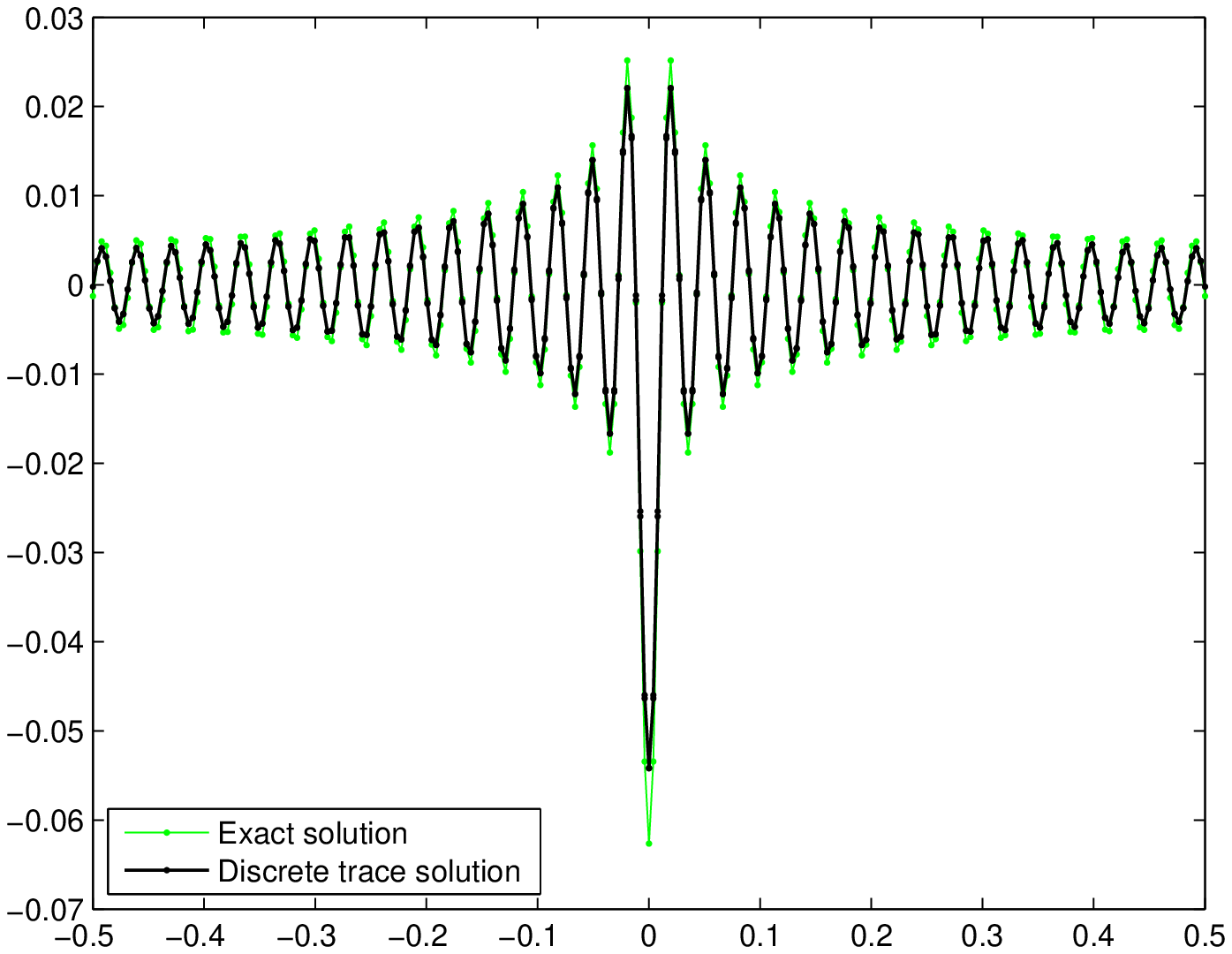}
    \includegraphics[width=2.7in]{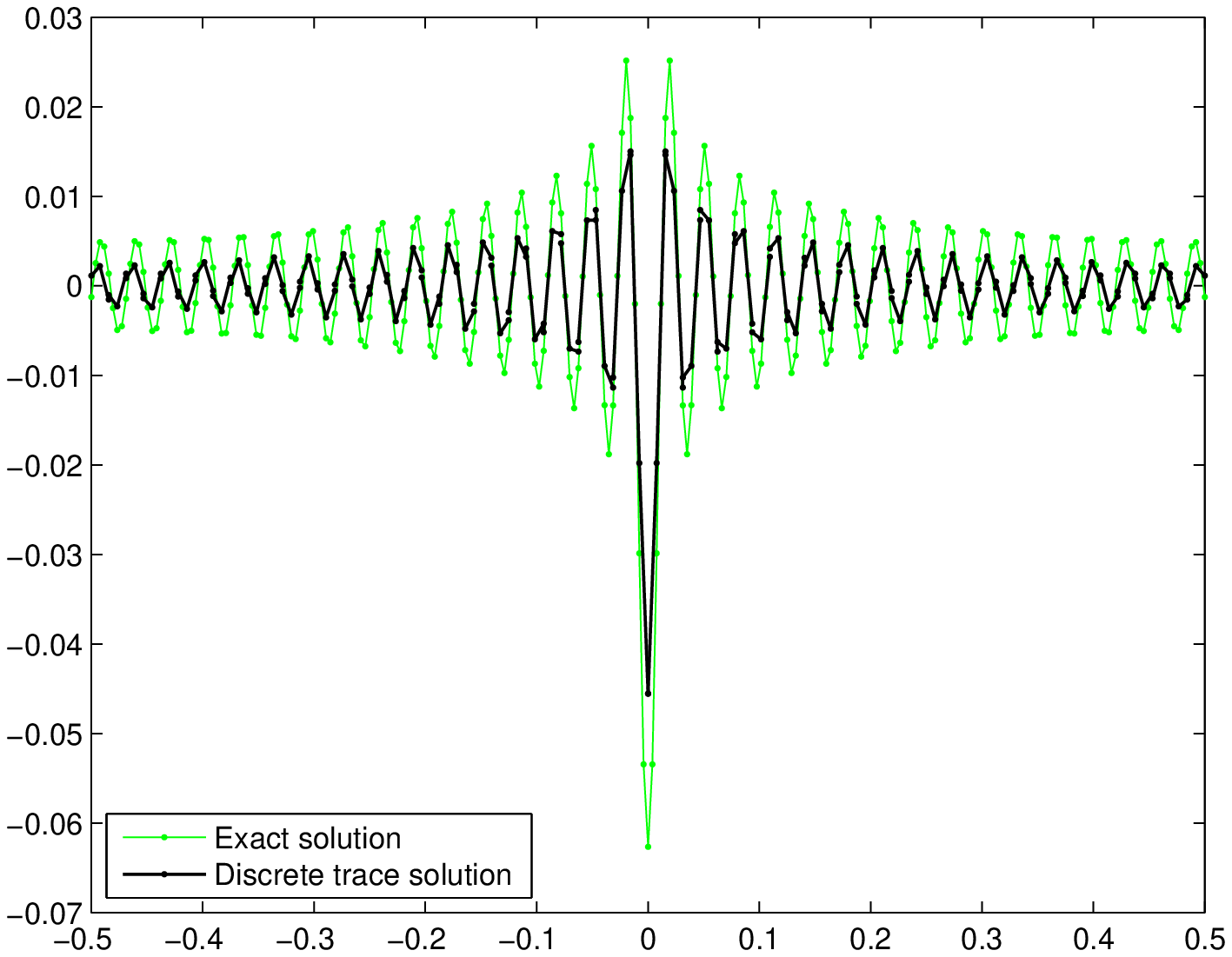}
    \includegraphics[width=2.7in]{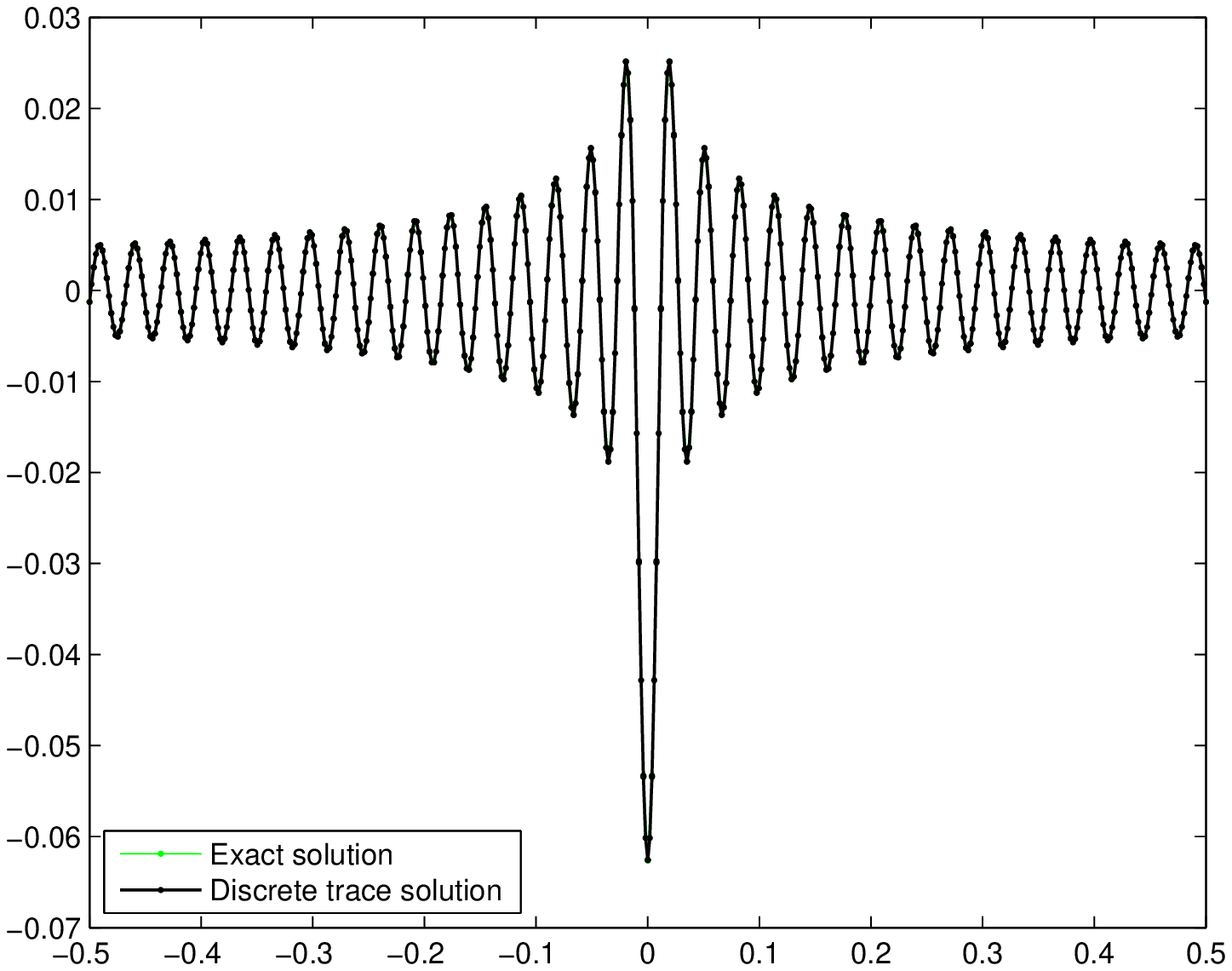}
    \includegraphics[width=2.7in]{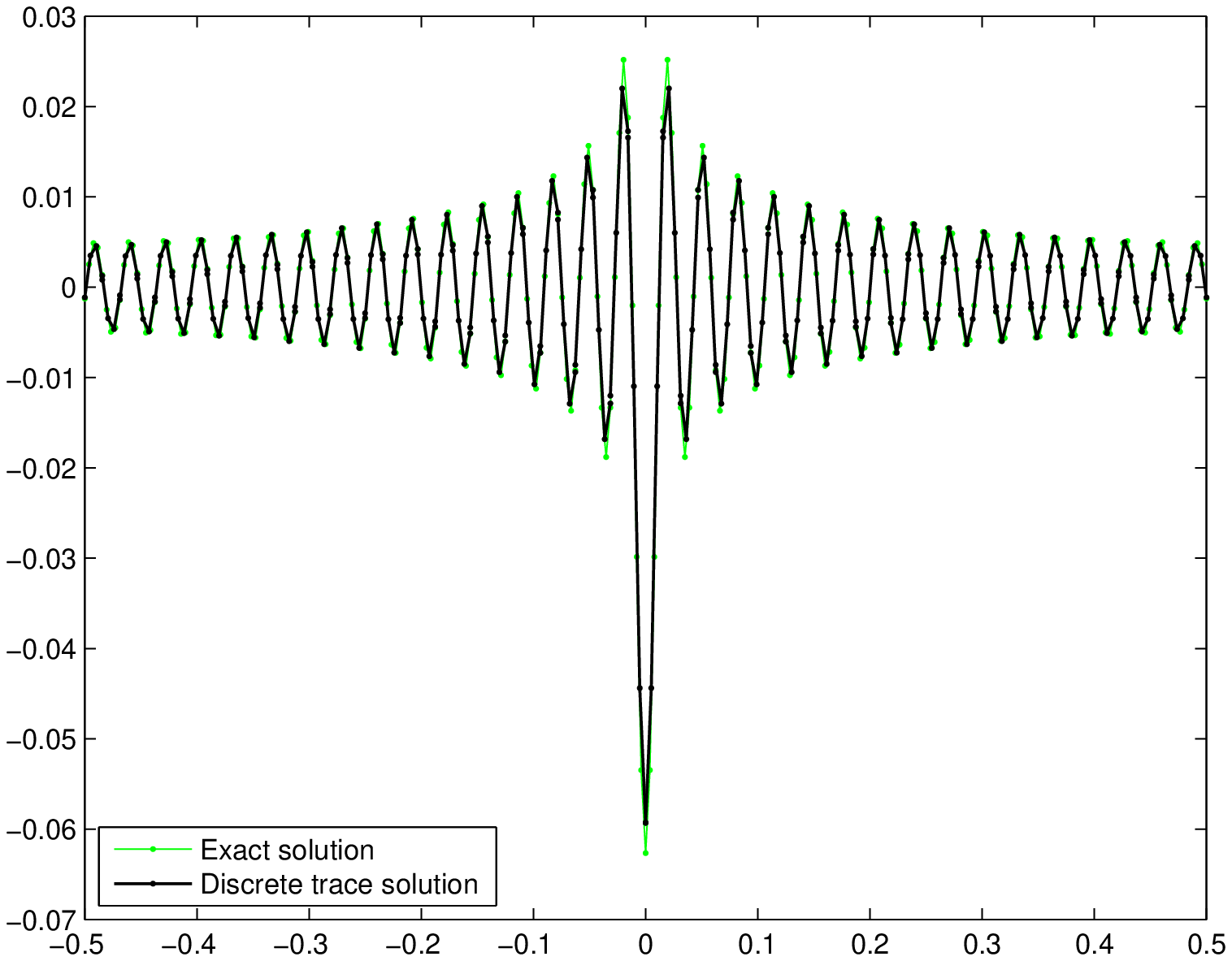}
    \includegraphics[width=2.7in]{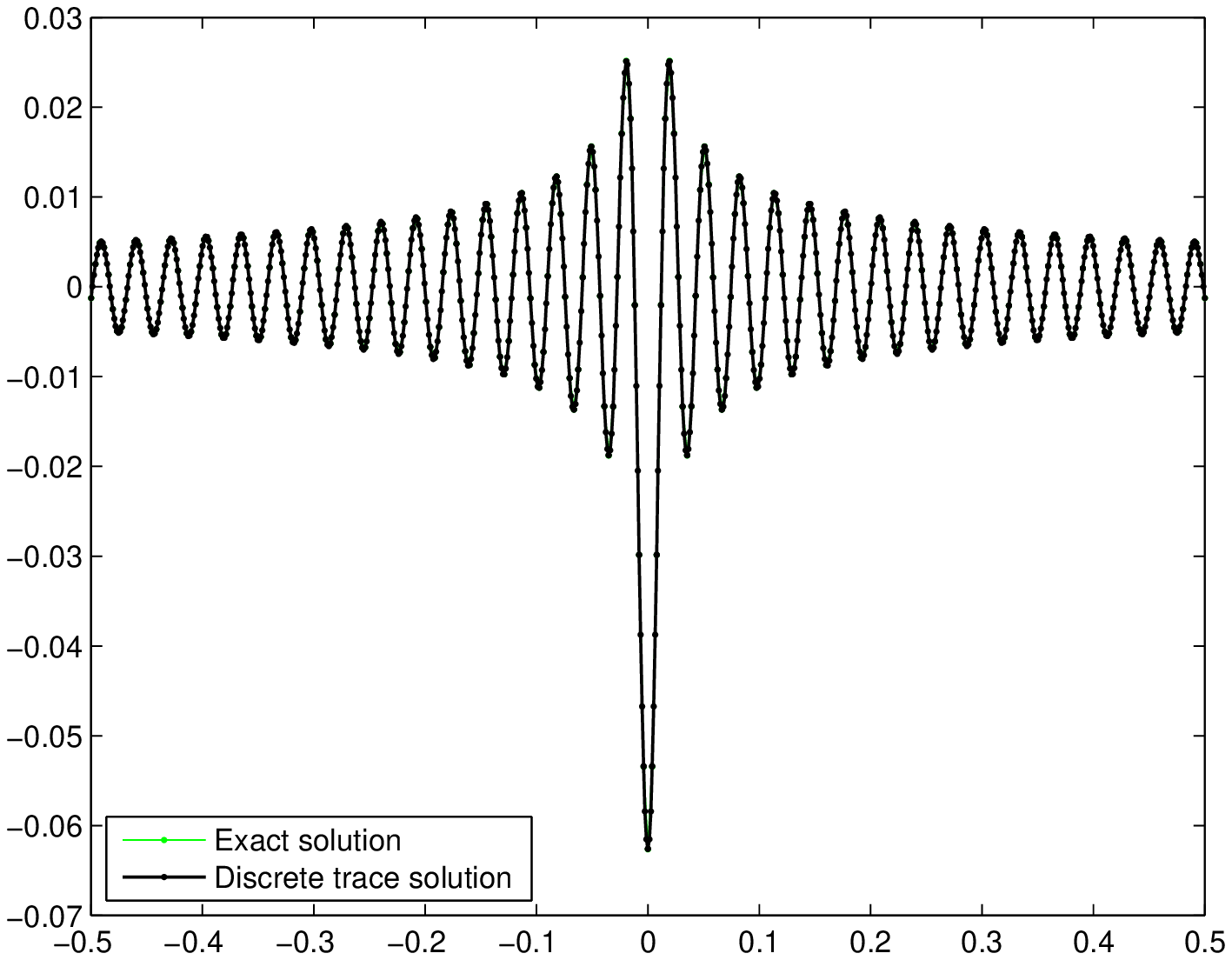}
    \caption{\small The traces of imaginary part of the HDG
solution $\hat{u}_h$ by HDG-P1, HDG-P2 and HDG-P3 (top downbottom) with
mesh sizes $h \approx 0.022, 0.0055 $ (left, right). The trace of
imaginary part of the exact solution is plotted in the green
lines.}\label{fig5}
\end{figure}

\end{document}